\documentclass[onetabnum,onefignum,nohypdvips,final]{siamart171218}
\usepackage{lipsum}
\usepackage{amsfonts}
\usepackage{graphicx}
\usepackage{epstopdf}
\usepackage{algorithmic}
\usepackage{amssymb,amsmath,verbatim,url}
\ifpdf
  \DeclareGraphicsExtensions{.eps,.pdf,.png,.jpg}
\else
  \DeclareGraphicsExtensions{.eps,.ps}
\fi
\newtheorem{remark}[theorem]{Remark}
\newcommand{\bRplus}{{\mathbb R}_{>0}}

\newcommand{\RZ}{{\mathbb R} \slash {\mathbb Z}}

\newcommand{\bR}{{\mathbb R}}

\newcommand{\bD}{{\mathbb D}}
\newcommand{\bH}{{\mathbb H}}

\newcommand{\Ass}{\mathcal{C}}
\newcommand{\spa}{\operatorname{span}}

\newcommand{\ratio}{{\mathfrak r}}

\newcommand{\drho}{\;{\rm d}\rho}

\newcommand{\Vh}{\underline{V}^h}

\newcommand{\Id}{\rm Id}

\newcommand{\deldel}[1]{\frac{\delta}{{\delta}#1}}
\newcommand{\dd}[1]{\frac{\rm d}{{\rm d}#1}}
\newcommand{\ddt}{\dd{t}}

\newcommand{\ek}{e}
\newcommand{\ttau}{\Delta t}

\newcommand{\BGNwf}{\mathcal{U}}
\newcommand{\BGNwfwf}{\mathcal{W}}
\newcommand{\BGNpwf}{\mathcal{P}}
\newcommand{\BGNpwfwf}{\mathcal{Q}}
\def\epsilon{\varepsilon}

\newcommand{\mat}[1]{\underline{\underline{#1}}\rule{0pt}{0pt}}

\newcommand{\Ds}{{D}_s}
\newcommand{\mDs}{{\widehat{D}}_s}
\newcommand{\errorXx}{\|\Gamma - \Gamma^h\|_{L^\infty}}
\allowdisplaybreaks
\newcommand{\revised}[1]{{#1}}

\begin{document}
\title{
Stable discretizations of elastic flow \\ in {R}iemannian manifolds
}
\author{John W. Barrett\footnotemark[2] \and 
        Harald Garcke\footnotemark[3]\ \and 
        Robert N\"urnberg\footnotemark[2]}

\renewcommand{\thefootnote}{\fnsymbol{footnote}}
\footnotetext[2]{Department of Mathematics, 
Imperial College London, London, SW7 2AZ, UK}
\footnotetext[3]{Fakult{\"a}t f{\"u}r Mathematik, Universit{\"a}t Regensburg, 
93040 Regensburg, Germany}

\date{}

\maketitle

\begin{abstract}
  The elastic flow, which is the $L^2$-gradient flow of the elastic energy,
  has several applications in geometry and elasticity theory. We
  present stable discretizations for the elastic flow in two-dimensional
  Riemannian manifolds that are conformally flat, i.e.\ conformally
  equivalent to the Euclidean space. Examples include the hyperbolic
  plane, the hyperbolic disk, the elliptic plane as well as any
  conformal parameterization of a two-dimensional manifold in ${\mathbb R}^d$,
  $d\geq 3$. 
  Numerical results show the robustness of the method, as
  well as quadratic convergence with respect to the space discretization.
\end{abstract} 


\renewcommand{\thefootnote}{\arabic{footnote}}

\pagestyle{myheadings}
\thispagestyle{plain}
\markboth{J. W. BARRETT, H. GARCKE, AND R. N\"URNBERG}
{ELASTIC FLOW IN {R}IEMANNIAN MANIFOLDS}

\setcounter{equation}{0}
\section{Introduction} \label{sec:intro}

Elastic flow of
curves in a two-dimensional Riemannian manifold $({\mathcal M}, g)$ is
given as the $L^2$-gradient flow of the elastic energy $\frac12
\int\varkappa^2_g$, where $\varkappa_g$ is the geodesic
curvature. It has been shown, see \cite{hypbol} for the general case
and \cite{DallAcquaS17preprint} for the hyperbolic plane, that the gradient flow
of the elastic energy is given as 
\begin{equation} \label{eq:g_elastflow}
\mathcal{V}_g = - (\varkappa_g)_{s_gs_g} - \tfrac12\,\varkappa_g^3 
- S_0\,\varkappa_g\,,
\end{equation}
where $\mathcal{V}_g$ is the normal velocity of the curve with respect
to the metric $g$, $\partial_{s_g} = g^{-\frac12}\,\partial_s$, $s$
denoting arclength, and $S_0$ is the sectional curvature of $g$.
The evolution law \eqref{eq:g_elastflow} decreases the curvature
energy $\frac12 \int\varkappa^2_g$, 
and long term limits are expected to be critical points of this
energy. 
These critical points are called free
elasticae, \revised{and are of interest in geometry, \cite{LangerS84}, and
mechanics, \cite{Truesdell83,Antman95}.}
In particular, let us mention that a curve is an absolute
minimizer if and only if it is a geodesic.
Recently the flow \eqref{eq:g_elastflow} was studied in
\cite{DallAcquaS17preprint,DallAcquaS18}, for the case of the
hyperbolic plane, relying on earlier results in
\cite{DziukKS02} for a flat background metric. The hyperbolic plane 
is a particular case of a manifold with non-positive
sectional curvature, which is of particular interest as the set of free
elasticae is much richer, see \cite{LangerS84}.

In this paper, we allow for a general conformally flat metric. 
Examples include
the hyperbolic plane, the hyperbolic disk, the elliptic plane, as
well as any conformal parameterization of a two-dimensional manifold in
$\mathbb{R}^d$, $d\ge 3$. For parameterized hypersurfaces in $\bR^3$,
earlier authors, see e.g.\ 
\cite{BrunnettC94,Linner04,LinnerR05,ArroyoGM06,curves3d}, used the
surrounding space in their numerical approximations, which
leads to errors in directions normal to the hypersurface. This will be
avoided by the intrinsic approach used in this paper. In particular,
our numerical method leads to approximate solutions which remain on
the hypersurface after application of the parameterization map.
In addition, in this paper we will present a first numerical analysis
for elastic flow in manifolds not embedded in $\mathbb{R}^3$. 
This in particular makes it possible to compute elastic flow of curves in the
hyperbolic plane in a stable way. 

For finite element approximations of
(\ref{eq:g_elastflow}) introduced in \cite{hypbol} it does not appear
possible to prove a stability result. It is the aim of this paper to
introduce novel approximations for (\ref{eq:g_elastflow}) that can be
shown to be stable. In particular, we will show that the semidiscrete
continuous-in-time approximations admit a gradient flow structure.
For relevant literature on conformal metrics we refer to
\cite{Schippers07,KrausR13}. Curvature driven flows in hyperbolic
spaces have been studied by \cite{Cabezas-RivasM07,AndrewsC17,
DallAcquaS17preprint,DallAcquaS18,hypbol}, and related numerical
approximations of elastic flow of curves can be found in 
\cite{DziukKS02,DeckelnickD09,pwf,Bartels13a} for the Euclidean case,
and in \cite{BrunnettC94,Linner04,LinnerR05,ArroyoGM06,curves3d}
for the case of curves on hypersurfaces in $\bR^3$.

The outline of this paper is as follows. After formulating the problem
in detail in the next section, we will derive in Section~\ref{sec:weak} 
weak formulations which will be the basis for our finite element 
approximation.
In Section~\ref{sec:sd} we introduce continuous-in-time, discrete-in-space
discretizations which are based on the weak formulations. For these
semidiscrete formulations a stability result will be shown, which is the
main contribution of this work. In Section~\ref{sec:fd} we then formulate fully
discrete variants for which we show existence and uniqueness. 
In Section~\ref{sec:nr} 
we present several numerical computations which show convergence
rates as well as the robustness of the approach. Finally, in the
appendix we show the consistency of the weak formulations presented in
Section~\ref{sec:weak}. 

\setcounter{equation}{0}
\section{Mathematical formulations} \label{sec:old1}

Let $I=\RZ$ be the periodic interval $[0,1]$.
Let $\vec x : I \to \bR^2$ 
be a parameterization of a closed curve $\Gamma \subset \bR^2$. 
On assuming that $|\vec x_\rho| > 0$ on $I$, 
we introduce the arclength $s$ of the curve, i.e.\ $\partial_s =
|\vec{x}_\rho|^{-1}\,\partial_\rho$, and set
\begin{equation} \label{eq:tau}
\vec\tau = \vec x_s 
\qquad \mbox{and}
\qquad \vec\nu = -\vec\tau^\perp\,,
\end{equation}
where $\cdot^\perp$ denotes a clockwise rotation by $\frac{\pi}{2}$.
For the curvature $\varkappa$ of \revised{$\Gamma$} it holds that
\begin{equation} \label{eq:varkappa}
\varkappa\,\vec\nu = \vec\varkappa = \vec\tau_s = \vec x_{ss} = 
\frac1{|\vec x_\rho|} \left[ \frac{\vec x_\rho}{|\vec x_\rho|} \right]_\rho .
\end{equation}

Let $H \subset \bR^2$ be an open set with metric tensor
\begin{equation} \label{eq:g}
[(\vec v, \vec w)_g](\vec z) = g(\vec z)\,\vec v\,.\,\vec w \quad
\forall\ \vec v, \vec w \in \bR^2
\qquad \text{ for } \vec z \in H\,,
\end{equation}
where $\vec v\,.\,\vec w = \vec v^T\,\vec w$ is the standard Euclidean inner
product, and where $g:H \to \bRplus$ is a smooth positive weight 
function. The length induced by (\ref{eq:g}) is defined as
\begin{equation} \label{eq:normg}
[|\vec v|_g](\vec z) = \left([(\vec v, \vec v)_g](\vec z) \right)^\frac12
= g^\frac12(\vec z)\,|\vec v| \quad \forall\ \vec v \in \bR^2
\qquad \text{ for } \vec z \in H\,.
\end{equation}

For $\lambda \in \bR$, we define the generalized elastic energy as
\begin{equation} \label{eq:Wglambda0}
W_{g,\lambda}(\vec x) = \tfrac12\,\int_I (\varkappa_g^2 + 2\,\lambda)\, 
|\vec x_\rho|_g \drho\,,
\end{equation}
where 
\begin{equation}\label{eq:varkappag}
  \varkappa_g = g^{-\frac12}(\vec{x})\left[\varkappa -\tfrac12\,
  \vec{\nu}\,.\, \nabla\, \ln g(\vec{x})\right]
\end{equation}
is the curvature of the curve with respect to the metric $g$,
see \cite{hypbol} for details.
Generalized elastic flow is defined as the $L^2$--gradient flow of 
(\ref{eq:Wglambda0}), and it was established in \cite{hypbol} that a strong
formulation is given by
\begin{equation} \label{eq:g_elastflowlambda}
\mathcal{V}_g = g^{\frac12}(\vec{x})\,\vec{x}_t\,.\,\vec{\nu} 
= - (\varkappa_g)_{s_gs_g} - \tfrac12\,\varkappa_g^3 
- S_0\,\varkappa_g + \lambda\,\varkappa_g\,,
\end{equation}
where $\partial_{s_g} = g^{-\frac12}(\vec x)\,\partial_s$ and
\begin{equation} \label{eq:Gaussg}
S_0 = - \frac{\Delta\,\ln\,g}{2\,g}
\end{equation}
is the sectional curvature of $g$. We refer to \cite{hypbol} for further
details. 

The two weak formulations of (\ref{eq:g_elastflowlambda}), for $\lambda=0$, 
introduced in \cite{hypbol} are based on the equivalent equation
\begin{equation} \label{eq:g_elastflow2}
g(\vec x)\,\vec x_t\,.\,\vec\nu = 
- \frac1{|\vec x_\rho|} 
\left(  \frac{[\varkappa_g]_\rho}{g^{\frac12}(\vec x)\,|\vec
x_\rho|}\right)_\rho 
- \tfrac12\,g^\frac12(\vec x)\,\varkappa_g^3 
- g^\frac12(\vec x)\,S_0(\vec x)\,\varkappa_g\,.
\end{equation}
The first uses $\varkappa$ as a variable, while the
second uses $\varkappa_g$ as a variable. 
\\ \noindent
$(\BGNwf)$:
Let $\vec x(0) \in [H^1(I)]^2$. For $t \in (0,T]$
find $\vec x(t) \in [H^1(I)]^2$ and $\varkappa(t)\in H^1(I)$ such that
\begin{subequations}
\begin{align}
& \int_I g(\vec x)\,\vec x_t\,.\,\vec\nu\,\chi\,|\vec x_\rho|\drho
= \int_I g^{-\frac12}(\vec x)\,
\left( g^{-\frac12}(\vec x)\left[\varkappa - 
\tfrac12\,\vec\nu\,.\,\nabla\,\ln g(\vec x)
 \right] \right)_\rho \chi_\rho\,|\vec x_\rho|^{-1} \drho 
\nonumber \\ & \qquad
- \tfrac12\, \int_I g^{-1}(\vec x) \left[
\varkappa -\tfrac12\,\vec\nu\,.\,\nabla\,\ln g(\vec x)
  \right]^3 \chi\,|\vec x_\rho| \drho 
\nonumber \\ & \qquad
- \int_I S_0(\vec x) 
\left[ \varkappa -\tfrac12\,\vec\nu\,.\,\nabla\,\ln g(\vec x)
\right] \chi\,|\vec x_\rho| \drho 
\quad \forall\ \chi \in H^1(I)\,, \label{eq:sdwfa} \\ &
\int_I \varkappa\,\vec\nu\,.\,\vec\eta\, |\vec x_\rho| \drho
+ \int_I (\vec x_\rho\,.\,\vec\eta_\rho)\,|\vec x_\rho|^{-1} \drho = 0
\quad \forall\ \vec\eta \in [H^1(I)]^2\,.  \label{eq:sdwfb} 
\end{align}
\end{subequations}
$(\BGNwfwf)$:
Let $\vec x(0) \in [H^1(I)]^2$. For $t \in (0,T]$
find $\vec x(t) \in [H^1(I)]^2$ and $\varkappa_g(t) \in H^1(I)$ such that
\begin{subequations}
\begin{align}
& \int_I g(\vec x)\,\vec x_t\,.\,\vec\nu\,
\chi\,|\vec x_\rho|\drho
= \int_I g^{-\frac12}(\vec x)\,[\varkappa_g]_\rho\,\chi_\rho\,
|\vec x_\rho|^{-1} \drho 
- \tfrac12\, \int_I g^\frac12(\vec x)\,
\varkappa_g^3 \,\chi\,|\vec x_\rho| \drho 
\nonumber \\ & \hspace{4cm}
- \int_I S_0(\vec x)\,g^{\frac12}(\vec x)\,\varkappa_g
\,\chi\,|\vec x_\rho| \drho \quad \forall\ \chi \in H^1(I)\,,
\label{eq:sdwfwfa} \\ &
 \int_I g(\vec x)\,\varkappa_g\,\vec\nu\,.\,
\vec\eta\,|\vec x_\rho|\drho
+ \int_I \left[\nabla\,g^\frac12(\vec x)\,.\,\vec\eta
+ g^\frac12(\vec x)\,\frac{\vec x_\rho\,.\,\vec\eta_\rho}{|\vec x_\rho|^2}
\right] |\vec x_\rho| \drho 
= 0 \quad \forall\ \vec\eta \in [H^1(I)]^2\,.
\label{eq:sdwfwfb}
\end{align}
\end{subequations}
For the numerical approximations based on $(\BGNwf)$ and $(\BGNwfwf)$ it does
not appear possible to prove stability results that show that 
discrete analogues of (\ref{eq:Wglambda0}), for $\lambda=0$, 
decrease monotonically in time.
\revised{%
The reason is that these formulations are directly based on the divergence form
in \eqref{eq:g_elastflow2}, which does not immediately
capture the variational structure of the gradient flow for 
\eqref{eq:Wglambda0}. For the same reason, the first stable finite element
approximations for Euclidean elastic flow were given in \cite{DeckelnickD09},
whereas the earlier schemes in \cite{DziukKS02} do not appear to admit a
stability proof.
In fact, based on the ideas in \cite{DeckelnickD09}, and utilizing the 
techniques in \cite{pwf}, it is possible to 
introduce alternative weak formulations of \eqref{eq:g_elastflowlambda},
for which semidiscrete continuous-in-time finite element
approximations admit such a stability result. 
Novel aspects compared to our previous work \cite{pwf} include the highly
nonlinear nature of the energy \eqref{eq:Wglambda0} and the curvature}
\revised{ 
definition \eqref{eq:varkappag}. Their variations give rise to new
nontrivial contributions, see e.g.\ \eqref{eq:33} below. Moreover, exploiting
the variational structure of the problem by treating $\varkappa_g$ as an
independent variable is, of course, new to this work compared to 
\cite{pwf}.}

We end this section with some example metrics that are of particular interest
in differential geometry. Two families of metrics are given by
\begin{subequations}
\begin{equation} \label{eq:gmu}
g(\vec z) = (\vec z \,.\,\vec\ek_2)^{-2\,\mu}\,,\ \mu \in \bR\,,
\quad \text{with} \quad
H = \bH^2 = \{ \vec z \in \bR^2 : \vec z \,.\,\vec\ek_2 > 0 \}\,,
\end{equation}
and
\begin{equation} \label{eq:galpha}
g(\vec z) = \frac4{(1 - \alpha\, |\vec z|^2)^2}\,,
\quad \text{with}\quad H = \begin{cases}
\bD_{\alpha} 
= \{ \vec z \in \bR^2 : |\vec z| < \alpha^{-\frac12} \}
& \alpha > 0\,, \\
\bR^2 & \alpha \leq 0\,.
\end{cases}
\end{equation} 
The case (\ref{eq:gmu}) with $\mu=1$ models the hyperbolic plane, 
while $\mu=0$ corresponds to the Euclidean case.
The case (\ref{eq:galpha}) with $\alpha=1$ gives a model for the
hyperbolic disk, while $\alpha=-1$ models the 
geometry of the elliptic plane. Of course, $\alpha=0$ collapses
to the Euclidean case.

Further metrics of interest are induced by conformal parameterizations
$\vec\Phi : H \to \bR^d$, $d\geq3$, of the two-dimensional Riemannian manifold 
$\mathcal{M} \subset \bR^d$, i.e.\ $\mathcal{M} = \vec\Phi(H)$ 
and $|\partial_{\vec\ek_1} \vec\Phi(\vec z)|^2 = |\partial_{\vec\ek_2} 
\vec\Phi(\vec z)|^2$ and $\partial_{\vec\ek_1} \vec\Phi(\vec z) \,.\, 
\partial_{\vec\ek_2} \vec\Phi(\vec z) = 0$ for all $\vec z \in H$.
Here examples include
\revised{
the stereographic projection of the unit sphere without the north pole, 
$\vec\Phi(\vec z) = (1 + |\vec z|^2)^{-1}\,
(2\,\vec z\,.\,\vec\ek_1, 2\,\vec z\,.\,\vec\ek_2, |\vec z|^2 - 1)^T$, so that
$g(\vec z) = 4\,(1 + |\vec z|^2)^{-2}$ and $H = \bR^2$,
which yields a geometric interpretation to (\ref{eq:galpha}) with $\alpha=-1$. 
Further examples are}
the Mercator projection of the unit sphere without
the north and the south pole, 
$\vec\Phi(\vec z) = \cosh^{-1}(\vec z\,.\,\vec\ek_1)\,
(\cos (\vec z\,.\,\vec\ek_2), \sin (\vec z\,.\,\vec\ek_2), 
\sinh (\vec z\,.\,\vec\ek_1))^T$, so that
\begin{equation} \label{eq:gMercator}
g(\vec z) = \cosh^{-2}(\vec z\,.\,\vec\ek_1) \,,\quad\text{with}\quad
H = \bR^2\,,
\end{equation}
as well as the catenoid parameterization
$\vec\Phi(\vec z) = 
(\cosh (\vec z\,.\,\vec\ek_1)\,\cos (\vec z\,.\,\vec\ek_2), 
\cosh (\vec z\,.\,\vec\ek_1)\,$ $\sin (\vec z\,.\,\vec\ek_2),
\vec z\,.\,\vec\ek_1)^T$, so that
\begin{equation} \label{eq:gcatenoid}
g(\vec z) = \cosh^2(\vec z\,.\,\vec\ek_1)\,, \quad\text{with}\quad
H = \bR^2\,.
\end{equation}
Based on \cite[p.\,593]{Sullivan11} we also recall the following conformal 
parameterization of a torus with large radius $R > 1$ and small radius $r=1$
from \cite{hypbol}. 
In particular, we let $\mathfrak s = [R^2 - 1]^\frac12$ and define
$\vec\Phi(\vec z) = 
\mathfrak s\,
([\mathfrak s^2 + 1]^\frac12- \cos (\vec z\,.\,\vec\ek_2))^{-1}\,
(\mathfrak s\, \cos \tfrac{\vec z\,.\,\vec\ek_1}{\mathfrak s},
\mathfrak s\, \sin \tfrac{\vec z\,.\,\vec\ek_1}{\mathfrak s},
\sin (\vec z\,.\,\vec\ek_2))^T$, so that
\begin{equation} \label{eq:gtorus}
g(\vec z) = \mathfrak s^2\,
([\mathfrak s^2 + 1]^\frac12 - \cos (\vec z\,.\,\vec\ek_2))^{-2} \,,
\quad\text{with}\quad H = \bR^2\,. 
\end{equation}
\end{subequations}

\setcounter{equation}{0}
\section{Weak formulations} \label{sec:weak}

We define the first variation of a quantity depending in a
differentiable way on $\vec x$, in the
direction $\vec\chi$ as
\begin{equation} \label{eq:A}
\left[\deldel{\vec x}\,A(\vec x)\right](\vec\chi)
= \lim_{\varepsilon \rightarrow 0}
\frac{A(\vec x + \varepsilon\, \vec \chi)-A(\vec x)}{\varepsilon}\,,
\end{equation}
and observe that, for $\vec x$ sufficiently smooth,
\begin{equation} \label{eq:ddA}
\left[\deldel{\vec x}\, A(\vec x) \right] (\vec x_t) 
= \ddt\, A(\vec x)\,.
\end{equation}
For later use, on noting (\ref{eq:A}), (\ref{eq:tau}) and (\ref{eq:normg}), we observe that 
\begin{subequations} \label{eq:33}
\begin{align}
\left[\deldel{\vec x}\, g^{\beta}(\vec x) \right](\vec\chi)
& = \beta\,g^{\beta-1}(\vec x)\,\vec\chi\,.\,\nabla\,g(\vec x) =
\beta\,g^{\beta}(\vec x)\,\vec\chi\,.\,\nabla\,\ln\,g(\vec x)
\quad \forall \ \beta \in {\mathbb R}\,,
\label{eq:ddgbeta} \\
\left[\deldel{\vec x}\, \nabla\,\ln\,g(\vec x)\right](\vec\chi) & =
(D^2\,\ln\,g(\vec x))\,\vec\chi\,, \label{eq:ddnablalng}\\
\left[\deldel{\vec x}\, |\vec x_\rho|\right](\vec\chi)
& = \frac{\vec x_\rho\,.\,\vec\chi_\rho}{|\vec x_\rho|}
= \vec\tau\,.\,\vec\chi_\rho = 
\vec\tau\,.\,\vec\chi_s\, |\vec x_\rho|\,, \label{eq:ddxrho} \\
\left[\deldel{\vec x}\, |\vec x_\rho|_g\right](\vec\chi)
& = \left( \vec\tau\,.\,\vec\chi_s + 
\tfrac12\,\vec\chi\,.\,\nabla\,\ln\,g(\vec x) \right) |\vec x_\rho|_g\,,
\label{eq:ddxrhog} \\
\left[\deldel{\vec x}\, \vec\tau\right](\vec\chi)
& = \left[\deldel{\vec x}\,\frac{\vec x_\rho}{|\vec x_\rho|}
\right](\vec\chi) = \frac{\vec\chi_\rho}{|\vec x_\rho|}
- \frac{\vec x_\rho}{|\vec x_\rho|^2}\,
 \frac{\vec x_\rho\,.\,\vec\chi_\rho}{|\vec x_\rho|}
= \vec\chi_s - \vec\tau\,(\vec\chi_s\,.\,\vec\tau) \nonumber \\ &
= (\vec\chi_s\,.\,\vec\nu)\,\vec\nu \,, \label{eq:ddtau} \\
\left[\deldel{\vec x}\, \vec\nu\right](\vec\chi)
& = - \left[\deldel{\vec x}\, \vec\tau^\perp\right](\vec\chi) = 
- (\vec\chi_s\,.\,\vec\nu)\,\vec\nu^\perp =
- (\vec\chi_s\,.\,\vec\nu)\,\vec\tau \,, \label{eq:ddnu} \\
\left[\deldel{\vec x}\,\vec \nu\,|\vec x_\rho|\right](\vec\chi) &=
- \left[\deldel{\vec x}\,\vec x_{\rho}^\perp\right](\vec\chi)
= - \vec \chi_\rho^\perp = - \vec \chi_s^\perp \,|\vec x_\rho|\,, 
\label{eq:ddnuxrho}
\end{align}
\end{subequations}
where we always assume that $\vec\chi$ is sufficiently smooth so that all the
quantities are defined almost everywhere. E.g.\ 
$\vec\chi \in [L^\infty(I)]^2$ for (\ref{eq:ddgbeta}), (\ref{eq:ddnablalng}),
and $\vec\chi \in [W^{1,\infty}(I)]^2$ for 
(\ref{eq:ddxrho})--(\ref{eq:ddnuxrho}).
In addition, on recalling (\ref{eq:tau}), we have for all 
$\vec a,\,\vec b \in \bR^2$ that
\begin{subequations}
\begin{align}
&\vec a \,.\,\vec b^\perp= -\vec a^\perp .\,\vec b\,, \label{eq:abperp}\\
&\vec a^\perp 
= (\vec a^\perp\,.\,\vec\tau)\,\vec\tau +(\vec a^\perp\,.\,\vec\nu)\,\vec\nu
= (\vec a^\perp\,.\,\vec\nu^\perp)\,\vec\tau  
  - (\vec a^\perp\,.\,\vec\tau^\perp)\,\vec\nu
= (\vec a\,.\,\vec\nu)\,\vec\tau  
  - (\vec a\,.\,\vec\tau)\,\vec\nu \,. \label{eq:aperp}
\end{align}
\end{subequations}

Let $(\cdot,\cdot)$ denote the $L^2$--inner product on $I$.
In the following we will discuss the $L^2$--gradient flow of the energy 
\begin{equation} \label{eq:Wglambda}
W_{g,\lambda}(\vec x) =  \left( \tfrac12\,\varkappa_g^2 + \lambda, |\vec
x_\rho|_g \right)
= \left(\tfrac12\,g^{-\frac12}(\vec x)
\left(\varkappa - \tfrac12\,\vec\nu\,.\,\nabla\,\ln g(\vec x) \right)^2 
+ \lambda\,g^\frac12(\vec x),|\vec x_\rho| \right) ,
\end{equation}
treating either $\varkappa$ or $\varkappa_g$ formally as an independent
variable that has to satisfy the side constraint (\ref{eq:sdwfb}), or
(\ref{eq:sdwfwfb}), respectively. 
\revised{%
The necessary techniques are obtained from the formal calculus of 
PDE constrained optimization, and were used by the authors for the first
time in \cite{pwf} in the present context.}
For the weak formulations of the $L^2$--gradient flow obtained in this
way, it can be shown that they are consistent with the strong formulation
(\ref{eq:g_elastflowlambda}), see the appendix.
Moreover, we will formally establish that solutions to
these weak formulations are indeed solutions to the $L^2$--gradient flow
of (\ref{eq:Wglambda}). 
Mimicking these stability proofs on the discrete level, 
\revised{which in essence
reduces to the seminal idea introduced in \cite[Rem.\ 2.1]{DeckelnickD09},} 
will yield the main \revised{theoretical} results of this paper.

\subsection{Based on $\varkappa$} \label{sec:31}

We define the Lagrangian
\begin{align}
\mathcal{L}(\vec x, \varkappa^\star, \vec y) & = 
\tfrac12\left( g^{-\frac12}(\vec x)
\left(\varkappa^\star - \tfrac12\,\vec\nu\,.\,\nabla\,\ln g(\vec x) \right)^2 
+ 2\,\lambda \,g^\frac12(\vec x), |\vec x_\rho| \right) 
\nonumber \\ & \quad
- \left( \varkappa^\star\,\vec\nu, \vec y\,|\vec x_\rho| \right) 
- \left( \vec x_s,\vec y_s\,|\vec x_\rho|\right) , \label{eq:Lag}
\end{align}
which is obtained on combining (\ref{eq:Wglambda}) and the side constraint 
\begin{equation} \label{eq:kappastar}
\left(\varkappa^\star\,\vec\nu,\vec\eta\, |\vec x_\rho| \right)
+ \left( \vec x_s,\vec\eta_s\,|\vec x_\rho| \right) = 0
\quad \forall\ \vec\eta \in [H^1(I)]^2\,,
\end{equation}
recall (\ref{eq:sdwfb}) and (\ref{eq:tau}). 
Taking variations $\vec\eta \in [H^1(I)]^2$ in $\vec y$, and setting
$\left[\deldel{\vec y}\, \mathcal{L}\right](\vec\eta) = 0$
we obtain (\ref{eq:kappastar}). Combining (\ref{eq:kappastar}) and
(\ref{eq:sdwfb}) yields, on recalling (\ref{eq:tau}), 
that $\varkappa^\star = \varkappa$,
and we are going to use this identity from now.
Taking variations $\chi \in L^2(I)$ in $\varkappa^\star$ and setting
$\left[\deldel{\varkappa^\star}\, \mathcal{L}\right](\chi) = 0$ leads to
\begin{equation} \label{eq:kappay}
\left(g^{-\frac12}(\vec x)
\left(\varkappa - \tfrac12\,\vec\nu\,.\,\nabla\,\ln g(\vec x) \right) 
- \vec y\,.\,\vec\nu, \chi\,|\vec x_\rho| \right) = 0
\qquad \forall\ \chi \in L^2(I)\,,
\end{equation}
which implies that
\begin{equation} \label{eq:kappaid}
\vec y\,.\,\vec\nu = g^{-\frac12}(\vec x)
\left(\varkappa - \tfrac12\,\vec\nu\,.\,\nabla\,\ln g(\vec x) \right) 
\quad\iff\quad
\varkappa = g^{\frac12}(\vec x)\,\vec y\,.\,\vec\nu + \tfrac12\,
\vec\nu\,.\,\nabla\,\ln g(\vec x)\, .
\end{equation}
Taking variations $\vec\chi \in [H^1(I)]^2$ in $\vec x$, and then setting
$(\mathcal V_g, g^\frac12(\vec x)\,\vec\chi\,.\,\vec\nu\,
|\vec x_\rho|_g ) = $ \linebreak $
(g^\frac32(\vec x)\,\vec x_t\,.\,\vec\nu, \vec\chi\,.\,\vec\nu\,
|\vec x_\rho| ) = - \left[\deldel{\vec x}\, \mathcal{L}\right](\vec\chi)$,
where we have noted (\ref{eq:g_elastflowlambda}) and (\ref{eq:normg}),
yields, on recalling (\ref{eq:tau}), that
\begin{align}
& \left(g^\frac32(\vec x)\,\vec x_t\,.\,\vec\nu, \vec\chi\,.\,\vec\nu\,
|\vec x_\rho| \right) = 
-\tfrac12\left(
\left(\varkappa - \tfrac12\,\vec\nu\,.\,\nabla\,\ln g(\vec x) \right)^2 , 
\left[\deldel{\vec x}\,g^{-\frac12}(\vec x)\,|\vec x_\rho|\right] (\vec\chi) 
\right) \nonumber \\ & \quad
+ \tfrac12 
\left( g^{-\frac12}(\vec x)
\left(\varkappa - \tfrac12\,\vec\nu\,.\,\nabla\,\ln g(\vec x) \right), 
\left[\deldel{\vec x}\, \vec \nu\,.\,\nabla\,\ln\,g(\vec x) \right] 
(\vec\chi) \,|\vec x_\rho| \right) \nonumber \\ & \quad
+ \left( \varkappa\,\vec y, 
\left[\deldel{\vec x}\,\vec\nu\,|\vec x_\rho| \right] (\vec\chi) \right)
+ \left( \vec y_\rho,  \left[\deldel{\vec x}\, \vec\tau \right] (\vec\chi) 
\right) 
- \lambda \left( 1, 
\left[\deldel{\vec x}\,g^{\frac12}(\vec x)\,|\vec x_\rho|\right] (\vec\chi) 
\right) \,, \label{eq:dLdx}
\end{align}
for all $\vec\chi\in [H^1(I)]^2$.
On choosing $\vec\chi = \vec x_t$ in (\ref{eq:dLdx}) we obtain, on noting
(\ref{eq:ddA}), that 
\begin{align}
& \left(g^\frac32(\vec x)\,(\vec x_t\,.\,\vec\nu)^2, |\vec x_\rho| \right) = 
-\tfrac12\left(
\left(\varkappa - \tfrac12\,\vec\nu\,.\,\nabla\,\ln g(\vec x) \right)^2 , 
\left[g^{-\frac12}(\vec x)\,|\vec x_\rho|\right]_t 
\right) \nonumber \\ & \quad
+ \tfrac12 
\left( g^{-\frac12}(\vec x)
\left(\varkappa - \tfrac12\,\vec\nu\,.\,\nabla\,\ln g(\vec x) \right), 
\left[\vec \nu\,.\,\nabla\,\ln\,g(\vec x) \right]_t
|\vec x_\rho| \right) \nonumber \\ & \quad
+ \left( \varkappa\,\vec y, 
\left[\vec\nu\,|\vec x_\rho| \right]_t \right)
+ \left( \vec y_\rho,  \vec\tau_t \right) 
- \lambda \left( 1, 
\left[g^{\frac12}(\vec x)\,|\vec x_\rho|\right]_t \right)
. \label{eq:dLdxxt}
\end{align}

Differentiating (\ref{eq:kappastar}) with respect to time, and then choosing
$\vec\eta = \vec y$ yields,
on recalling that $\varkappa^\star = \varkappa$, that
\begin{equation}
\left( \varkappa_t, \vec y\,.\,\vec\nu\,|\vec x_\rho| \right) 
+ \left( \varkappa\,\vec y, (\vec\nu\,|\vec x_\rho|)_t \right) 
+ \left(\vec\tau_t,\vec y_\rho \right) = 0 \,.
\label{eq:dtside}
\end{equation}
Combining (\ref{eq:dLdxxt}), (\ref{eq:dtside}) and (\ref{eq:kappaid}) gives,
on noting (\ref{eq:Wglambda}), that
\begin{align} \label{eq:dLstab}
&
\left(g^\frac32(\vec x)\,(\vec x_t\,.\,\vec\nu)^2, |\vec x_\rho| \right) = 
-\tfrac12\left(
\left(\varkappa - \tfrac12\,\vec\nu\,.\,\nabla\,\ln g(\vec x) \right)^2 , 
\left[g^{-\frac12}(\vec x)\,|\vec x_\rho|\right]_t 
\right) \nonumber \\ & \quad
+ \tfrac12 
\left( g^{-\frac12}(\vec x)
\left(\varkappa - \tfrac12\,\vec\nu\,.\,\nabla\,\ln g(\vec x) \right), 
\left[\vec \nu\,.\,\nabla\,\ln\,g(\vec x) \right]_t
|\vec x_\rho| \right) \nonumber \\ & \quad
- \left( \varkappa_t, 
g^{-\frac12}(\vec x)
\left(\varkappa - \tfrac12\,\vec\nu\,.\,\nabla\,\ln g(\vec x) \right) 
|\vec x_\rho| \right)
- \lambda \left( 1, 
\left[g^{\frac12}(\vec x)\,|\vec x_\rho|\right]_t \right)
\nonumber \\ & \quad
= - \ddt\,W_{g,\lambda}(\vec x)\,.
\end{align}
The above yields the gradient flow property of the new weak formulation,
on noting from (\ref{eq:g_elastflowlambda}) and
(\ref{eq:normg}) that the left hand side of 
(\ref{eq:dLstab}) can be equivalently written as $( \mathcal{V}_g^2, |\vec
x_\rho|_g)$.

In order to derive a suitable weak formulation, we now return to
(\ref{eq:dLdx}). 
Combining (\ref{eq:dLdx}), (\ref{eq:33}) 
and (\ref{eq:abperp}) yields that
\begin{align}
& \left(g^\frac32(\vec x)\,\vec x_t\,.\,\vec\nu, \vec\chi\,.\,\vec\nu\,
|\vec x_\rho| \right)
= -\tfrac12 \left( g^{-\frac12}(\vec x)
\left(\varkappa - \tfrac12\,\vec\nu\,.\,\nabla\,\ln g(\vec x) \right)^2
+ 2\,\lambda\,g^\frac12(\vec x), 
\vec\chi_s\,.\,\vec\tau\,|\vec x_\rho| \right) \nonumber \\ & \qquad 
+ \tfrac14 \left( g^{-\frac12}(\vec x)
 \left(\varkappa - \tfrac12\,\vec\nu\,.\,\nabla\,\ln g(\vec x) \right)^2
 - 2\,\lambda\,g^\frac12(\vec x),
\vec\chi\,.\,(\nabla\,\ln\,g(\vec x))\,|\vec x_\rho| \right)
\nonumber \\ & \qquad
+ \tfrac12 \left( 
g^{-\frac12}(\vec x)
\left(\varkappa - \tfrac12\,\vec\nu\,.\,\nabla\,\ln g(\vec x) \right) 
\,\vec \nu ,(D^2\,\ln\,g(\vec x))\,
\vec\chi\,|\vec x_\rho| \right) \nonumber \\ & \qquad
-\tfrac12\left( g^{-\frac12}(\vec x)
\left(\varkappa - \tfrac12\,\vec\nu\,.\,\nabla\,\ln g(\vec x) \right) 
[\ln\,g(\vec x)]_s, \vec\nu\,.\,\vec\chi_s
\,|\vec x_\rho| \right)
+ \left(\vec y_s\,.\,\vec\nu, \vec \chi_s\,.\,\vec\nu\,|\vec x_\rho| \right)
\nonumber \\ & \qquad
+ \left( \varkappa\, \vec y^\perp,
\vec\chi_s\,|\vec x_\rho| \right)
\qquad \forall  \ \vec \chi \in [H^1(I)]^2\,. 
\label{eq:dLdxflow2}
\end{align}

Overall we obtain the following weak formulation.
\\ \noindent
$(\BGNpwf)$:
Let $\vec x(0) \in [H^1(I)]^2$. For $t \in (0,T]$
find $\vec x(t),\,\vec y(t) \in [H^1(I)]^2$ and $\varkappa \in L^2(I)$ 
such that (\ref{eq:dLdxflow2}), (\ref{eq:kappay}) and 
\begin{equation} \label{eq:kappa}
\left(\varkappa\,\vec\nu,\vec\eta\, |\vec x_\rho| \right)
+ \left( \vec x_s,\vec\eta_s\,|\vec x_\rho| \right) = 0
\quad \forall\ \vec\eta \in [H^1(I)]^2
\end{equation}
hold.
We remark that in the Euclidean case (\ref{eq:kappay}) collapses to
$\varkappa = \vec y\,.\,\vec\nu$, and so on eliminating $\varkappa$
from (\ref{eq:dLdxflow2}) and (\ref{eq:kappa}), 
and on noting (\ref{eq:aperp}), we obtain that 
the formulation $(\BGNpwf)$ collapses
to \cite[(2.4a,b)]{pwf} for the Euclidean elastic flow.

\subsection{Based on $\varkappa_g$} \label{sec:32}

We recall that $(\BGNpwf)$ was inspired by the formulation
$(\BGNwf)$, which is based on $\varkappa$ acting as a variable. In order to
derive an alternative formulation, we now start from $(\BGNwfwf)$, where
the curvature $\varkappa_g$ is a variable.

We begin by equivalently rewriting the side constraint (\ref{eq:sdwfwfb}) as
\begin{equation} \label{eq:side}
\left(g^\frac12(\vec x)\,\varkappa_g\,\vec\nu, \vec\eta\,|\vec x_\rho|_g
\right) 
+ \left(\vec x_s,\vec\eta_s\,|\vec x_\rho|_g\right) 
+ \tfrac12 \left( \nabla\,\ln\,g(\vec x), \vec\eta\,|\vec x_\rho|_g\right) 
= 0
\quad \forall\ \vec\eta \in [H^1(I)]^2\,,
\end{equation}
where we have noted (\ref{eq:tau}), (\ref{eq:normg}) and
$\tfrac12\,\nabla\,\ln g(\vec x) = g^{-\frac12}(\vec x)\, \nabla\, g^{\frac12}
(\vec x)$.
Combining (\ref{eq:Wglambda0}) and (\ref{eq:side}) leads to the
Lagrangian
\begin{align}
\mathcal{L}_g(\vec x, \varkappa_g^\star, \vec y_g) & = 
\tfrac12\left((\varkappa_g^\star)^2 + 2\,\lambda, |\vec x_\rho|_g\right)
- \left( g^\frac12(\vec x)\,\varkappa_g^\star\,\vec\nu, 
\vec y_g\,|\vec x_\rho|_g
\right) 
- \left( \vec x_s,(\vec y_g)_s\,|\vec x_\rho|_g\right)
\nonumber \\ & \quad
- \tfrac12 \left( \nabla\,\ln\,g(\vec x), \vec y_g\,|\vec x_\rho|_g\right) .
\label{eq:Lagg}
\end{align}
Taking variations $\vec\eta \in [H^1(I)]^2$ in $\vec y_g$, and setting
$\left[\deldel{\vec y_g}\, \mathcal{L}_g\right](\vec\eta) = 0$
we obtain 
\begin{equation} \label{eq:kappagstar}
\left(g^\frac12(\vec x)\,\varkappa_g^\star\,\vec\nu, \vec\eta\,|\vec x_\rho|_g
\right) 
+ \left(\vec x_s,\vec\eta_s\,|\vec x_\rho|_g\right) 
+ \tfrac12 \left( \nabla\,\ln\,g(\vec x), \vec\eta\,|\vec x_\rho|_g\right) 
= 0
\quad \forall\ \vec\eta \in [H^1(I)]^2\,.
\end{equation}
Combining (\ref{eq:kappagstar}) and
(\ref{eq:side}) yields that $\varkappa_g^\star = \varkappa_g$,
and we are going to use this identity from now.
Taking variations $\chi \in L^2(I)$ in $\varkappa_g^\star$ and setting
$\left[\deldel{\varkappa_g^\star}\, \mathcal{L}_g\right](\chi) = 0$
yields that
\begin{equation} \label{eq:kappagy}
\left(\varkappa_g - g^\frac12(\vec x)\,\vec y_g\,.\,\vec\nu, 
\chi\,|\vec x_\rho|_g \right) = 0
\qquad \forall\ \chi \in L^2(I)\,,
\end{equation}
which implies that
\begin{equation} \label{eq:kappagid}
\varkappa_g = g^\frac12(\vec x)\,\vec y_g\,.\,\vec\nu\,.
\end{equation}
Taking variations $\vec\chi \in [H^1(I)]^2$ in $\vec x$,
and then setting 
$(\mathcal V_g, g^\frac12\,\vec\chi\,.\,\vec\nu\,
|\vec x_\rho|_g ) =$\linebreak 
$ (g(\vec x)\,\vec x_t\,.\,\vec\nu, \vec\chi\,.\,\vec\nu\,
|\vec x_\rho|_g ) = - \left[\deldel{\vec x}\, \mathcal{L}_g\right](\vec\chi)$,
where we have noted (\ref{eq:g_elastflowlambda}),
yields, on recalling (\ref{eq:tau}) and (\ref{eq:normg}), that
\begin{align}
& \left(g(\vec x)\,\vec x_t\,.\,\vec\nu, \vec\chi\,.\,\vec\nu\,
|\vec x_\rho|_g \right) 
= - \tfrac12 \left( \varkappa_g^2 + 2\,\lambda, 
\left[\deldel{\vec x}\, |\vec x_\rho|_g\right](\vec\chi) \right)
\nonumber \\ & \quad \quad
+\left( \varkappa_g\,\vec y_g, \left[\deldel{\vec x}\, g^\frac12(\vec x)\, 
\vec \nu\,|\vec x_\rho|_g \right] (\vec\chi) \right)
+ \left( (\vec y_g)_\rho,  \left[\deldel{\vec x}\, g^\frac12(\vec x)\,\vec\tau 
\right] (\vec\chi) \right)
\nonumber \\ & \quad \quad
+ \tfrac12 \left( \vec y_g, \left[\deldel{\vec x}\,(\nabla\,\ln\,g(\vec x))\,
|\vec x_\rho|_g \right] (\vec\chi) \right) 
\quad \forall\ \vec\chi \in [H^1(I)]^2\,.
\label{eq:dLgdxflow}
\end{align}
Choosing $\vec\chi = \vec x_t$ in (\ref{eq:dLgdxflow}), 
and noting (\ref{eq:ddA}), yields that
\begin{align}
& \left(g(\vec x)\,(\vec x_t\,.\,\vec\nu)^2, 
|\vec x_\rho|_g \right) 
= - \tfrac12 \left( (\varkappa_g)^2 + 2\,\lambda, 
(|\vec x_\rho|_g)_t \right)
+ \left( \varkappa_g\,\vec y_g, 
(g^\frac12(\vec x)\, \vec\nu \,|\vec x_\rho|_g)_t \right)
\nonumber \\ & \qquad
+ \left( (\vec y_g)_\rho,  (g^\frac12(\vec x)\,\vec\tau)_t 
\right)
+ \tfrac12 \left( \vec y_g, 
((\nabla\,\ln\,g(\vec x))\,|\vec x_\rho|_g )_t \right) .
\label{eq:b5xt}
\end{align}

On differentiating (\ref{eq:side}) with respect to time, and then choosing
$\vec\eta = \vec y_g$, we obtain, on recalling (\ref{eq:tau}) and
(\ref{eq:normg}), that
\begin{align} \label{eq:b3dt}
& 
\left((\varkappa_g)_t\,\vec y_g, g^\frac12(\vec x)\,\vec\nu\, 
|\vec x_\rho|_g\right)
+\left(\varkappa_g\,\vec y_g, (g^\frac12(\vec x)\,\vec\nu\, 
|\vec x_\rho|_g)_t \right)
\nonumber \\ & \qquad
+ \left((\vec y_g)_\rho,(g^\frac12(\vec x)\,\vec\tau)_t\right)
+ \tfrac12 \left(\vec y_g, ((\nabla\,\ln\,g(\vec x))\, 
|\vec x_\rho|_g)_t \right)
= 0 .
\end{align}
Choosing $\chi = (\varkappa_g)_t$ in (\ref{eq:kappagy}), and combining with
(\ref{eq:b5xt}) and (\ref{eq:b3dt}), yields, on recalling (\ref{eq:Wglambda}), 
that
\begin{equation} \label{eq:dLgstab}
\left(g(\vec x)\,(\vec x_t\,.\,\vec\nu)^2, 
|\vec x_\rho|_g \right) = - \ddt\,W_{g,\lambda}(\vec x) \,,
\end{equation}
which once again reveals the gradient flow structure,
on noting from (\ref{eq:g_elastflowlambda}) that the left hand side of 
(\ref{eq:dLgstab}) can be equivalently written as 
$( \mathcal{V}_g^2, |\vec x_\rho|_g)$.

In order to derive a suitable weak formulation, we now return to
(\ref{eq:dLgdxflow}). 
Substituting (\ref{eq:33}) 
into (\ref{eq:dLgdxflow})
yields, on noting (\ref{eq:normg}), that
\begin{align}
& \left(g(\vec x)\,\vec x_t\,.\,\vec\nu, \vec\chi\,.\,\vec\nu\,
|\vec x_\rho|_g \right) 
 \nonumber \\ & \
= -\tfrac12 \left( \varkappa_g^2 + 2\,\lambda
- \vec y_g\,.\,\nabla\,\ln\,g(\vec x), 
\left[\deldel{\vec x}\, |\vec x_\rho|_g\right](\vec\chi) \right)
\nonumber \\ & \quad 
+ \tfrac12 \left( \vec y_g, \left[\deldel{\vec x}\,(\nabla\,\ln\,g(\vec x))
\right] (\vec\chi)\,
|\vec x_\rho|_g  \right)
+ \left( \varkappa_g\,\vec y_g\,.\,\vec\nu, \left[\deldel{\vec x}\,
g(\vec x)\right](\vec\chi)\,|\vec x_\rho| \right)
\nonumber \\ & \quad
+ \left( g(\vec x)\,\varkappa_g\,\vec y_g, 
\left[\deldel{\vec x}\,\vec \nu\, |\vec x_\rho|\right](\vec\chi) \right)
+ \left( (\vec y_g)_\rho\,.\,\vec\tau ,  \left[\deldel{\vec x}\, 
g^\frac12(\vec x) \right] (\vec\chi) \right)
\nonumber \\ &\quad 
+ \left( g^\frac12(\vec x)\,(\vec y_g)_\rho,  \left[\deldel{\vec x}\, \vec\tau 
\right] (\vec\chi) \right)
\nonumber \\ &\
= -\tfrac12 \left( \varkappa_g^2 + 2\,\lambda
- \vec y_g\,.\,\nabla\,\ln\,g(\vec x), \left[\vec\tau\,.\,\vec\chi_s + 
\tfrac12\,\vec\chi\,.\,\nabla\,\ln\,g(\vec x)\right] |\vec x_\rho|_g
 \right) \nonumber \\ & \quad 
+ \tfrac12 \left((D^2\,\ln\,g(\vec x))\,\vec y_g, \vec\chi\,
|\vec x_\rho|_g \right) 
+ \left( g^\frac12(\vec x)\,\varkappa_g\,\vec y_g\,.\,\vec\nu
+ \tfrac12\,(\vec y_g)_s\,.\,\vec\tau, (\nabla\,\ln\,g(\vec x))\,.\,\vec\chi\,
|\vec x_\rho|_g \right) \nonumber \\ & \quad 
- \left( g^\frac12(\vec x)\,\varkappa_g\,\vec y_g,
\vec\chi_s^\perp\,|\vec x_\rho|_g \right)
+ \left( (\vec y_g)_s\,.\,\vec\nu, 
\vec\chi_s\,.\vec \nu\,|\vec x_\rho|_g  \right)
\qquad \forall\ \vec\chi \in [H^1(I)]^2\,.
\label{eq:dLgdxflow2}
\end{align}

Then, on recalling (\ref{eq:abperp}), we obtain the following weak formulation.
\\ \noindent
$(\BGNpwfwf)$:
Let $\vec x(0) \in [H^1(I)]^2$. For $t \in (0,T]$
find $\vec x(t),\,\vec y_g(t) \in [H^1(I)]^2$ and $\varkappa_g(t)\in L^2(I)$ 
such that
\begin{align} 
& \left(g(\vec x)\,\vec x_t\,.\,\vec\nu, \vec\chi\,.\,\vec\nu\,
|\vec x_\rho|_g \right) 
\nonumber \\ &\
= -\tfrac12 \left( \varkappa_g^2 + 2\,\lambda
- \vec y_g\,.\,\nabla\,\ln\,g(\vec x), \left[\vec\chi_s\,.\,\vec\tau + 
\tfrac12\,\vec\chi\,.\,\nabla\,\ln\,g(\vec x)\right] |\vec x_\rho|_g
 \right) \nonumber \\ & \quad 
+ \tfrac12 \left((D^2\,\ln\,g(\vec x))\,\vec y_g, \vec\chi\,
|\vec x_\rho|_g \right) 
+  \left( g^\frac12(\vec x)\,\varkappa_g\,\vec y_g\,.\,\vec\nu
+ \tfrac12\,(\vec y_g)_s\,.\,\vec\tau, \vec\chi\,.\,(\nabla\,\ln\,g(\vec x))\,
|\vec x_\rho|_g \right) \nonumber \\ & \quad 
+ \left( g^\frac12\,\varkappa_g,
\vec\chi_s\,.\,\vec y_g^\perp\,|\vec x_\rho|_g \right)
+ \left( (\vec y_g)_s\,.\,\vec\nu, 
\vec\chi_s\,.\vec \nu\,|\vec x_\rho|_g  \right)
\qquad \forall\ \vec\chi \in [H^1(I)]^2\,,\label{eq:BGNpwfwfa} 
\end{align}
(\ref{eq:kappagy}) and (\ref{eq:side}) hold.
We remark that in the Euclidean case (\ref{eq:kappagy}) collapses to
$\varkappa_g = \vec y_g\,.\,\vec\nu$, and so on eliminating $\varkappa_g$
from (\ref{eq:BGNpwfwfa}) and (\ref{eq:side}), 
and on noting (\ref{eq:aperp}), we obtain that 
the formulation $(\BGNpwfwf)$ collapses
to \cite[(2.4a,b)]{pwf} for the Euclidean elastic flow.

\setcounter{equation}{0}
\section{Semidiscrete finite element approximations} \label{sec:sd}

Let $[0,1]=\cup_{j=1}^J I_j$, $J\geq3$, be a
decomposition of $[0,1]$ into intervals given by the nodes $q_j$,
$I_j=[q_{j-1},q_j]$. 
For simplicity, and without loss of generality,
we assume that the subintervals form an equipartitioning of $[0,1]$,
i.e.\ that 
\begin{equation} \label{eq:Jequi}
q_j = j\,h\,,\quad \mbox{with}\quad h = J^{-1}\,,\qquad j=0,\ldots, J\,.
\end{equation}
Clearly, as $I=\RZ$ we identify $0=q_0 = q_J=1$.

The necessary finite element spaces are defined as follows:
\[
V^h = \{\chi \in C(I) : \chi\!\mid_{I_j} 
\mbox{ is linear}\ \forall\ j=1\to J\}
\quad\text{and}\quad \Vh = [V^h]^2\,.
\]
Let $\{\chi_j\}_{j=1}^J$ denote the standard basis of $V^h$,
and let $\pi^h:C(I)\to V^h$ 
be the standard interpolation operator at the nodes $\{q_j\}_{j=1}^J$.
We require also the local interpolation operator 
$\pi^h_j \equiv \pi^h \mid_{I_j}$, $j=1,\ldots,J$. 

We define the mass lumped $L^2$--inner product $(u,v)^h$,
for two piecewise continuous functions, with possible jumps at the 
nodes $\{q_j\}_{j=1}^J$, via
\begin{equation}
( u, v )^h = 
\sum_{j=1}^J \int_{I_j} \pi^h_j\,[u\,v] \,{\rm d}\rho
= \tfrac12\,\revised{h}\,\sum_{j=1}^J
\left[(u\,v)(q_j^-) + (u\,v)(q_{j-1}^+)\right],
\label{eq:ip0}
\end{equation}
where we define
$u(q_j^\pm)=\underset{\delta\searrow 0}{\lim}\ u(q_j\pm\delta)$.
The interpolation operators $\pi^h, \,\pi^h_j$ and the 
definition (\ref{eq:ip0}) naturally extend to vector valued functions.

Let $(\vec X^h(t))_{t\in[0,T]}$, with $\vec X^h(t)\in \Vh$, 
be an approximation to $(\vec x(t))_{t\in[0,T]}$. Then, 
similarly to (\ref{eq:tau}), we set
\begin{equation} \label{eq:tauh}
\vec\tau^h = \vec X^h_s = \frac{\vec X^h_\rho}{|\vec X^h_\rho|} 
\qquad \mbox{and} \qquad \vec\nu^h = -(\vec\tau^h)^\perp\,.
\end{equation}
For later use, we let $\vec\omega^h \in \Vh$ be the mass-lumped 
$L^2$--projection of $\vec\nu^h$ onto $\Vh$, i.e.\
\begin{equation} \label{eq:omegah}
\left(\vec\omega^h, \vec\varphi \, |\vec X^h_\rho| \right)^h 
= \left( \vec\nu^h, \vec\varphi \, |\vec X^h_\rho| \right)
= \left( \vec\nu^h, \vec\varphi \, |\vec X^h_\rho| \right)^h
\qquad \forall\ \vec\varphi\in\Vh\,.
\end{equation}

On noting (\ref{eq:A}), (\ref{eq:tauh}) and (\ref{eq:normg}), we have
the following discrete analogues of (\ref{eq:33}) for all $\vec \chi \in \Vh$
and 
for $j=1,\ldots,J$ 
\begin{subequations} \label{eq:45}
\begin{align}
\left[\deldel{\vec X^h}\, g^{\beta}(\vec X^h) \right](\vec\chi)
& = \beta\,
g^{\beta-1}(\vec X^h)\,\vec\chi\,.\,\nabla\,g(\vec X^h)
\nonumber \\ &=
\beta\,g^{\beta}(\vec X^h)\,\vec\chi\,.\,\nabla\,\ln\,g(\vec X^h) 
\quad \mbox{on }I_j,
\quad \forall \ \beta \in {\mathbb R} \,,
\label{eq:ddgbetah} \\
\left[\deldel{\vec X^h}\, \nabla\,\ln\,g(\vec X^h)
\right](\vec\chi) & =
(D^2\,\ln\,g(\vec X^h))\,\vec\chi
\quad \mbox{on }I_j\,, \label{eq:ddnablalngh} \\
\left[\deldel{\vec X^h}\, |\vec X^h_\rho|\right](\vec\chi)
& = \frac{\vec X^h_\rho\,.\,\vec\chi_\rho}{|\vec X^h_\rho|}
= \vec\tau^h\,.\,\vec\chi_\rho = 
\vec\tau^h\,.\,\vec\chi_s\, |\vec X^h_\rho|\quad \mbox{on }I_j\,, 
\label{eq:ddxrhoh} \\
\left[\deldel{\vec X^h}\, |\vec X^h_\rho|_g \right](\vec\chi)
& = \left( \vec\tau^h\,.\,\vec\chi_s + 
\tfrac12\,\vec\chi\,.\,\nabla\,\ln\,g(\vec X^h) \right) |\vec X^h_\rho|_g
\quad \mbox{on }I_j\,,
\label{eq:ddxrhogh} \\
\left[\deldel{\vec X^h}\, \vec\tau^h\right](\vec\chi)
& = \left[\deldel{\vec X^h}\,\frac{\vec X^h_\rho}{|\vec X^h_\rho|}
\right](\vec\chi) = \frac{\vec\chi_\rho}{|\vec X^h_\rho|}
- \frac{\vec X^h_\rho}{|\vec X^h_\rho|^2}\,
 \frac{\vec X^h_\rho\,.\,\vec\chi_\rho}{|\vec X^h_\rho|} \nonumber \\
&= \vec\chi_s - \vec\tau^h\,(\vec\chi_s\,.\,\vec\tau^h)
= (\vec\chi_s\,.\,\vec\nu^h)\,\vec\nu^h 
\quad \mbox{on }I_j\,, \label{eq:ddtauh} \\
\left[\deldel{\vec X^h}\, \vec\nu^h\right](\vec\chi)
& = - \left[\deldel{\vec X^h}\, (\vec \tau^h)^\perp\right](\vec\chi) 
= - (\vec\chi_s\,.\,\vec\nu^h)\,\vec\tau^h \quad \mbox{on }I_j\,, \label{eq:ddnuh} \\
\left[\deldel{\vec X^h}\,\vec \nu^h\,|\vec X^h_\rho|\right](\vec\chi) &=
- \left[\deldel{\vec X^h}\,(\vec X^h_{\rho})^\perp\right](\vec\chi)
= - \vec \chi_\rho^\perp = - \vec \chi_s^\perp \,|\vec X^h_\rho|
\quad \mbox{on }I_j
\,. \label{eq:ddnuxrhoh}
\end{align}
\end{subequations}

\subsection{Based on $\kappa^h$} \label{sec:41}

In the following we will discuss the $L^2$--gradient flow of the energy 
\begin{equation} \label{eq:Wglambdah}
W^h_{g,\lambda}(\vec X^h,\revised{\kappa^h}) 
= \tfrac12\left( g^{-\frac12}(\vec X^h)
\left(\kappa^h - \tfrac12\,\frac{\vec\omega^h}{|\vec\omega^h|}\,.\,
\nabla\,\ln g(\vec X^h) \right)^2 
+ 2\,\lambda\,g^\frac12(\vec X^h),|\vec X^h_\rho| \right)^h\,,
\end{equation}
subject to the side constraint 
\begin{equation} \label{eq:sideh}
\left( \kappa^h\,\vec\nu^h, \vec\eta\, |\vec X^h_\rho| \right)^h
+ \left( \vec X^h_s, \vec\eta_s \,|\vec X^h_\rho| \right) = 0
\quad \forall\ \vec\eta \in \Vh\,.
\end{equation}
On recalling (\ref{eq:omegah}), we see that (\ref{eq:Wglambdah}) 
and (\ref{eq:sideh})
are discrete analogues of (\ref{eq:Wglambda}) and (\ref{eq:sdwfb}), 
respectively. 
We define the Lagrangian
\begin{align}
\mathcal{L}^h(\vec X^h, \kappa^h, \vec Y^h) & = 
\tfrac12\left( g^{-\frac12}(\vec X^h)
\left(\kappa^h - \tfrac12\,\frac{\vec\omega^h}{|\vec\omega^h|}\,.\,
\nabla\,\ln g(\vec X^h) \right)^2 
+ 2\,\lambda \,g^\frac12(\vec X^h), |\vec X^h_\rho| \right)^h 
\nonumber \\ & \quad
- \left( \kappa^h\,\vec\nu^h, \vec Y^h\,|\vec X^h_\rho| \right)^h 
- \left( \vec X^h_s,\vec Y^h_s\,|\vec X^h_\rho|\right) , \label{eq:Lagh}
\end{align}
which is the corresponding discrete analogue of (\ref{eq:Lag}).

In addition to (\ref{eq:45}), 
we will require $\left[\deldel{\vec X^h}\,\pi^h\left[\frac{\vec \omega^h}
{|\vec \omega^h|}\right] \right](\vec\chi)$ in order to compute variations of
(\ref{eq:Lagh}). 
We establish this along the lines of \cite[(3.2a,b)--(3.7)]{pwf}.
To this end, we introduce the following operators.
On recalling (\ref{eq:ip0}) and (\ref{eq:tauh}),
let $\Ds, \,\mDs : V^h \rightarrow V^h$ be such that
for any $t \in [0,T]$
\begin{subequations} \label{eq:vDs}
\begin{align}
(\Ds \,\eta)(q_j) &= \frac{|\vec{X}^h(q_j,t)-\vec{X}^h(q_{j-1},t)|\,\eta_s(q_j^-) +  
|\vec{X}^h(q_{j+1},t)-\vec{X}^h(q_{j},t)|\,\eta_s(q_j^+)}
{|\vec{X}^h(q_j,t)-\vec{X}^h(q_{j-1},t)| + |\vec{X}^h(q_{j+1},t)-\vec{X}^h(q_{j},t)|}
\nonumber \\
& =\frac{\eta(q_{j+1}) - \eta(q_{j-1})}{|\vec{X}^h(q_j,t)-\vec{X}^h(q_{j-1},t)|
+ |\vec{X}^h(q_{j+1},t)-\vec{X}^h(q_{j},t)|}\,,
\qquad j=1, \ldots ,J\,,
\label{eq:vDse}
\\
(\mDs \,\eta)(q_j) &=\frac{(\Ds \,\eta)(q_j)}{|(\Ds \,\vec X^h(t))(q_j)|}
=\frac{\eta(q_{j+1}) - \eta(q_{j-1})}{|\vec X^h(q_{j+1},t) - \vec X^h(q_{j-1},t)|}
\,,
\qquad j=1,\ldots, J\,,
\label{eq:vDshe}
\end{align}
\end{subequations}
where $q_{J+1}=q_1$. 
Here, we make the following
natural assumption
\begin{align*}
(\Ass^h)&\hspace{0.5in} 
\vec{X}^h(q_j,t) \ne \vec{X}^h(q_{j+1},t) 
\qquad \mbox{and} \qquad \vec{X}^h(q_{j-1},t) \ne \vec{X}^h(q_{j+1},t),\\ 
&\hspace{7cm} j=1,\ldots, J,\qquad    
\mbox{for all } t\in [0,T]\,.
\end{align*}
Hence (\ref{eq:vDs}) is well-defined. 
As usual, $\Ds,\,\mDs : \Vh \rightarrow \Vh$ are defined component-wise.

It follows from
(\ref{eq:omegah}), (\ref{eq:tauh}) and (\ref{eq:vDse}) that,
for all $ \vec\varphi\in\Vh$,
\begin{equation} \label{eq:omegaDSh}
\left(\vec\omega^h, \vec\varphi \, |\vec X^h_\rho| \right)^h 
= -\left( (\vec \tau^h)^\perp, \vec\varphi \, |\vec X^h_\rho| \right)^h
=-\left( (\vec X^h_\rho)^\perp, \vec\varphi \right)^h 
= -\left( (\Ds\, \vec X^h)^\perp, \vec\varphi \, |\vec X^h_\rho| \right)^h.
\end{equation}
Therefore, we have from (\ref{eq:omegaDSh}),
$(\Ass^h)$ and (\ref{eq:vDshe}) that
\begin{equation} 
\vec \omega^h = -(\Ds\, \vec X^h)^\perp \qquad \mbox{and} \qquad
\pi^h\left[\frac{\vec \omega^h}{|\vec \omega^h|}\right] 
= -(\mDs\, \vec X^h)^\perp \,.
\label{eq:omegadh}
\end{equation} 
Then it is a simple matter to compute, for any $\vec\chi\in\Vh$, 
\begin{align*}
\left[\deldel{\vec X^h}\,\mDs\, \vec X^h \right] (\vec\chi)
&= \vec\pi^h \left[ \left[\mat \Id 
- (\mDs \,\vec X^h) \otimes (\mDs \,\vec X^h)\right]
\left( \mDs \,\vec \chi\right) \right] \nonumber \\ & 
= \vec\pi^h \left[ |\vec\omega^h|^{-2}\left(
(\mDs \,\vec \chi)\,.\,\vec\omega^h\right)\vec\omega^h \right],
\end{align*}
so that
\begin{equation}
\left[\deldel{\vec X^h}\,\vec\pi^h\frac{\vec \omega^h}{|\vec \omega^h|}
\right](\vec\chi)
= - \left(\left[\deldel{\vec X^h}\,\mDs\, \vec X^h \right]
(\vec\chi)\right)^\perp
= - \vec\pi^h \left[ |\vec\omega^h|^{-2}\left(
(\mDs \,\vec \chi)\,.\,\vec\omega^h\right) (\vec\omega^h)^\perp \right] .
\label{eq:mDsvarom}
\end{equation}

Similarly to (\ref{eq:omegaDSh}), we have for any $\vec\eta \in \Vh$ that
\begin{equation} \label{eq:NItang}
\left( \vec \eta_s, \vec \varphi \,|\vec X^h_\rho|\right)^h 
= \left( \Ds \,\vec \eta, \vec \varphi \,|\vec X^h_\rho|\right)^h \qquad \forall\ 
\vec\varphi \in
\Vh\,,
\end{equation}
Hence, it follows from (\ref{eq:NItang}), (\ref{eq:vDshe}) and (\ref{eq:omegadh}) that 
\begin{equation} \label{eq:NIttang}
\left( |\vec{\omega}^h|^{-1}\,\vec{\eta}_s, \vec{\varphi} \,|\vec X^h_\rho|\right)^h 
= \left( \mDs \,\vec \eta , \vec \varphi  \,|\vec X^h_\rho|\right)^h \qquad 
\forall\ \vec \eta,
\,\vec\varphi \in
\Vh\,.
\end{equation}
Therefore, combining (\ref{eq:mDsvarom})
and (\ref{eq:NIttang}) yields for any $\vec \varphi,\,\vec \chi \in \Vh$ that
\begin{align}
&\left(\vec \varphi, \left[\deldel{\vec X^h}\,\frac{\vec \omega^h}{|\vec \omega^h|}
 \right](\vec\chi)
\,|\vec X^h_\rho| \right)^h
= -\left( |\vec\omega^h|^{-2}\,\vec \varphi,
\left((\mDs \, \vec \chi)\,.\,\vec\omega^h\right) (\vec\omega^h)^\perp
|\vec X^h_\rho| \right)^h
\nonumber \\ &\
= -\left( |\vec\omega^h|^{-3}\left(\vec \varphi\,.\,(\vec\omega^h)^\perp
\right) \vec\omega^h, \vec \chi_s\,|\vec X^h_\rho| \right)^h
= -\left(|\vec \omega^h|^{-1}\, 
\vec \varphi\,.\,\left(\frac{{\vec \omega^h}}{|\vec \omega^h|}\right)^\perp,
\frac{\vec \omega^h}{|\vec \omega^h|}\,.\,\vec \chi_s\,|\vec X^h_\rho| 
\right)^h.
\label{eq:normomvar}
\end{align}

Taking variations $\chi \in V^h$ in $\kappa^h$ and setting
$\left[\deldel{\kappa^h}\, \mathcal{L}^h\right](\chi) = 0$ leads to
\begin{equation} \label{eq:kappayh}
\left(g^{-\frac12}(\vec X^h)
\left(\kappa^h - \tfrac12\,\frac{\vec\omega^h}{|\vec\omega^h|}\,.\,
\nabla\,\ln g(\vec X^h) \right)
- \vec Y^h\,.\,\vec\nu^h, \chi\,|\vec X^h_\rho| \right)^h = 0
\quad \forall\ \chi \in V^h\,,
\end{equation}
which, on recalling (\ref{eq:omegah}), 
implies the discrete analogue of (\ref{eq:kappaid})
\begin{align} \label{eq:kappaidh}
\pi^h \left[\vec Y^h\,.\,\vec\omega^h \right]& = \pi^h \left[ g^{-\frac12}(\vec X^h)
\left(\kappa^h - \tfrac12\,\frac{\vec\omega^h}{|\vec\omega^h|}\,.\,
\nabla\,\ln g(\vec X^h) \right) 
\right]
\nonumber \\ 
\quad\iff\quad
\kappa^h & = \pi^h \left[ 
g^{\frac12}(\vec X^h)\,\vec Y^h\,.\,\vec\omega^h + \tfrac12\,
\frac{\vec\omega^h}{|\vec\omega^h|}\,.\,\nabla\,\ln g(\vec X^h) \right] .
\end{align}
Taking variations $\vec\eta \in \Vh$ in $\vec Y^h$, and setting
$\left[\deldel{\vec Y^h}\, \mathcal{L}^h\right](\vec\eta) = 0$
we obtain (\ref{eq:sideh}).
Setting $\left(g^\frac32(\vec X^h)\,\vec X^h_t\,.\,\vec\omega^h, 
\vec\chi\,.\,\vec\omega^h\,|\vec X^h_\rho| \right)^h 
= - \left[\deldel{\vec X^h}\, \mathcal{L}^h\right](\vec\chi)$,
for variations $\vec\chi \in \Vh$ in $\vec X^h$ yields, as a discrete analogue
to (\ref{eq:dLdx}),
\begin{align}
&\left(g^\frac32(\vec X^h)\,\vec X^h_t\,.\,\vec\omega^h, 
\vec\chi\,.\,\vec\omega^h\,
|\vec X^h_\rho| \right)^h \nonumber \\ & \
= -\tfrac12\left(
\left(\kappa^h - \tfrac12\,\frac{\vec\omega^h}{|\vec\omega^h|}\,.\,
\nabla\,\ln g(\vec X^h) \right)^2 , 
\left[\deldel{\vec X^h}\,g^{-\frac12}(\vec X^h)\,|\vec X^h_\rho|\right] (\vec\chi) 
\right)^h \nonumber \\ & \quad
- \lambda \left( 1, 
\left[\deldel{\vec X^h}\,g^{\frac12}(\vec X^h)\,|\vec X^h_\rho|\right] (\vec\chi) 
\right)^h \nonumber \\ & \quad  
+ \tfrac12 
\left( g^{-\frac12}(\vec X^h)
\left(\kappa^h - \tfrac12\,\frac{\vec\omega^h}{|\vec\omega^h|}\,.\,
\nabla\,\ln g(\vec X^h) \right), 
\left[\deldel{\vec X^h}\, \frac{\vec\omega^h}{|\vec\omega^h|}\,.\,
\nabla\,\ln\,g(\vec X^h) \right] 
(\vec\chi) \,|\vec X^h_\rho| \right)^h \nonumber \\ & \quad  
+ \left( \kappa^h\,\vec Y^h, 
\left[\deldel{\vec X^h}\,\vec \nu^h\,|\vec X^h_\rho| \right] (\vec\chi)
\right)^h
+ \left( \vec Y^h_\rho,  \left[\deldel{\vec X^h}\, \vec\tau^h \right] 
(\vec\chi) \right) .
\label{eq:dLdxh}
\end{align}
Choosing $\vec\chi = \vec X^h_t$ in (\ref{eq:dLdxh}),
where we observe a discrete variant of (\ref{eq:ddA}), yields that
\begin{align}
&\left(g^\frac32(\vec X^h)\,(\vec X^h_t\,.\,\vec\omega^h)^2, 
|\vec X^h_\rho| \right)^h \nonumber \\ & \
= -\tfrac12\left(
\left(\kappa^h - \tfrac12\,\frac{\vec\omega^h}{|\vec\omega^h|}\,.\,
\nabla\,\ln g(\vec X^h) \right)^2 , 
\left[g^{-\frac12}(\vec X^h)\,|\vec X^h_\rho|\right]_t \right)^h 
\nonumber \\ & \quad
- \lambda \left( 1, 
\left[g^{\frac12}(\vec X^h)\,|\vec X^h_\rho|\right]_t \right)^h 
\nonumber \\ & \quad  
+ \tfrac12 \left( g^{-\frac12}(\vec X^h)
\left(\kappa^h - \tfrac12\,\frac{\vec\omega^h}{|\vec\omega^h|}\,.\,
\nabla\,\ln g(\vec X^h) \right), 
\left[\frac{\vec\omega^h}{|\vec\omega^h|}\,.\,
\nabla\,\ln\,g(\vec X^h) \right]_t |\vec X^h_\rho| \right)^h 
\nonumber \\ & \quad  
+ \left( \kappa^h\,\vec Y^h, 
\left[\vec \nu^h\,|\vec X^h_\rho| \right]_t \right)^h
+ \left( \vec Y^h_\rho,  \vec\tau^h_t \right) .
\label{eq:dLdxxth}
\end{align}

Differentiating (\ref{eq:sideh}) with respect to time, and then choosing
$\vec\eta = \vec Y^h$ yields that
\begin{equation}
\left( \kappa^h_t, \vec Y^h\,.\,\vec\nu^h\,|\vec X^h_\rho| \right)^h
+ \left( \kappa^h\,\vec Y^h, (\vec\nu^h\,|\vec X^h_\rho|)_t \right)^h 
+ \left(\vec\tau^h_t,\vec Y^h_\rho \right) = 0 \,.
\label{eq:dtsideh}
\end{equation}
Combining (\ref{eq:dLdxxth}), (\ref{eq:dtsideh}) and (\ref{eq:kappayh}) 
with $\chi = \kappa^h_t$ gives, on noting (\ref{eq:Wglambdah}), that
\begin{align} \label{eq:dLstabh}
&
\left(g^\frac32(\vec X^h)\,(\vec X^h_t\,.\,\vec\omega)^2, |\vec X^h_\rho| 
\right)^h = 
-\tfrac12\left(
\left(\kappa^h - \tfrac12\,\frac{\vec\omega}{|\vec\omega|}
\,.\,\nabla\,\ln g(\vec X^h) \right)^2 , 
\left[g^{-\frac12}(\vec X^h)\,|\vec X^h_\rho|\right]_t 
\right)^h \nonumber \\ & \quad
+ \tfrac12 
\left( g^{-\frac12}(\vec X^h)
\left(\kappa^h - \tfrac12\,\frac{\vec\omega}{|\vec\omega|}
\,.\,\nabla\,\ln g(\vec X^h) \right), 
\left[\frac{\vec\omega}{|\vec\omega|}\,.\,\nabla\,\ln\,g(\vec X^h) \right]_t
|\vec X^h_\rho| \right)^h \nonumber \\ & \quad
- \left( \kappa^h_t, 
g^{-\frac12}(\vec X^h)
\left(\kappa^h - \tfrac12\,\frac{\vec\omega}{|\vec\omega|}
\,.\,\nabla\,\ln g(\vec X^h) \right) |\vec X^h_\rho| \right)^h
- \lambda \left( 1, 
\left[g^{\frac12}(\vec X^h)\,|\vec X^h_\rho|\right]_t \right)^h
\nonumber \\ & \quad
= - \ddt\,W_{g,\lambda}^h(\vec X^h,\vec\kappa^h)\,.
\end{align}

In order to derive a suitable approximation of $(\BGNpwf)$, we now return to
(\ref{eq:dLdxh}). 
Combining (\ref{eq:dLdxh}), (\ref{eq:45}) 
and (\ref{eq:normomvar}), on noting (\ref{eq:abperp}), yields
\begin{align}
& \left(g^\frac32(\vec X^h)\,\vec X^h_t\,.\,\vec\omega^h, 
\vec\chi\,.\,\vec\omega^h\, |\vec X^h_\rho| \right)^h
= 
\left(\vec Y^h_s\,.\,\vec\nu^h, \vec\chi_s\,.\,\vec\nu^h\,|\vec X^h_\rho| 
\right)
\nonumber \\ & \
-\tfrac12 \left( g^{-\frac12}(\vec X^h)
\left[\kappa^h - \tfrac12\,\frac{\vec\omega^h}{|\vec\omega^h|}\,.\,
\nabla\,\ln g(\vec X^h) \right]^2
+ 2\,\lambda\,g^\frac12(\vec X^h) ,
\vec\chi_s\,.\,\vec\tau^h\,|\vec X^h_\rho| \right)^h \nonumber \\ & \
+ \tfrac14 \left( 
g^{-\frac12}(\vec X^h)
\left[\kappa^h - \tfrac12\,\frac{\vec\omega^h}{|\vec\omega^h|}\,.\,
\nabla\,\ln g(\vec X^h) \right]^2
- 2\,\lambda\,g^\frac12(\vec X^h),
\vec\chi\,.\,(\nabla\,\ln\,g(\vec X^h))\,|\vec X^h_\rho| \right)^h
\nonumber \\ & \
+ \tfrac12 \left( 
g^{-\frac12}(\vec X^h)
\left[\kappa^h - \tfrac12\,\frac{\vec\omega^h}{|\vec\omega^h|}\,.\,
\nabla\,\ln g(\vec X^h) \right]
\frac{\vec\omega^h}{|\vec\omega^h|},
(D^2\,\ln\,g(\vec X^h))\,
\vec\chi\,|\vec X^h_\rho| \right)^h 
\nonumber \\ & \
-\tfrac12\left( 
g^{-\frac12}(\vec X^h)
\left[\kappa^h - \tfrac12\,\frac{\vec\omega^h}{|\vec\omega^h|}\,.\,
\nabla\,\ln g(\vec X^h) \right]
\frac{\nabla \ln\,g(\vec X^h)}{|\vec\omega^h|}
\,.\,\left(\frac{\vec\omega^h}{|\vec \omega^h|}\right)^\perp, 
\frac{\vec\omega^h}{|\vec \omega^h|}
\,.\,\vec\chi_s
\,|\vec X^h_\rho| \right)^h 
\nonumber \\ & \quad
+ \left( \kappa^h\, (\vec Y^h)^\perp,
\vec\chi_s\,|\vec X^h_\rho| \right)^h
\quad \forall \ \vec \chi \in \Vh\,,
\label{eq:dLdxflow2h}
\end{align}
which is the discrete analogue of (\ref{eq:dLdxflow2}), on noting that 
$\vec \nu^\perp =\vec \tau$.

Hence we obtain the following approximation of $(\BGNpwf)$.
\\ \noindent
$(\BGNpwf_h)^h$:
Let $\vec X^h(0) \in \Vh$. For $t \in (0,T]$
find $(\vec X^h(t),\kappa^h(t),\vec Y^h(t))\in \Vh \times V^h\times\Vh$ 
such that (\ref{eq:dLdxflow2h}), (\ref{eq:kappayh}) and (\ref{eq:sideh}) hold.

We note that in the Euclidean case it follows from (\ref{eq:kappaidh}) 
that $\kappa^h = \pi^h\,[\vec Y^h\,.\,\vec\omega^h]$, and so on eliminating
$\kappa^h$, and on noting (\ref{eq:omegah}), 
the approximation $(\BGNpwf_h)^h$ 
collapses to the isotropic closed curve version of (3.36a,b), with $\beta=0$,
in \cite{pwf}.

\begin{theorem} \label{thm:stab}
Let the assumption $(\Ass^h)$ be satisfied and 
let $(\vec X^h(t),\vec Y^h(t)) \in \Vh\times \Vh$, for $t\in (0,T]$,
be a solution to $(\BGNpwf_h)^h$.
Then the solution satisfies the stability bound \eqref{eq:dLstabh}. 
\end{theorem}
\begin{proof}
The proof is given in (\ref{eq:dLdxxth}), (\ref{eq:dtsideh}) and 
(\ref{eq:dLstabh}).
\end{proof}

\begin{remark}\label{rem:why}
We note why we choose $\frac{\vec \omega^h}{|\vec \omega^h|}$ in 
\eqref{eq:Wglambdah} as opposed to $\vec \nu^h$ or $\vec \omega^h$.
In the case of $\vec \nu^h$, \eqref{eq:kappaidh} 
and \eqref{eq:dLdxh} still hold 
with $\frac{\vec\omega^h}{|\vec\omega^h|}$ replaced by $\vec\omega^h$ and 
$\vec\nu^h$, respectively.
However, then the elimination of $\kappa^h$ from the modified 
\eqref{eq:dLdxh}
via the modified \eqref{eq:kappaidh} now leads to a far more complicated 
version of \eqref{eq:dLdxflow2h}. 
In the case of $\vec \omega^h$, one needs to compute
$\left[\deldel{\vec X^h}\,\vec\omega^h \right]$ as opposed to 
$\left[\deldel{\vec X^h}\, \frac{\vec\omega^h}{|\vec\omega^h|} \right]$.
However, on noting \eqref{eq:omegadh} and \eqref{eq:vDs}, 
it is easier to compute the latter. Hence, the choice of 
$\frac{\vec \omega^h}{|\vec \omega^h|}$ in \eqref{eq:Wglambdah}.
\end{remark}

\begin{remark}\label{rem:equid}
Due to \eqref{eq:sideh}, the approximation 
$(\BGNpwf_h)^h$ satisfies \revised{an} equidistribution property,
i.e.\ any two neighbouring elements are either parallel or of the same length,
at every $t>0$. For this property to hold, it is crucial to employ mass lumping
in \eqref{eq:sideh}.
We refer to {\rm \cite[Rem.\ 2.4]{triplej}} for more details.
\end{remark}

\subsection{Based on $\kappa^h_g$} \label{sec:42}

Let $(\cdot,\cdot)^\diamond$ denote a discrete $L^2$--inner product based on
some numerical quadrature rule. In particular, 
for two piecewise continuous functions, with possible jumps at the 
nodes $\{q_j\}_{j=1}^J$, we let 
$(u,v)^\diamond = I^\diamond(u\,v)$, where
\begin{equation} \label{eq:Idiamond}
I^\diamond(f) = \sum_{j=1}^J h_j\,\sum_{k = 1}^K w_k\,
f(\alpha_k\,q_{j-1} + (1-\alpha_k)\,q_j)\,,\quad
w_k > 0\,,\ \alpha_k \in [0,1]\,,\quad k = 1,\ldots,K\,,
\end{equation}
with $K\geq2$, $\sum_{k=1}^K w_k = 1$,
and with distinct $\alpha_k$, $k=1,\ldots,K$.
A special case is
$(\cdot,\cdot)^\diamond = (\cdot,\cdot)^h$, recall
(\ref{eq:ip0}), but we also allow for more accurate quadrature rules.

We define the Lagrangian
\begin{align}
\mathcal{L}^h_g(\vec X^h, \kappa^h_g, \vec Y^h_g) & = 
\tfrac12\left((\kappa^h_g)^2 + 2\,\lambda, |\vec X^h_\rho|_g\right)^\diamond
- \left( g^\frac12(\vec X^h)\,\kappa^h_g\,\vec\nu^h, \vec Y^h_g\,|\vec X^h_\rho|_g
\right)^\diamond
\nonumber \\ & \quad
- \left( \vec X^h_s,(\vec Y^h_g)_s\,|\vec X^h_\rho|_g\right)^\diamond
- \tfrac12 \left( \nabla\,\ln\,g(\vec X^h), \vec Y^h_g\,|\vec X^h_\rho|_g
\right)^\diamond ,
\label{eq:B1}
\end{align}
which is the corresponding discrete analogue of (\ref{eq:Lagg}).
Taking variations $\chi \in V^h$ in $\kappa^h_g$ and setting
$\left[\deldel{\kappa^h_g}\, \mathcal{L}^h_g\right](\chi) = 0$
yields that
\begin{equation} \label{eq:B2}
\left(\kappa^h_g - g^\frac12(\vec X^h)\,\vec Y^h_g\,.\,\vec\nu^h, 
\chi\,|\vec X^h_\rho|_g \right)^\diamond = 0
\qquad \forall\ \chi \in V^h\,.
\end{equation}
Taking variations $\vec\eta \in \Vh$ in $\vec Y^h_g$, and setting
$\left[\deldel{\vec Y^h_g}\, \mathcal{L}^h_g\right](\vec\eta) = 0$
we obtain 
\begin{equation} \label{eq:B3}
\left(g^\frac12(\vec X^h)\,\kappa_g^h\,\vec\nu^h, \vec\eta\,|\vec X^h_\rho|_g
\right)^\diamond
+ \left(\vec X^h_s,\vec\eta_s\,|\vec X^h_\rho|_g\right)^\diamond
+ \tfrac12 \left( \nabla\,\ln\,g(\vec X^h), \vec\eta\,|\vec
X^h_\rho|_g\right)^\diamond
= 0\,, 
\end{equation}
for all $\vec\eta \in \Vh$,
as a discrete analogue of (\ref{eq:side}).
Taking variations $\vec\chi \in \Vh$ in $\vec X^h$, 
and then setting
$(g(\vec X^h)\,\vec X^h_t\,.\,\vec\omega^h
, \vec\chi \,.\,\vec\omega^h\,|\vec X^h_\rho|_g )^\diamond
= - \left[\deldel{\vec X^h}\, \mathcal{L}^h_g\right](\vec\chi)$, we obtain
\begin{align}
& \left(g(\vec X^h)\,\vec X^h_t\,.\,\vec\omega^h, 
\vec\chi \,.\,\vec\omega^h \,|\vec X^h_\rho|_g \right)^\diamond
\nonumber \\
& \ = - \tfrac12 \left( (\kappa^h_g)^2 + 2\,\lambda, 
\left[\deldel{\vec X^h}\, |\vec X^h_\rho|_g\right](\vec\chi) \right)^\diamond
+ \left( \kappa^h_g\,\vec Y^h_g, 
\left[\deldel{\vec X^h}\, g(\vec X^h)\, 
\vec\nu^h \,|\vec X^h_\rho|
\right] (\vec\chi) \right)^\diamond
\nonumber \\ & \quad
+ \left( (\vec Y^h_g)_\rho,  \left[\deldel{\vec X^h}\, g^\frac12(\vec X^h)\,
\vec\tau^h \right] (\vec\chi) \right)^\diamond
+ \tfrac12 \left( \vec Y^h_g, 
\left[\deldel{\vec X^h}\,(\nabla\,\ln\,g(\vec X^h))\,
|\vec X^h_\rho|_g \right] (\vec\chi) \right)^\diamond,
\label{eq:B5}
\end{align}
for all $\vec\chi \in \Vh$.
Choosing $\vec\chi = \vec X^h_t$ in (\ref{eq:B5}), and noting 
a discrete variant of (\ref{eq:ddA}), as well as
(\ref{eq:normg}), yields that
\begin{align}
& \left(g(\vec X^h)\,(\vec X^h_t\,.\,\vec\omega^h)^2, 
|\vec X^h_\rho|_g \right)^\diamond \nonumber \\ & \quad 
= - \tfrac12 \left( (\kappa^h_g)^2 + 2\,\lambda, 
(|\vec X^h_\rho|_g)_t \right)^\diamond
+ \left( \kappa^h_g\,\vec Y^h_g, 
(g^\frac12(\vec X^h)\, \vec\nu^h \,|\vec X^h_\rho|_g)_t \right)^\diamond
\nonumber \\ & \qquad
+ \left( (\vec Y^h_g)_\rho,  (g^\frac12(\vec X^h)\,\vec\tau^h)_t 
\right)^\diamond
+ \tfrac12 \left( \vec Y^h_g, 
((\nabla\,\ln\,g(\vec X^h))\,|\vec X^h_\rho|_g )_t \right)^\diamond .
\label{eq:B5Xt}
\end{align}

On differentiating (\ref{eq:B3}) with respect to time, and then choosing
$\vec\eta = \vec Y^h_g$, we obtain, on recalling (\ref{eq:tauh}) and
(\ref{eq:normg}), that
\begin{align} \label{eq:B3dt}
& 
\left((\kappa_g^h)_t\,\vec Y^h_g, g^\frac12(\vec X^h)\,\vec\nu^h\, 
|\vec X^h_\rho|_g\right)^\diamond
+\left(\kappa_g^h\,\vec Y^h_g, (g^\frac12(\vec X^h)\,\vec\nu^h\, 
|\vec X^h_\rho|_g)_t \right)^\diamond
\nonumber \\ & \qquad
+ \left((\vec Y^h_g)_\rho,(g^\frac12(\vec X^h)\,\vec\tau^h)_t\right)^\diamond
+ \tfrac12 \left(\vec Y^h_g, ((\nabla\,\ln\,g(\vec X^h))\, 
|\vec X^h_\rho|_g)_t \right)^\diamond
= 0 .
\end{align}
Choosing $\chi = (\kappa^h_g)_t$ in (\ref{eq:B2}), and combining with
(\ref{eq:B5Xt}) and (\ref{eq:B3dt}), yields that
\begin{equation} \label{eq:Qhstab}
 \left(g(\vec X^h)\,(\vec X^h_t\,.\,\vec\omega^h)^2, 
|\vec X^h_\rho|_g \right)^\diamond 
+\tfrac12\,\ddt
\left((\kappa^h_g)^2 + 2\,\lambda, |\vec X^h_\rho|_g\right)^\diamond
= 0\,,
\end{equation}
which reveals the discrete gradient flow structure. Also note that
(\ref{eq:B5Xt})--(\ref{eq:Qhstab}) are the discrete analogues of
(\ref{eq:b5xt})--(\ref{eq:dLgstab}).

In order to derive a suitable finite element approximation, we now return to
(\ref{eq:B5}). 
Substituting (\ref{eq:45}) 
into (\ref{eq:B5}) yields, 
on noting (\ref{eq:tauh}) and (\ref{eq:normg}), that
\begin{align}
& \left(g(\vec X^h)\,\vec X^h_t\,.\,\vec\omega^h, 
\vec\chi \,.\,\vec\omega^h\,|\vec X^h_\rho|_g \right)^\diamond
= 
 \tfrac12 \left( \vec Y^h_g, 
\left[\deldel{\vec X^h}\,(\nabla\,\ln\,g(\vec X^h)) \right] (\vec\chi)\,
|\vec X^h_\rho|_g \right)^\diamond
\nonumber \\ & \
- \tfrac12 \left( (\kappa^h_g)^2 + 2\,\lambda
-\vec Y^h_g \,.\,\nabla\,\ln\,g(\vec X^h), 
\left[\deldel{\vec X^h}\, |\vec X^h_\rho|_g\right](\vec\chi) \right)^\diamond
\nonumber \\ & \
+ \left( \kappa^h_g\,\vec Y^h_g\,.\,\vec\nu^h
\left[\deldel{\vec X^h}\, g(\vec X^h)\right] (\vec\chi)
\,|\vec X^h_\rho| \right)^\diamond
+ \left( g(\vec X^h)\,\kappa^h_g\,\vec Y^h_g, \left[\deldel{\vec X^h}\,  
\vec\nu^h\,|\vec X^h_\rho|\right] (\vec\chi) \right)^\diamond
\nonumber \\ & \
+ \left( (\vec Y^h_g)_\rho\,.\,\vec\tau^h ,  
\left[\deldel{\vec X^h}\, g^\frac12(\vec X^h) \right] (\vec\chi)
\right)^\diamond
+ \left( g^\frac12(\vec X^h)\,(\vec Y^h_g)_\rho, 
\left[\deldel{\vec X^h}\, \vec\tau^h 
\right] (\vec\chi) \right)^\diamond \nonumber \\ &
 = -\tfrac12 \left( (\kappa^h_g\,)^2 + 2\,\lambda
- \vec Y^h_g\,.\,\nabla\,\ln\,g(\vec X^h), \left[\vec\tau^h\,.\,\vec\chi_s + 
\tfrac12\,\vec\chi\,.\,\nabla\,\ln\,g(\vec X^h)\right] |\vec X^h_\rho|_g
 \right)^\diamond \nonumber \\ & \
+ \tfrac12 \left((D^2\,\ln\,g(\vec X^h))\,\vec Y^h_g, \vec\chi\,
|\vec X^h_\rho|_g \right)^\diamond \nonumber \\ & \
+  \left( g^\frac12(\vec X^h)\,\kappa^h_g\,\vec Y^h_g\,.\,\vec\nu^h
+ \tfrac12\,(\vec Y^h_g)_s\,.\,\vec\tau^h, 
(\nabla\,\ln\,g(\vec X^h))\,.\,\vec\chi\,
|\vec X^h_\rho|_g \right)^\diamond \nonumber \\ & \
- \left( g^\frac12(\vec X^h)\,\kappa^h_g\,\vec Y^h_g,
\vec\chi_s^\perp\,|\vec X^h_\rho|_g \right)^\diamond
+ \left( (\vec Y^h_g)_s\,.\,\vec\nu^h, 
\vec\chi_s\,.\,\vec\nu^h\,|\vec X^h_\rho|_g  \right)^\diamond
\qquad \forall\ \vec\chi \in \Vh\,.
\label{eq:B6}
\end{align}
Then (\ref{eq:B6}), (\ref{eq:B2}) and (\ref{eq:B3}), on recalling
(\ref{eq:abperp}), give rise to the following approximation of 
$(\BGNpwfwf)$.

$(\BGNpwfwf_h)^\diamond$:
Let $\vec X^h(0) \in \Vh$. For $t \in (0,T]$
find $(\vec X^h(t),\kappa^h_g(t),\vec Y_g^h(t)) \in \Vh\times V^h \times \Vh$ 
such that
\begin{align} 
& \left(g(\vec X^h)\,\vec X^h_t\,.\,\vec\omega^h, 
\vec\chi \,.\,\vec\omega^h\,|\vec X^h_\rho|_g \right)^\diamond
- \left( (\vec Y^h_g)_s\,.\,\vec\nu^h, 
\vec\chi_s\,.\,\vec\nu^h\,|\vec X^h_\rho|_g  \right)^\diamond
 \nonumber \\ 
& \quad
 = -\tfrac12 \left( (\kappa^h_g\,)^2 + 2\,\lambda
- \vec Y^h_g\,.\,\nabla\,\ln\,g(\vec X^h), \left[\vec\chi_s\,.\,\vec\tau^h + 
\tfrac12\,\vec\chi\,.\,\nabla\,\ln\,g(\vec X^h)\right] |\vec X^h_\rho|_g
 \right)^\diamond \nonumber \\ & \qquad
+ \tfrac12 \left((D^2\,\ln\,g(\vec X^h))\,\vec Y^h_g, \vec\chi\,
|\vec X^h_\rho|_g \right)^\diamond 
+ \left( g^\frac12(\vec X^h)\,\kappa^h_g,
\vec\chi_s\,.\,(\vec Y^h_g)^\perp\,|\vec X^h_\rho|_g \right)^\diamond
\nonumber \\ & \qquad
+ \left( g^\frac12(\vec X^h)\,\kappa^h_g\,\vec Y^h_g\,.\,\vec\nu^h
+ \tfrac12\,(\vec Y^h_g)_s\,.\,\vec\tau^h, \vec\chi\,.\,
(\nabla\,\ln\,g(\vec X^h))\,
|\vec X^h_\rho|_g \right)^\diamond 
\quad \forall\ \vec\chi \in \Vh\,, \label{eq:B7a} 
\end{align}
(\ref{eq:B2}) and (\ref{eq:B3}) hold. 

\begin{theorem} \label{thm:stabg}
Let $|\vec X^h_\rho| > 0$ almost everywhere in $I \times (0,T)$.
Let $(\vec X^h(t),\kappa^h_g(t),$ $\vec Y_g^h(t)) 
\in \Vh\times V^h \times \Vh$, for $t\in (0,T]$,
be a solution to $(\BGNpwfwf_h)^\diamond$. Then the solution satisfies the
stability bound {\rm (\ref{eq:Qhstab})}. 
\end{theorem}
\begin{proof}
We have already shown that a solution to $(\BGNpwfwf_h)^\diamond$ satisfies
(\ref{eq:B5Xt}) and (\ref{eq:B3dt}). Hence choosing $\chi = (\kappa_g^h)_t$
in (\ref{eq:B2}), and combining with (\ref{eq:B5Xt}) and (\ref{eq:B3dt}),
yields (\ref{eq:Qhstab}) as before.
\end{proof}

\begin{remark} \label{rem:noequid}
We stress that unlike for $(\BGNpwf_h)^h$, recall {\rm Remark~\ref{rem:equid}},
it is not possible to prove an equidistribution property for
$(\BGNpwfwf_h)^\diamond$, even if we employ mass lumping in 
\eqref{eq:B3}.
It is for this reason that we also consider higher order quadrature rules.
The motivation behind considering $(\BGNpwfwf_h)^\diamond$ as an alternative
to $(\BGNpwf_h)^h$ is twofold. Firstly, from a variational point of view, it
is more natural to work with $\varkappa_g$ as a variable, since
\eqref{eq:Wglambda0} is naturally defined in terms of $\varkappa_g$. 
Secondly, the techniques introduced for $(\BGNpwfwf_h)^\diamond$
will be exploited in \cite{axipwf} for stable approximations of 
Willmore flow for axisymmetric hypersurfaces in $\bR^3$.
\end{remark}

\setcounter{equation}{0}
\section{Fully discrete finite element approximations} \label{sec:fd}

Let $0= t_0 < t_1 < \ldots < t_{M-1} < t_M = T$ be a
partitioning of $[0,T]$ into possibly variable time steps 
$\ttau_m = t_{m+1} - t_{m}$, $m=0\to M-1$. 
We set $\ttau = \max_{m=0\to M-1}\ttau_m$.
For a given $\vec{X}^m\in \Vh$ we set
$\vec\nu^m = - \frac{[\vec X^m_\rho]^\perp}{|\vec X^m_\rho|}$, as the discrete
analogue to (\ref{eq:tau}). We also let
$\vec\omega^m \in \Vh$ be the natural fully discrete analogue of
$\vec\omega^h \in \Vh$, recall (\ref{eq:omegah}).
Given $\vec{X}^m\in \Vh$, the fully discrete
approximations we propose in this section will always seek a 
parameterization $\vec{X}^{m+1}\in \Vh$ at the new time level, together with a
suitable approximation of curvature. 

For the metrics we consider in this paper, we summarize in 
Table~\ref{tab:g} the quantities that are necessary in order to implement the
numerical schemes presented below.
\begin{table}
\center
\def\arraystretch{1.5}%
\caption{Expressions for terms that are relevant for the implementation of the
presented finite element approximations.}
\begin{tabular}{|c|c|c|}
\hline
$g$ & 
$\nabla\,\ln g(\vec x)$ & 
$D^2\,\ln g(\vec x)$ \\
\hline 
(\ref{eq:gmu}) & 
$- 2\,\mu\,(\vec x\,.\,\vec\ek_2)^{-1}\,\vec\ek_2$ &
$2\,\mu\,(\vec x\,.\,\vec\ek_2)^{-2}\,\vec\ek_2 \otimes \vec \ek_2$
 \\ 
(\ref{eq:galpha}) & 
$\frac{4\,\alpha}{1 - \alpha\,|\vec x|^2}\,\vec x$ & 
$\frac{4\,\alpha}{1 - \alpha\,|\vec x|^2}\,\mat \Id
+ \frac{8\,\alpha^2}{(1 - \alpha\,|\vec x|^2)^2}\,\vec x \otimes \vec x$
 \\ 
(\ref{eq:gMercator}) & $-2\,\tanh(\vec x\,.\,\vec\ek_1)\,\vec \ek_1$
& $-2\,\cosh^{-2}(\vec x\,.\,\vec\ek_1)\,\vec \ek_1 \otimes \vec \ek_1$
\\
(\ref{eq:gcatenoid}) & $2\,\tanh(\vec x\,.\,\vec\ek_1)\,\vec \ek_1$
& $2\,\cosh^{-2}(\vec x\,.\,\vec\ek_1)\,\vec \ek_1 \otimes \vec \ek_1$
\\
(\ref{eq:gtorus}) & $-2\,\frac{\sin(\vec x\,.\,\vec\ek_2)}
{[\mathfrak s^2 + 1]^\frac12 - \cos (\vec x\,.\,\vec\ek_2)}\,\vec\ek_2$ &
$2\,\frac{1 - [\mathfrak s^2 + 1]^\frac12\,\cos (\vec x\,.\,\vec\ek_2)}
{([\mathfrak s^2 + 1]^\frac12 - \cos (\vec x\,.\,\vec\ek_2))^2}
\,\vec\ek_2 \otimes \vec \ek_2$
\\
\hline
\end{tabular}
\label{tab:g}
\end{table}%

\subsection{Based on $\kappa^{m+1}$} \label{sec:51}

We propose the following fully discrete approximation of
$(\BGNpwf_h)^h$.
\\ \noindent
$(\BGNpwf_m)^h$:
Let $(\vec X^0,\kappa^0,\vec Y^0) \in \Vh \times V^h \times \Vh$. 
For $m=0,\ldots,M-1$, we define
$\kappa^m_g = \pi^h\left[ g^{-\frac12}(\vec X^m)\,[\kappa^m - \tfrac12\,
\frac{\vec\omega^m}{|\vec\omega^m|}\,.\,\nabla\,\ln g(\vec X^m)]\right]$, 
and then 
find $(\vec X^{m+1},\kappa^{m+1}, \vec Y^{m+1})\in$ \linebreak
$\Vh \times V^h \times \Vh$ such that
\begin{subequations} \label{eq:fdpwf}
\begin{align} 
& \left(g^{\frac{3}{2}}(\vec X^m)\,\frac{\vec X^{m+1} - \vec X^m}{\ttau_m}\,.\,
\vec\omega^m, \vec\chi\,.\,\vec\omega^m\,
|\vec X^m_\rho| \right)^h
- \left(\vec Y^{m+1}_s, \vec\chi_s\,|\vec X^m_\rho| \right)
\nonumber \\ 
& \quad
+ \left(\vec Y^m_s\,.\,\vec\tau^m, \vec\chi_s\,.\,\vec\tau^m
\,|\vec X^m_\rho| \right)
= -\tfrac12 \left( g^\frac12(\vec X^m)\left[ 
(\kappa^m_g)^2 +2\,\lambda\right] , \vec\chi_s\,.\,\vec\tau^m
\,|\vec X^m_\rho| \right)^h \nonumber \\ & \qquad
+ \tfrac14 \left( g^\frac12(\vec X^m)\left[ (\kappa^m_g)^2 -2\,\lambda\right],
\vec\chi\,.\,(\nabla\,\ln\,g(\vec X^m))\,|\vec X^m_\rho| \right)^h
\nonumber \\ & \qquad
+ \tfrac12 \left( \kappa^m_g\,\frac{\vec\omega^m}{|\vec\omega^m|},
(D^2\,\ln\,g(\vec X^m))\,
\vec\chi\,|\vec X^m_\rho| \right)^h 
+ \left(\kappa^{m}\,(\vec Y^m)^\perp,\vec\chi_s\,|\vec X^m_\rho| \right)^h
\nonumber \\ & \qquad  
-\tfrac12\left( \frac{\kappa^m_g}{|\vec\omega^m|}\, 
\nabla \ln\,g(\vec X^m)
\,.\,\left(\frac{\vec\omega^m}{|\vec \omega^m|}\right)^\perp, 
\vec\chi_s\,.\,\frac{\vec\omega^m}{|\vec \omega^m|}\,|\vec X^m_\rho| \right)^h 
\qquad \forall\ \vec\chi \in \Vh\,, \label{eq:fdpwfa}\\
&\left( g^{\frac12}(\vec X^m)\,\vec Y^{m+1}\,.\,\vec\omega^m,
\vec\eta\,.\,\vec\omega^m \,|\vec X^m_\rho| \right)^h 
+ \tfrac12 \left( 
\frac{\vec\omega^m}{|\vec\omega^m|}
\,.\,\nabla\,\ln g(\vec X^m),\vec\eta\,.\,\vec\omega^m
\,|\vec X^m_\rho| \right)^h 
\nonumber \\ & \hspace{5cm}
+ \left( \vec X^{m+1}_s,\vec\eta_s\,|\vec X^m_\rho|\right) = 0
\qquad \forall\ \vec\eta \in \Vh \label{eq:fdpwfb}
\end{align}
\end{subequations}
and
\begin{equation} \label{eq:kappam}
\kappa^{m+1} = \pi^h \left[ 
g^{\frac12}(\vec X^m)\,\vec Y^{m+1}\,.\,\vec\omega^m + \tfrac12\,
\frac{\vec\omega^m}{|\vec\omega^m|}\,.\,\nabla\,\ln g(\vec X^m) \right].
\end{equation}
Notice that (\ref{eq:fdpwfb}) was obtained on combining (\ref{eq:kappam}) with
a fully discrete variant of (\ref{eq:sideh}), and noting (\ref{eq:omegah}),
in order to obtain a lower
dimensional linear system to solve for the unknowns $\vec X^{m+1}$ 
and $\vec Y^{m+1}$ that is decoupled from (\ref{eq:kappam}). 
Moreover,
(\ref{eq:fdpwfa}) is a fully discrete approximation of (\ref{eq:dLdxflow2h}),
on noting the definition of $\kappa^m_g$.

We make the following mild assumption.
\begin{tabbing}
$(\mathfrak A)^h$\quad \=
Let $|\vec{X}^m_\rho| > 0$ for almost all $\rho\in I$, let
$\dim \spa \{ \vec \omega^m (q_j) : j = 1,\ldots,J\} = 2$, 
\\ \> and let $\vec\omega^m(q_j) \not= \vec 0$, $j = 1,\ldots,J$.
\end{tabbing}
\revised{%
The above assumption can only be violated if all the 
vertex normals $\vec\omega^m(q_j)$ of $\Gamma^m$ are collinear,
or if two neighbouring edges of $\Gamma^m$ overlap.
Clearly, this almost never happens in practice, and it certainly cannot 
happen if $\Gamma^m$ has no self-intersections. See also 
\cite[Remark~2.2]{triplej} for more details.}

\begin{lemma} \label{lem:ex}
Let the assumption $(\mathfrak A)^h$ hold.
Then there exists a unique solution \linebreak
$(\vec X^{m+1}, \kappa^{m+1},\vec Y^{m+1}) \in \Vh \times V^h\times\Vh$ to 
$(\BGNpwf_m)^h$.
\end{lemma}
\begin{proof}
As (\ref{eq:fdpwf}) is linear, existence follows from uniqueness.
To investigate the latter, we consider the system: 
Find $(\vec X,\vec Y) \in \Vh \times \Vh$ such that
\begin{subequations}
\begin{align} 
\left(g^\frac32(\vec X^m)\,
\revised{\vec X}\,.\,\vec\omega^m,
\vec\chi\,.\,\vec\omega^m\,|\vec X^m_\rho| \right)^h
- \ttau_m \left(\vec Y_s, \vec\chi_s\,|\vec X^m_\rho| \right)
&= 0 \qquad \forall \ \vec \chi \in \Vh\,,
\label{eq:proofa}\\
\left( g^{\frac12}(\vec X^m)\,\vec Y\,.\,\vec\omega^m,
\vec\eta\,.\,\vec\omega^m \,|\vec X^m_\rho| \right)^h 
+ \left( \vec X_s,\vec\eta_s\,|\vec X^m_\rho|\right) &= 0
\qquad \forall\ \vec\eta \in \Vh\,. \label{eq:proofb}
\end{align}
\end{subequations}
Choosing $\vec\chi = \vec X$ in (\ref{eq:proofa}) and $\vec\eta = \vec Y$
in (\ref{eq:proofb}), and combining, yields that 
\begin{equation} \label{eq:Xomega}
\pi^h\,[\vec X\,.\,\vec\omega^m] = \pi^h\,[\vec Y\,.\,\vec\omega^m] = 0
\in V^h\,. 
\end{equation}
As a consequence, it follows from choosing $\vec\chi = \vec Y$ in 
(\ref{eq:proofa}) and $\vec\eta = \vec X$ in (\ref{eq:proofb}) that
$\vec X$ and $\vec Y$ are constant vectors. Combining (\ref{eq:Xomega}) and
the assumption $(\mathfrak A)^h$ then yields that
$\vec X = \vec Y = \vec 0 \in \Vh$.

Hence we have shown the existence of a unique $(\vec X^{m+1}, \vec Y^{m+1}) 
\in \Vh \times\Vh$ solving (\ref{eq:fdpwf}), which via (\ref{eq:kappam}) yields
existence and uniqueness of $\kappa^{m+1} \in V^h$.
\end{proof}

\subsection{Based on $\kappa^{m+1}_g$} \label{sec:52}
We propose the following fully discrete approximation of
$(\BGNpwfwf_h)^\diamond$.
\\ \noindent
$(\BGNpwfwf_m)^\diamond$:
Let $(\vec X^0,\kappa^0_g,\vec Y^0_g) \in \Vh \times V^h \times \Vh$. 
For $m=0,\ldots,M-1$, 
find $(\vec X^{m+1}, \kappa^{m+1}_g,$ $ \vec Y^{m+1}_g) 
\in \Vh \times V^h \times \Vh$ such that
\begin{subequations} \label{eq:B8}
\begin{align} 
& \left(g(\vec X^m)\,\frac{\vec X^{m+1} - \vec X^m}{\ttau_m}\,.\,
\vec\omega^m, \vec\chi\,.\,\vec\omega^m\,
|\vec X^m_\rho|_g \right)^\diamond
- \left((\vec Y^{m+1}_g)_s, \vec\chi_s\,|\vec X^m_\rho|_g \right)^\diamond
\nonumber \\ & \qquad
+ \left((\vec Y^m_g)_s\,.\,\vec\tau^m, \vec\chi_s\,.\,\vec\tau^m
\,|\vec X^m_\rho|_g \right)^\diamond
\nonumber \\ & 
 = -\tfrac12 \left( (\kappa^m_g\,)^2 + 2\,\lambda
- \vec Y^m_g\,.\,\nabla\,\ln\,g(\vec X^m), \left[\vec\chi_s\,.\,\vec\tau^m + 
\tfrac12\,\vec\chi\,.\,\nabla\,\ln\,g(\vec X^m)\right] |\vec X^m_\rho|_g
 \right)^\diamond \nonumber \\ & \qquad
+ \tfrac12 \left((D^2\,\ln\,g(\vec X^m))\,\vec Y^m_g, \vec\chi\,
|\vec X^m_\rho|_g \right)^\diamond \nonumber \\ & \qquad
+  \left( g^\frac12(\vec X^m)\,\kappa^m_g\,\vec Y^m_g\,.\,\vec\nu^m
+ \tfrac12\,(\vec Y^m_g)_s\,.\,\vec\tau^m, \vec\chi\,.\,
(\nabla\,\ln\,g(\vec X^m))\,
|\vec X^m_\rho|_g \right)^\diamond \nonumber \\ & \qquad
+ \left( g^\frac12(\vec X^m)\,\kappa^m_g,
\vec\chi_s\,.\,(\vec Y^m_g)^\perp\,|\vec X^m_\rho|_g \right)^\diamond
\quad \forall\ \vec\chi \in \Vh\,, \label{eq:B8a} \\
&\left(\kappa^{m+1}_g - g^\frac12(\vec X^m)\,\vec Y^{m+1}_g\,.\,\vec\nu^m, 
\chi\,|\vec X^m_\rho|_g \right)^\diamond = 0
\quad \forall\ \chi \in V^h\,, \label{eq:B8b} \\
&
\left(g^\frac12(\vec X^m)\,\kappa_g^{m+1}\,\vec\nu^m, 
\vec\eta\,|\vec X^m_\rho|_g \right)^\diamond
+ \left(\vec X^{m+1}_s,\vec\eta_s\,|\vec X^m_\rho|_g\right)^\diamond
+ \tfrac12 \left( \nabla\,\ln\,g(\vec X^m), \vec\eta\,|\vec
X^m_\rho|_g\right)^\diamond 
\nonumber \\ & \hspace{8cm}
= 0
\quad \forall\ \vec\eta \in \Vh\,. \label{eq:B8c}
\end{align}
\end{subequations}
Of course, in the case $(\cdot,\cdot)^\diamond = (\cdot,\cdot)^h$, 
(\ref{eq:B8b}) gives rise to
$\kappa^{m+1}_g = \pi^h\,[g^\frac12(\vec X^m)$ 
$\vec Y^{m+1}_g\,.\,
\vec\omega^m]$,
on noting (\ref{eq:omegah}), and so $\kappa^{m+1}_g$ can be eliminated from 
(\ref{eq:B8a}) to give rise to a coupled linear system 
involving only $\vec X^{m+1}$ and $\vec Y^{m+1}_g$, similarly to 
(\ref{eq:fdpwf}).

We make the following mild assumption.
\begin{tabbing}
$(\mathfrak B)^\diamond$\quad \=
Let $|\vec{X}^m_\rho| > 0$ for almost all $\rho\in I$, and let
$\dim \spa \mathcal Z^\diamond = 2$, where \\ \> $\mathcal Z^\diamond = 
\left\{ \left( g^\frac12(\vec X^m)\,\vec\nu^m, 
\chi\, |\vec X^m_\rho|_g \right)^\diamond : \chi \in V^h \right \} 
\subset \bR^2$.
\end{tabbing}
\revised{%
In the case $(\cdot,\cdot)^\diamond = (\cdot,\cdot)^h$ the above assumption
collapses to the first part of $(\mathfrak A)^h$, as was demonstrated below
(3.11) in \cite{hypbol}. For more general quadrature rules, it is only violated
if some
$g$-weighted vertex normals of $\Gamma^m$ are all collinear, which means that
$(\mathfrak B)^\diamond$ is clearly a very mild constraint.}

\begin{lemma} \label{lem:exg}
Let the assumptions $(\mathfrak A)^h$ and 
$(\mathfrak B)^\diamond$ hold.
Then there exists a unique solution
$(\vec X^{m+1}, \kappa^{m+1}_g, \vec Y^{m+1}_g) \in \Vh \times V^h \times \Vh$
to $(\BGNpwfwf_m)^\diamond$.
\end{lemma}
\begin{proof}
As (\ref{eq:B8}) is linear, existence follows from uniqueness. 
To investigate the latter, we consider the system: 
Find $(\vec X,\kappa_g, \vec Y_g) \in \Vh \times V^h \times \Vh$ such that
\begin{subequations}
\begin{align} 
\left(g(\vec X^m)\,\vec X\,.\,\vec\omega^m, \vec\chi\,.\,\vec\omega^m\,
|\vec X^m_\rho|_g \right)^\diamond
- \ttau_m \left((\vec Y_g)_s, \vec\chi_s\,|\vec X^m_\rho|_g \right)^\diamond
& = 0 \quad \forall\ \vec\chi \in \Vh\,, \label{eq:proofga} \\
\left(\kappa_g - g^\frac12(\vec X^m)\,\vec Y_g\,.\,\vec\nu^m, 
\chi\,|\vec X^m_\rho|_g \right)^\diamond & = 0
\quad \forall\ \chi \in V^h\,, \label{eq:proofgb} \\
\left(g^\frac12(\vec X^m)\,\kappa_g\,\vec\nu^m, 
\vec\eta\,|\vec X^m_\rho|_g \right)^\diamond
+ \left(\vec X_s,\vec\eta_s\,|\vec X^m_\rho|_g\right)^\diamond
& = 0
\quad \forall\ \vec\eta \in \Vh\,. \label{eq:proofgc}
\end{align}
\end{subequations}
Choosing $\vec\chi = \vec X$ in (\ref{eq:proofga}),
$\chi = \kappa_g$ in (\ref{eq:proofgb}) and $\vec\eta = \vec Y_g$
in (\ref{eq:proofgc}) yields that 
\begin{equation*} 
\left(g(\vec X^m)\,(\vec X\,.\,\vec\omega^m)^2, 
|\vec X^m_\rho|_g \right)^\diamond
+ \ttau_m \left((\kappa_g)^2,|\vec X^m_\rho|_g \right)^\diamond = 0\,,
\end{equation*}
and so it follows from (\ref{eq:Idiamond}), recall $K\geq2$, 
and the positivities of $g(\vec X^m)$ and $|\vec X^m_\rho|$, that 
\begin{equation} \label{eq:Xomegag}
\kappa_g = 0 \in V^h
\quad\text{and}\quad
\left(g(\vec X^m)\,\vec X\,.\,\vec\omega^m, \eta\,
|\vec X^m_\rho|_g \right)^\diamond = 0 \qquad \forall\ \eta \in C(\overline I)
\,. 
\end{equation}
As a consequence, it follows from choosing $\vec\chi = \vec Y_g$ in 
(\ref{eq:proofga}) and $\vec\eta = \vec X$ in (\ref{eq:proofgc}) that
$\vec X$ and $\vec Y_g$ are constant vectors. Combining 
(\ref{eq:proofgb}), $\kappa_g = 0$ and the assumption $(\mathfrak B)^\diamond$ 
then yields that $\vec Y_g = \vec 0 \in \Vh$.
Moreover, it follows from (\ref{eq:Xomegag}), (\ref{eq:Idiamond}), 
recall $K\geq2$, and $\vec X$ being a constant that
$\vec X\,.\,\vec\omega^m = 0 \in V^h$. Combining this with 
the assumption $(\mathfrak A)^h$ yields that
$\vec X = \vec 0 \in \Vh$.
Hence there exists a unique solution
$(\vec X^{m+1}, \kappa^{m+1}_g, \vec Y^{m+1}_g) \in \Vh \times V^h \times \Vh$
to $(\BGNpwfwf_m)^\diamond$.
\end{proof}

\setcounter{equation}{0}
\section{Numerical results} \label{sec:nr}

Unless otherwise stated, in all our computations we set $\lambda=0$.
For the scheme $(\BGNpwfwf_m)^\diamond$ we either consider
$(\BGNpwfwf_m)^h$, recall (\ref{eq:ip0}), or
$(\BGNpwfwf_m)^\star$, where $(\cdot,\cdot)^\star$ denotes a quadrature that 
is exact for polynomials of degree up to five.

On recalling \revised{\eqref{eq:Wglambdah} and \eqref{eq:kappaidh}}, 
for solutions of the scheme 
$(\BGNpwf_m)^h$ we define
$W_{g,\lambda}^{m+1} = \tfrac12\,( (\vec Y^{m+1}\,.\,\vec\omega^m)^2 +
2\,\lambda, g^\frac12(\vec X^m)\, |\vec X^m_\rho| )^h$
as the natural discrete analogue of (\ref{eq:Wglambda}),
\revised{%
while for solutions of $(\BGNpwfwf_m)^\diamond$ we let}
$\widetilde W_{g,\lambda}^{m+1} =
\tfrac12\, ( (\kappa_g^{m+1})^2 + 2\,\lambda,|\vec X^m_\rho|_g )^\diamond$

We also consider the ratio
\begin{equation} \label{eq:ratio}
\ratio^m = \dfrac{\max_{j=1\to J} |\vec{X}^m(q_j) - \vec{X}^m(q_{j-1})|}
{\min_{j=1\to J} |\vec{X}^m(q_j) - \vec{X}^m(q_{j-1})|}
\end{equation}
between the longest and shortest element of $\Gamma^m$, and are often
interested in the evolution of this ratio over time.

In order to define the initial data for the 
schemes $(\BGNpwf_m)^h$ and $(\BGNpwfwf_m)^\diamond$ we define, 
given $\Gamma^0 = \vec X^0(\overline I)$, the discrete curvature vector
$\vec\kappa^0\in \Vh$ such that
\begin{equation*} 
\left( \vec\kappa^{0},\vec\eta\, |\vec X^0_\rho| \right)^h
+ \left( \vec{X}^{0}_s , \vec\eta_s\,|\vec X^0_\rho| \right)
 = 0 \quad \forall\ \vec\eta \in \Vh\,,
\end{equation*}
recall (\ref{eq:varkappa}).
Then we set
$\kappa^0 = \pi^h\left[\frac{\vec\kappa^0\,.\,\vec\omega^0}{|\vec\omega^0|}
\right]$ and, as a discrete analogue to (\ref{eq:varkappag}), we let 
$\kappa_g^0 = \pi^h\left[ 
g^{-\frac12}(\vec X^0)\left[\kappa^0 
- \tfrac{1}{2} \, \frac{\vec\omega^0}{|\vec\omega^0|}
\,.\,\nabla\,\ln g(\vec X^0)\right]\right]$.
Finally, on recalling (\ref{eq:kappaidh}) and (\ref{eq:B2}), we set
$\vec Y^0 = \vec\pi^h\left[ |\vec\omega^0|^{-2}\,
\kappa_g^0\,\vec\omega^0 \right]$ 
and $\vec Y_g^0 = \vec\pi^h\left[ g^{-\frac12}(\vec X^0)\,|\vec\omega^0|^{-2}\,
\kappa_g^0\,\vec\omega^0\right]$.

\subsection{Elliptic plane: (\ref{eq:galpha}) with $\alpha=-1$}

For the elliptic plane, we recall the true solution 
\begin{equation} \label{eq:app_ansatz}
\vec x(\rho, t) = a(t)\,\vec\ek_2 + r(t)\left[
\cos2\,\pi\,\rho\,\vec\ek_1 + \sin2\,\pi\,\rho\,\vec\ek_2 \right]
\qquad \rho \in I\,,
\end{equation}
with
\begin{equation} \label{eq:appB_ODE}
a(t) = 0 \quad\text{and}\quad
\ddt\, r^4(t) = \tfrac1{8}\,(1 - \alpha^2\,r^4(t))\,
(1 - 6\,\alpha\,r^2(t) + \alpha^2\,r^4(t)) \,,
\end{equation}
for $\alpha = -1$, from Appendix~A.2 in \cite{hypbol}. 
\revised{An explicit formula for $r(t)$ is stated in \cite[(A.17)]{hypbol}.}
We use this true solution for a convergence test. 
To this end, we start with the initial data 
\begin{equation} \label{eq:X0}
\vec X^0(q_j) = a(0)\,\vec\ek_2 + 
r(0) \begin{pmatrix} 
\cos[2\,\pi\,q_j + 0.1\,\sin(2\,\pi\,q_j)] \\
\sin[2\,\pi\,q_j + 0.1\,\sin(2\,\pi\,q_j)]
\end{pmatrix}, \quad j = 1,\ldots,J\,,
\end{equation}
recall (\ref{eq:Jequi}), with $r(0)=1.5$ and $a(0)=0$,
for $J \in \{32,64,128,256,512\}$. We compute the error 
$\errorXx = \max_{m=1,\ldots,M} \max_{j=1,\ldots,J} | 
|\vec X^m(q_j) - a(t_m)\,\vec\ek_2| - r(t_m)|$
over the time interval $[0,1]$ between the true solution (\ref{eq:app_ansatz}) 
and the discrete solutions for the schemes
$(\BGNpwf_m)^h$, $(\BGNpwfwf_m)^h$ and $(\BGNpwfwf_m)^\star$. 
We note that the circle is shrinking, and
reaches a radius $r(T) = 1.148$ at time $T=1$.
Here, and in the convergence experiments that follow, 
we use the time step size $\ttau=0.1\,h^2_{\Gamma^0}$,
where $h_{\Gamma^0}$ is the maximal edge length of $\Gamma^0$.
The computed errors are reported in Table~\ref{tab:wf4}.
\begin{table}
\center
\caption{Errors for the convergence test for \eqref{eq:app_ansatz} with
\eqref{eq:appB_ODE} for $\alpha=-1$, with $r(0) = 1.5$,
over the time interval $[0,1]$.
\revised{The ratios \eqref{eq:ratio} at time $t=1$ for the last row are
$1.14$, $1.49$ and $1.14$, respectively.}
}
\begin{tabular}{|r|c|c|c|c|c|c|}
\hline
& \multicolumn{2}{c|}{$(\BGNpwf_m)^h$}&
\multicolumn{2}{c|}{$(\BGNpwfwf_m)^h$} &
\multicolumn{2}{c|}{$(\BGNpwfwf_m)^\star$} \\
 $h_{\Gamma^0}$ & $\errorXx$ & EOC & $\errorXx$ & EOC& $\errorXx$ & EOC 
\\ \hline
2.1544e-01&7.1380e-03&--- &1.4510e-02&--- &1.2582e-02&--- \\
1.0792e-01&1.7446e-03&2.04&3.5351e-03&2.04&3.0547e-03&2.05\\
5.3988e-02&4.3377e-04&2.01&8.7838e-04&2.01&7.5835e-04&2.01\\
2.6997e-02&1.0829e-04&2.00&2.1926e-04&2.00&1.8926e-04&2.00\\
1.3499e-02&2.7064e-05&2.00&5.4795e-05&2.00&4.7295e-05&2.00\\
\hline
\end{tabular}
\label{tab:wf4}
\end{table}%

\subsection{Hyperbolic disk: (\ref{eq:galpha}) with $\alpha=1$}

For the hyperbolic disk, we recall the true solution (\ref{eq:app_ansatz}),
(\ref{eq:appB_ODE}) for $\alpha = 1$, from Appendix~A.2 in \cite{hypbol}. 
\revised{%
A nonlinear equation satisfied by $r(t)$ is stated in \cite[(A.19)]{hypbol},
which we solve in practice with the help of a Newton iteration.}
Similarly to Table~\ref{tab:wf4}
we start with the initial data (\ref{eq:X0}) with $r(0)=0.1$ and $a(0) = 0$.
We compute the error $\errorXx$ over the time interval $[0,1]$ between
the true solution (\ref{eq:app_ansatz}) 
and the discrete solutions for the schemes
$(\BGNpwf_m)^h$, $(\BGNpwfwf_m)^h$ and $(\BGNpwfwf_m)^\star$. 
We note that the circle is expanding, and
reaches a radius $r(T) = 0.404$ at time $T=1$.
The computed errors are reported in Table~\ref{tab:wf3}.
\revised{
Surprisingly, and in contrast to the other tables, the ratio \eqref{eq:ratio}
reaches the value $1$ for all three schemes, and for all presented values of
$J$, at the final time. It is not clear why the tangential motion implicit in
$(\BGNpwfwf_h)^\diamond$ appears to lead to equidistribution in this example,
but it may have to do with the fact that we compute an expanding circle
solution in $\mathbb{D}_1$, the hyperbolic disk, recall \eqref{eq:galpha}. 
For more general evolutions we do not observe equidistribution for
$(\BGNpwfwf_m)^h$ or $(\BGNpwfwf_m)^\star$ in practice.
}
\begin{table}
\center
\caption{Errors for the convergence test for \eqref{eq:app_ansatz} with
\eqref{eq:appB_ODE} for $\alpha=1$, with $r(0) = 0.1$,
over the time interval $[0,1]$.
\revised{The ratios \eqref{eq:ratio} at time $t=1$ for the last row are
equal to $1.00$ for all three schemes.}
}
\begin{tabular}{|r|c|c|c|c|c|c|}
\hline
 & \multicolumn{2}{c|}{$(\BGNpwf_m)^h$}&
\multicolumn{2}{c|}{$(\BGNpwfwf_m)^h$} &
\multicolumn{2}{c|}{$(\BGNpwfwf_m)^\star$} \\
$h_{\Gamma^0}$ & $\errorXx$ & EOC & $\errorXx$ & EOC& $\errorXx$ & EOC 
\\ \hline
2.1544e-01&1.8356e-03&--- &1.8655e-03&--- &2.3602e-03&--- \\
1.0792e-01&4.5233e-04&2.03&4.5938e-04&2.03&5.8378e-04&2.02\\
5.3988e-02&1.1270e-04&2.01&1.1444e-04&2.01&1.4583e-04&2.00\\
2.6997e-02&2.8151e-05&2.00&2.8590e-05&2.00&3.6450e-05&2.00\\
1.3499e-02&7.0364e-06&2.00&7.1460e-06&2.00&9.1121e-06&2.00\\
\hline
\end{tabular}
\label{tab:wf3}
\end{table}%

\subsection{Hyperbolic plane: (\ref{eq:gmu}) with $\mu = 1$}

For the hyperbolic plane, 
we recall the true solution (\ref{eq:app_ansatz}) with
\begin{equation}
a(t) = a(0)\,\exp\left(-t + \tfrac12\,\int_0^t \sigma^2(u)\;{\rm d}u \right)
\quad\text{and}\quad
r(t) = \frac{a(t)}{\sigma(t)} \,,
\label{eq:ODEr}
\end{equation}
where $\sigma$ satisfies the ODE 
$\sigma'(t) = \sigma(t)\,(1 - \tfrac12\,\sigma^2(t))\,(\sigma^2(t) - 1)$,
from Appendix~A.1 in \cite{hypbol}. 
\revised{%
A nonlinear equation satisfied by $\sigma(t)$ is stated in 
\cite[Appendix~A.1]{hypbol}, which we}
\revised{solve in practice with the help of a Newton
iteration. Moveover, $a(t)$ can be obtained from \eqref{eq:ODEr} via numerical
integration using e.g.\ Romberg's method.}
As initial data we use (\ref{eq:X0}) with $r(0)=1$ and $a(0) = 2$.
We recall from Appendix~A.1 in \cite{hypbol} 
that the circle will raise and expand. In fact,
at time $T=1$ it holds that $r(T) = 1.677$ and $a(T) = 2.411$.
The computed errors are reported in Table~\ref{tab:wf2}, and they 
\revised{can} be
compared with the corresponding numbers in \revised{\cite[Tab.\ 7]{hypbol}}.
\revised{In particular, $(\BGNpwf_m)^h$ exhibits smaller errors than
$(\mathcal{U}_m)^h$ in \cite{hypbol}, while the errors of
$(\mathcal{W}_m)^h$ in \cite{hypbol}, are very close to
$(\BGNpwfwf_m)^\star$.}
\begin{table}
\center
\caption{Errors for the convergence test for \eqref{eq:app_ansatz} with
\eqref{eq:ODEr}, with $r(0) = 1$, $a(0) = 2$,
over the time interval $[0,1]$.
\revised{The ratios \eqref{eq:ratio} at time $t=1$ for the last row are
$1.07$, $2.59$ and $1.65$, respectively.}
}
\begin{tabular}{|r|c|c|c|c|c|c|}
\hline
& \multicolumn{2}{c|}{$(\BGNpwf_m)^h$}&
\multicolumn{2}{c|}{$(\BGNpwfwf_m)^h$} &
\multicolumn{2}{c|}{$(\BGNpwfwf_m)^\star$} \\
$h_{\Gamma^0}$ & $\errorXx$ & EOC & $\errorXx$ & EOC& $\errorXx$ & EOC 
\\ \hline
2.1544e-01&1.2690e-01&--- &7.5442e-02&--- &4.3265e-02&--- \\
1.0792e-01&3.1923e-02&2.00&1.9548e-02&1.95&1.0719e-02&2.02\\
5.3988e-02&7.9911e-03&2.00&4.9076e-03&2.00&2.6764e-03&2.00\\
2.6997e-02&1.9984e-03&2.00&1.2291e-03&2.00&6.6898e-04&2.00\\
1.3499e-02&4.9966e-04&2.00&3.0741e-04&2.00&1.6723e-04&2.00\\
\hline
\end{tabular}
\label{tab:wf2}
\end{table}%
We repeat the convergence test with the initial data $r(0)=1$ and $a(0)=1.1$,
so that the circle will now sink and shrink. In fact,
at time $T=1$ it holds that $r(T) = 0.645$ and $a(T) = 0.792$.
The computed errors are reported in Table~\ref{tab:wf1}, and they 
\revised{can} be
compared with the corresponding numbers in \revised{\cite[Tab.\ 6]{hypbol}}. 
\revised{In particular, $(\BGNpwf_m)^h$ exhibits significantly smaller errors 
than $(\mathcal{U}_m)^h$ in \cite{hypbol}, and similarly 
$(\BGNpwfwf_m)^h$ and $(\BGNpwfwf_m)^\star$ show significantly smaller errors 
than $(\mathcal{W}_m)^h$ in \cite{hypbol}.}
We observe that the approximation $(\BGNpwfwf_m)^h$ exhibits non-optimal
convergence rates for this experiment,
\revised{
which appears to have two causes. Firstly, the induced tangential
motion of $(\BGNpwfwf_m)^h$ leads to a large ratio $\ratio^m$, with
comparatively large elements at the bottom of the evolving circle. And
secondly, compared to the experiments in Table~\ref{tab:wf2}, the evolving
circle is now closer to the $\vec\ek_1$--axis, and hence the associated 
singularity of $g$ has a stronger effect.}
All the other experiments,
and all the other schemes, always show the expected quadratic convergence rate.
\begin{table}
\center
\caption{Errors for the convergence test for \eqref{eq:app_ansatz} with
\eqref{eq:ODEr}, with $r(0) = 1$, $a(0) = 1.1$,
over the time interval $[0,1]$.
\revised{The ratios \eqref{eq:ratio} at time $t=1$ for the last row are
$1.07$, $2.76$ and $1.20$, respectively.}
}
\begin{tabular}{|r|c|c|c|c|c|c|}
\hline
& \multicolumn{2}{c|}{$(\BGNpwf_m)^h$}&
\multicolumn{2}{c|}{$(\BGNpwfwf_m)^h$} &
\multicolumn{2}{c|}{$(\BGNpwfwf_m)^\star$} \\
$h_{\Gamma^0}$ & $\errorXx$ & EOC & $\errorXx$ & EOC & $\errorXx$ & EOC 
\\ \hline
 2.1544e-01&2.9884e-03&--- &5.3699e-02&--- &1.1530e-02&--- \\
 1.0792e-01&9.7352e-04&1.62&1.6346e-02&1.72&2.9345e-03&1.98\\
 5.3988e-02&2.6531e-04&1.88&5.3475e-03&1.61&7.3673e-04&2.00\\
 2.6997e-02&6.7844e-05&1.97&2.5787e-03&1.05&1.8436e-04&2.00\\
 1.3499e-02&1.7057e-05&1.99&5.8915e-04&2.13&4.6102e-05&2.00\\
\hline
\end{tabular}
\label{tab:wf1}
\end{table}%

We recall that in Figures~10, 11 and 13 of \cite{hypbol}, the authors 
show some curve evolutions for elastic flow in the hyperbolic plane. 
Repeating these
simulations, for the same discretization parameters, for the newly introduced
schemes $(\BGNpwf_m)^h$, $(\BGNpwfwf_m)^h$ and $(\BGNpwfwf_m)^\star$, yields
very similar curve evolutions. As expected, the main difference is in the
evolution of the ratio (\ref{eq:ratio}),
\revised{recall Remarks~\ref{rem:equid} and \ref{rem:noequid}.}
As an example, we show the evolution
of (\ref{eq:ratio}) for the experiment in \cite[Fig.\ 10]{hypbol} in
Figure~\ref{fig:wf_cigar_r}.
\begin{figure}
\center
\includegraphics[angle=-90,width=0.3\textwidth]{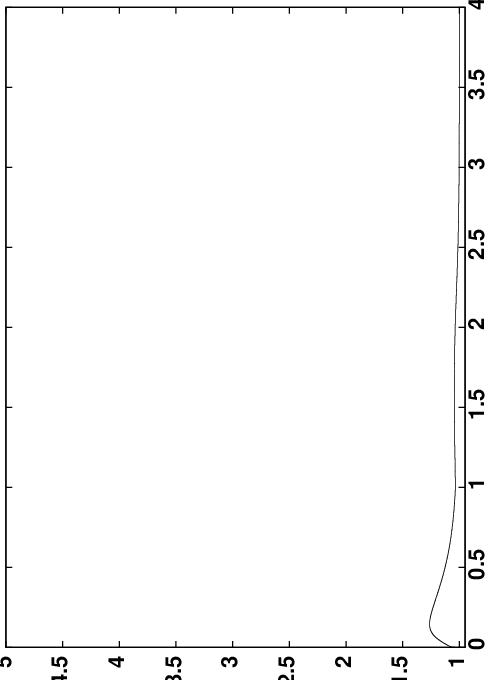}
\includegraphics[angle=-90,width=0.3\textwidth]{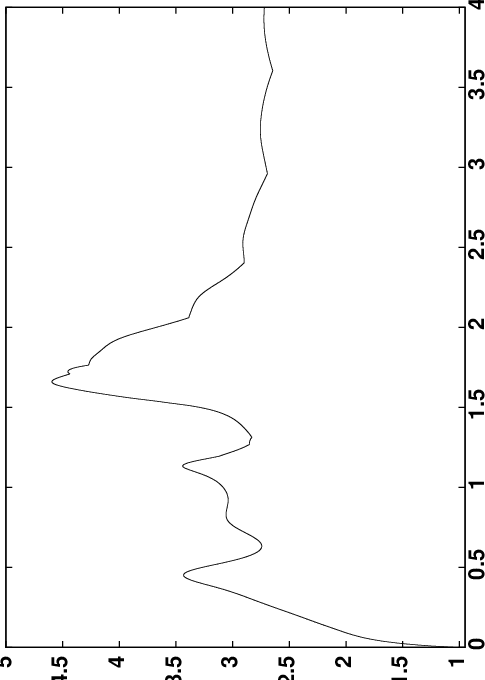}
\includegraphics[angle=-90,width=0.3\textwidth]{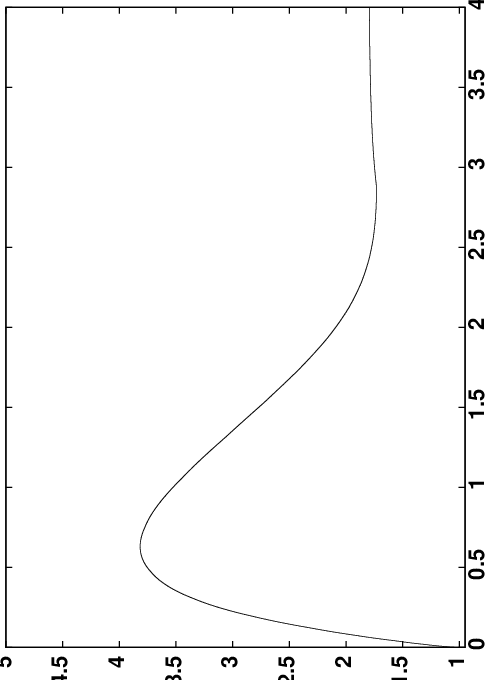}
\caption{
A plot of the ratio \eqref{eq:ratio}
for the schemes $(\BGNpwf_m)^h$, $(\BGNpwfwf_m)^h$ and $(\BGNpwfwf_m)^\star$.
}
\label{fig:wf_cigar_r}
\end{figure}%

\subsection{Geodesic elastic flow}
We \revised{begin with} 
two computations for geodesic elastic flow on a Clifford torus. 
To this end, we employ the metric induced by
(\ref{eq:gtorus}) with $\mathfrak s = 1$, so that the torus has radii $r=1$ and
$R = 2^\frac12$. As initial data we choose a circle in $H$ with radius $3$ and
centre $(0,2)^T$. For the simulation in Figure~\ref{fig:pwftorus2} we use the
scheme $(\BGNpwf_m)^h$ with the discretization parameters 
$J=256$ and $\ttau=10^{-3}$.
The scheme $(\BGNpwfwf_m)^h$ was not able to compute this evolution, 
due to a blow-up in the tangential part of $\vec Y^{m+1}$.
Hence we only present a comparison with $(\BGNpwfwf_m)^\star$, which gives
nearly identical results to $(\BGNpwf_m)^h$. However, the ratio 
(\ref{eq:ratio}) at time $t=50$ is $11.2$ for $(\BGNpwfwf_m)^\star$, while it
is only $1.1$ for $(\BGNpwf_m)^h$, 
\revised{recall the equidistribution property
from Remark~\ref{rem:equid}.}
\begin{figure}
\center
\includegraphics[angle=-90,width=0.2\textwidth]{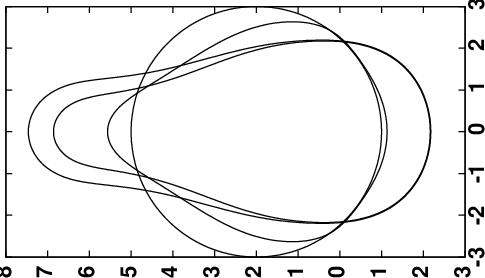}\qquad
\includegraphics[angle=-90,width=0.4\textwidth]{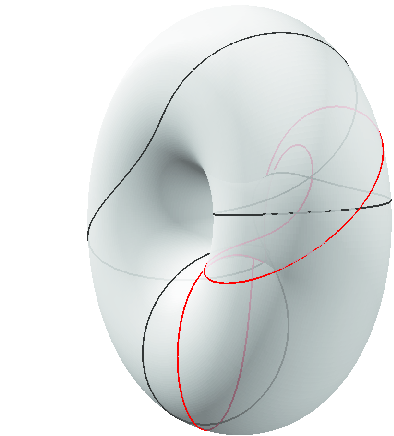} \\[2mm]
\includegraphics[angle=-90,width=0.3\textwidth]{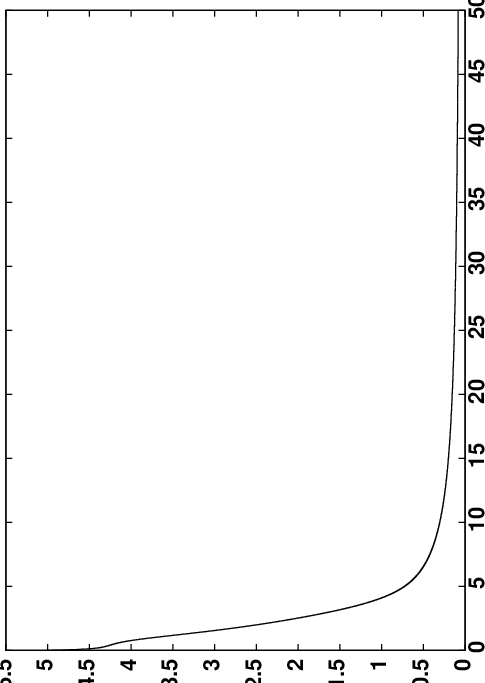}
\includegraphics[angle=-90,width=0.3\textwidth]{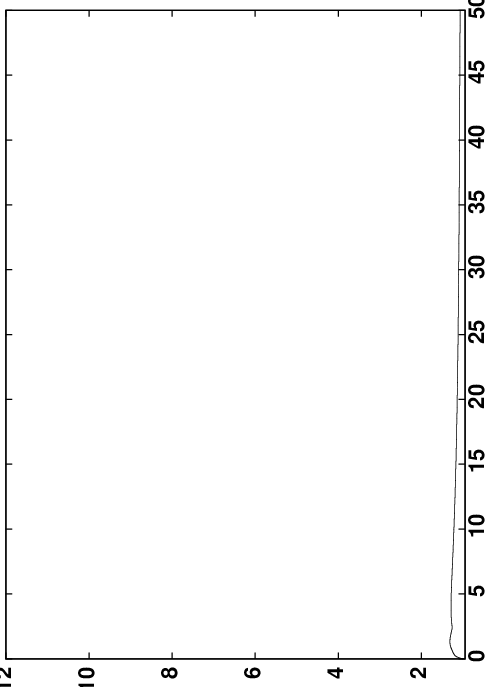}
\includegraphics[angle=-90,width=0.3\textwidth]{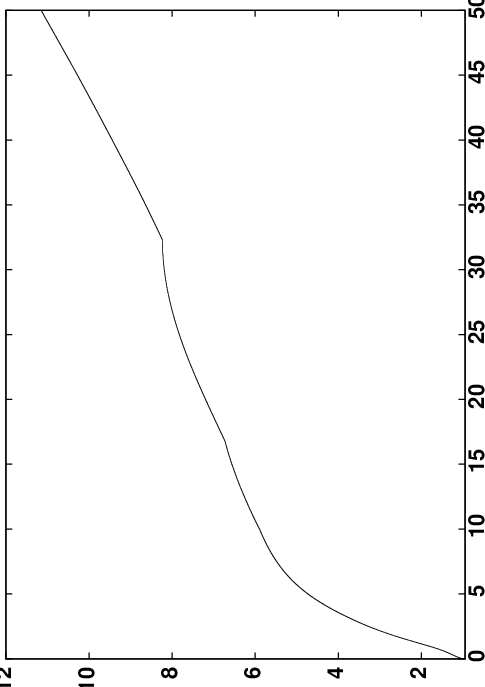}
\caption{
$(\BGNpwf_m)^h$
Geodesic elastic flow on a Clifford torus.
The solutions $\vec X^m$
at times $t = 0, 1, 10, 50$. On the right we visualize
$\vec\Phi(\vec X^m)$ at times $t=0$ (red) and $t=50$ (black), 
for \eqref{eq:gtorus} with $\mathfrak s=1$.
Below a plot of the discrete energy $W^{m+1}_{g,\lambda}$,
as well as of the ratio \eqref{eq:ratio} for $(\BGNpwf_m)^h$
and $(\BGNpwfwf_m)^\star$.
} 
\label{fig:pwftorus2}
\end{figure}%
Repeating the experiment with $\lambda=1$ gives the evolution shown in 
Figure~\ref{fig:pwftorus2lambda}.
In the case $\lambda=0$, 
the flow reduces the elastic energy and the absolute minimizer is given 
by geodesics which have geodesic curvature zero. 
However, in \revised{Figure~\ref{fig:pwftorus2lambda}} the elastic energy does 
not settle down to zero,
and the curves instead seem to converge to a non-trivial critical point of 
the elastic energy. This is in accordance with the analysis in
\cite{LangerS84}, which showed that in cases of hypersurfaces for which the 
Gaussian curvature is not non-negative at all points, the 
set of free elasticae, i.e., the set of critical points, is much richer.
\begin{figure}
\center
\includegraphics[angle=-90,width=0.3\textwidth]{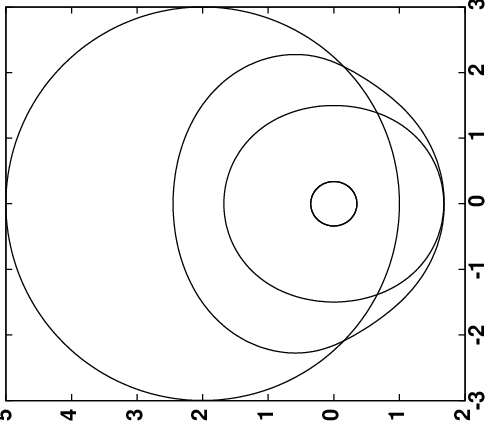}\qquad
\includegraphics[angle=-90,width=0.4\textwidth]{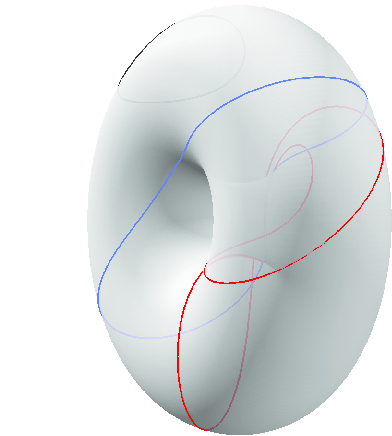} \\[2mm]
\includegraphics[angle=-90,width=0.3\textwidth]{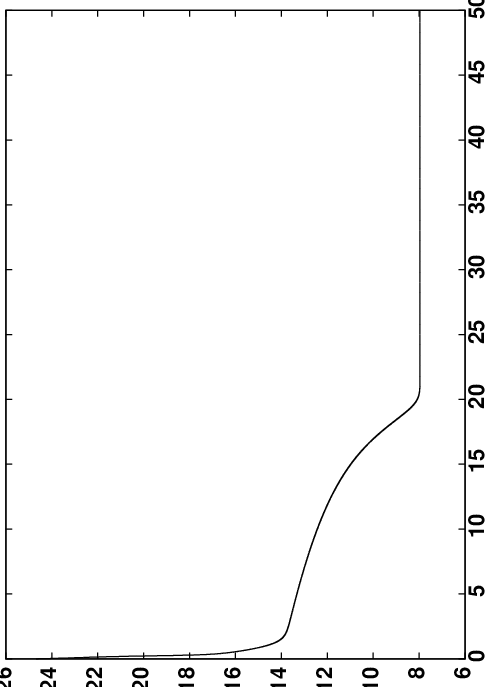}
\includegraphics[angle=-90,width=0.3\textwidth]{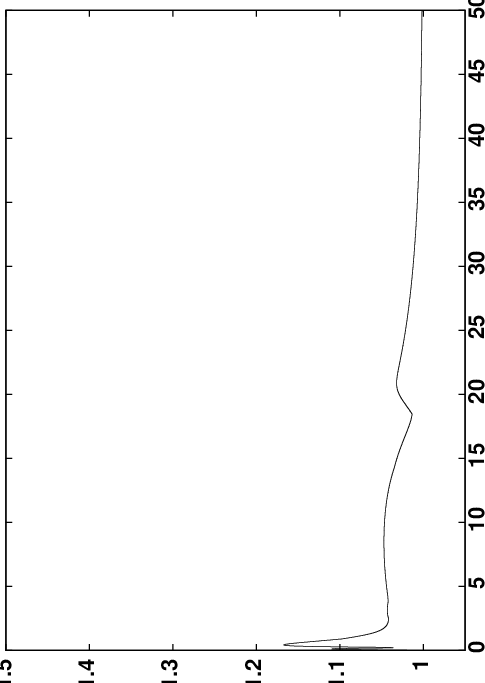}
\includegraphics[angle=-90,width=0.3\textwidth]{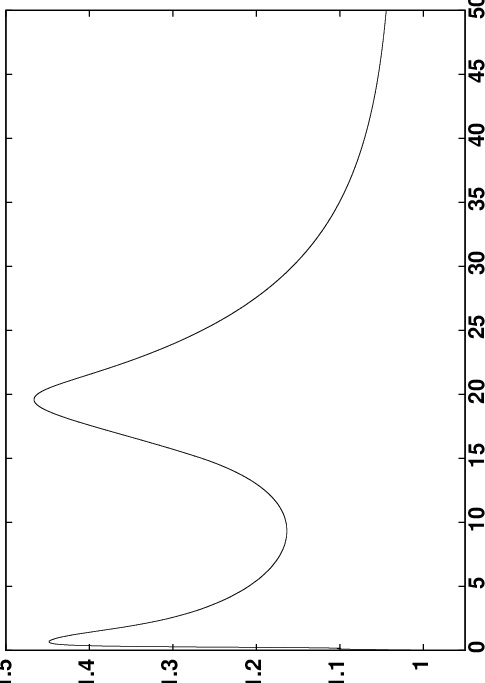}
\caption{
$(\BGNpwf_m)^h$
Generalized geodesic elastic flow, with $\lambda=1$, on a Clifford torus.
The solutions $\vec X^m$
at times $t = 0, 1, 10, 30, 50$. On the right we visualize
$\vec\Phi(\vec X^m)$ at times $t=0$ (red), $t=10$ (blue) and $t=50$ (black), 
for \eqref{eq:gtorus} with $\mathfrak s=1$.
Below a plot of the discrete energy $W^{m+1}_{g,\lambda}$,
as well as of the ratio \eqref{eq:ratio} for $(\BGNpwf_m)^h$
and $(\BGNpwfwf_m)^\star$.
} 
\label{fig:pwftorus2lambda}
\end{figure}%

\revised{%
We end this section with some computations of geodesic elastic flow on 
the unit sphere, inspired, for example, by the numerical results presented
in \cite[Figs.\ 1--8]{BrunnettC94} and \cite[Figs.\ 36, 37]{curves3d}.
To this end, we employ the metric induced by (\ref{eq:gMercator}),
which means that $\vec\Phi(H)$, the surface on which we compute geodesic 
elastic flow, is the unit sphere without the north and the south pole.
In particular, geodesic elastic flow evolutions on the unit sphere that pass
through these poles cannot be computed with our formulation. We demonstrate
this problem with a first simulation for (\ref{eq:gMercator}). The initial data
is chosen such that elastic flow on the unit sphere would lead to a simple
covering of a great circle. But as the curve would need to pass through the
north pole, this represents a blow-up in $H$ in finite time, and so our
approximation cannot compute the evolution beyond the crossing of the pole.
As initial data we choose}
\revised{%
a unit circle in $H$ centred at $2\,\vec\ek_1$.
For the simulation in Figure~\ref{fig:pwfspherepole} we use the
scheme $(\BGNpwf_m)^h$ with the discretization parameters
$J=256$ and $\ttau=10^{-6}$. We observe that $\Gamma^m$ expands into an
ellipse-like shape in $H$, leading to a blow-up to infinity some time after
$t=0.005$.}
\begin{figure}
\center
\includegraphics[angle=-90,width=0.55\textwidth]{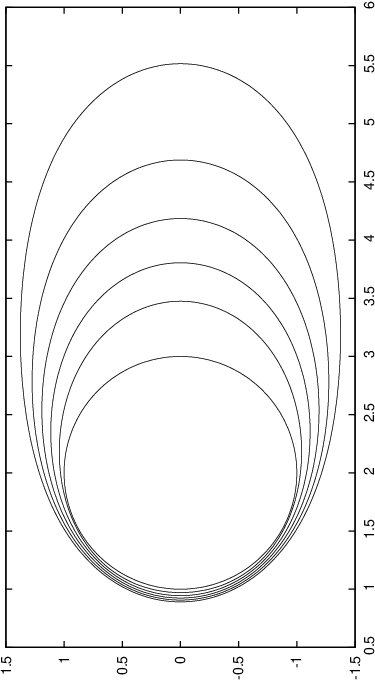}\quad
\includegraphics[angle=-90,width=0.4\textwidth]{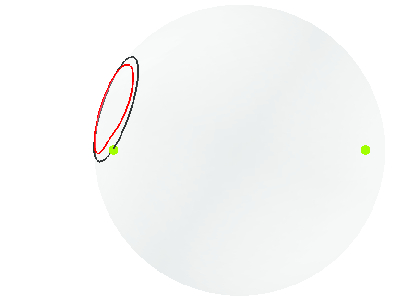} \\[2mm]
\includegraphics[angle=-90,width=0.3\textwidth]{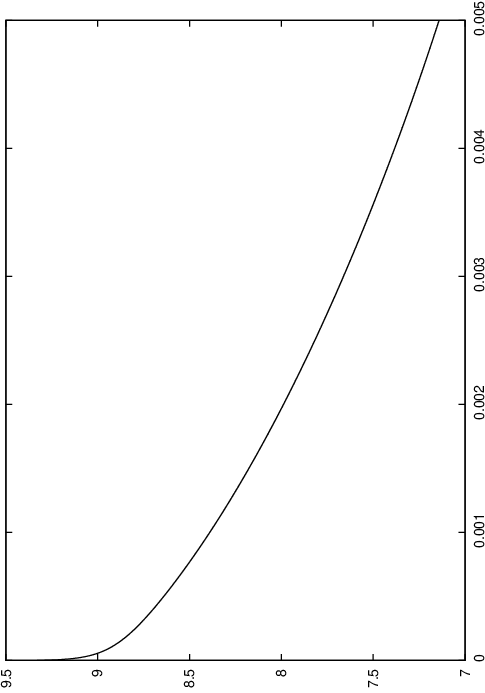}
\includegraphics[angle=-90,width=0.3\textwidth]{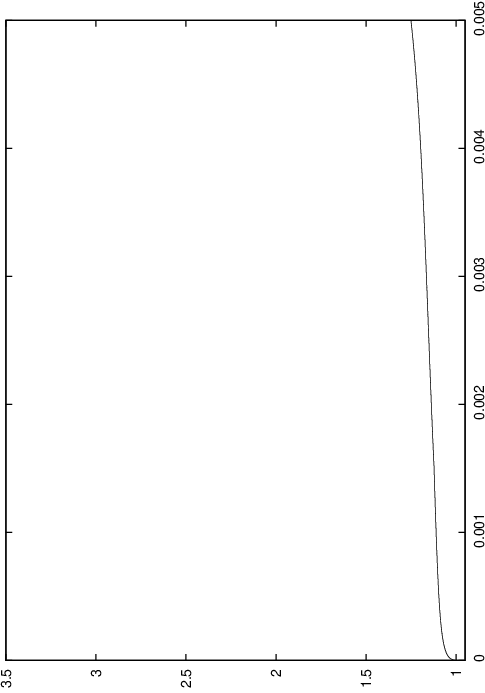}
\includegraphics[angle=-90,width=0.3\textwidth]{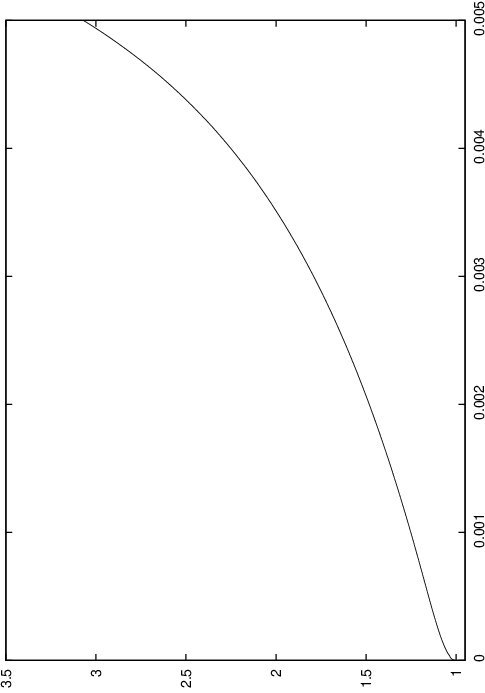}
\caption{
$(\BGNpwf_m)^h$
Geodesic elastic flow on the unit sphere.
The solutions $\vec X^m$ at times $t = 0, 0.001, \ldots, 0.005$. 
On the right we visualize
$\vec\Phi(\vec X^m)$ at times $t=0$ (red) and $t=0.005$ (black), 
for \eqref{eq:gMercator}, with the two poles, $\pm\vec\ek_3$, 
represented by green dots.
Below a plot of the discrete energy $W^{m+1}_{g,\lambda}$,
as well as of the ratio \eqref{eq:ratio} for $(\BGNpwf_m)^h$
and $(\BGNpwfwf_m)^\star$.
} 
\label{fig:pwfspherepole}
\end{figure}%
\revised{%
We remark that an alternative to the Mercator projection 
is given}
\revised{%
by the stereographic projection of the unit sphere, so that only one
of the two poles is missing, rather than two. This allows for the computation
of a slightly larger class of evolutions on the unit sphere.
Recall that the appropriate $g$ is defined by \eqref{eq:galpha} with 
$\alpha=-1$.}

\revised{
A more involved simulation is shown in Figure~\ref{fig:pwfsphere},
where we choose as initial data a $2\times8$ ellipse in $H$ centred at the 
origin and use the
scheme $(\BGNpwf_m)^h$ with the discretization parameters
$J=256$ and $\ttau=10^{-3}$.
We observe the evolution of the initial curve
towards a triple covering of a great circle}
\revised{on the sphere. Note that eventually the solution would like to settle
on the two poles, which would represent a singularity for the flow in $H$.
\begin{figure}
\center
\includegraphics[angle=-90,width=0.4\textwidth]{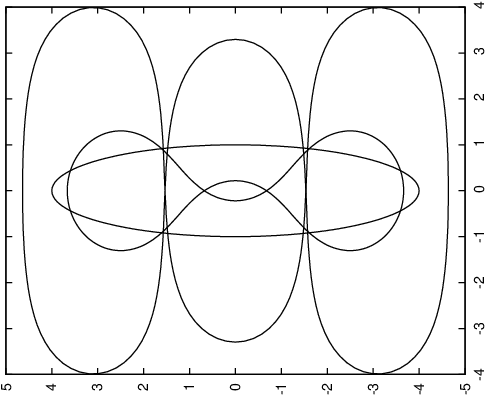}\qquad
\includegraphics[angle=-90,width=0.4\textwidth]{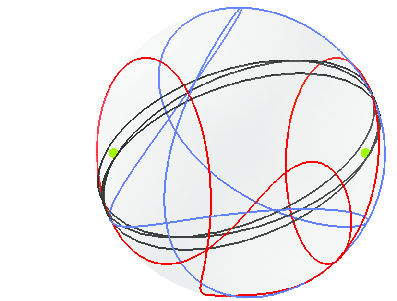} \\[2mm]
\includegraphics[angle=-90,width=0.3\textwidth]{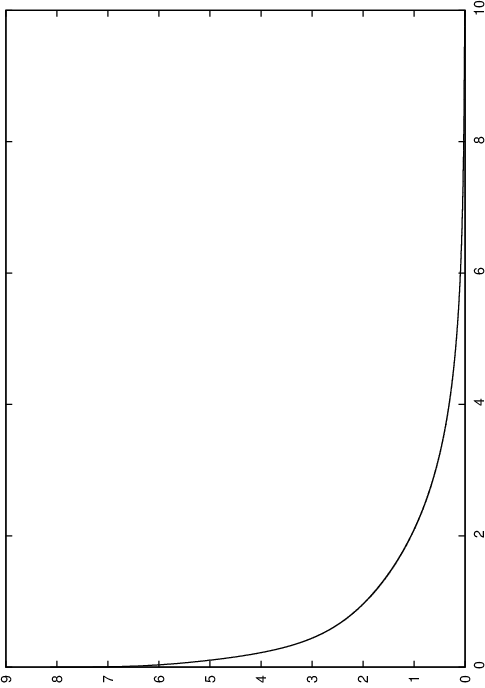}
\includegraphics[angle=-90,width=0.3\textwidth]{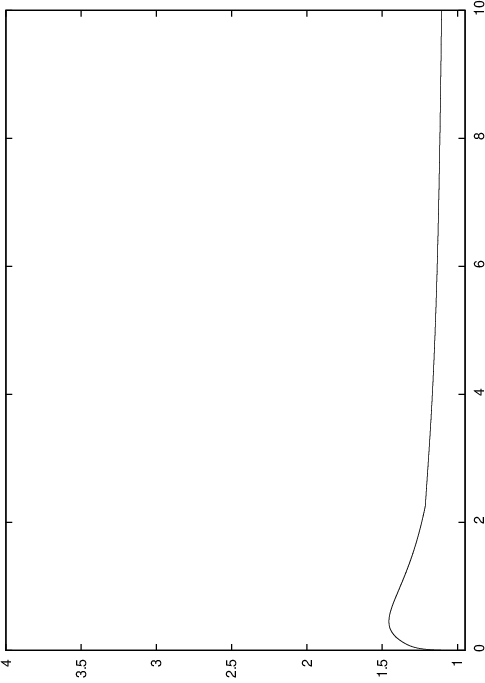}
\includegraphics[angle=-90,width=0.3\textwidth]{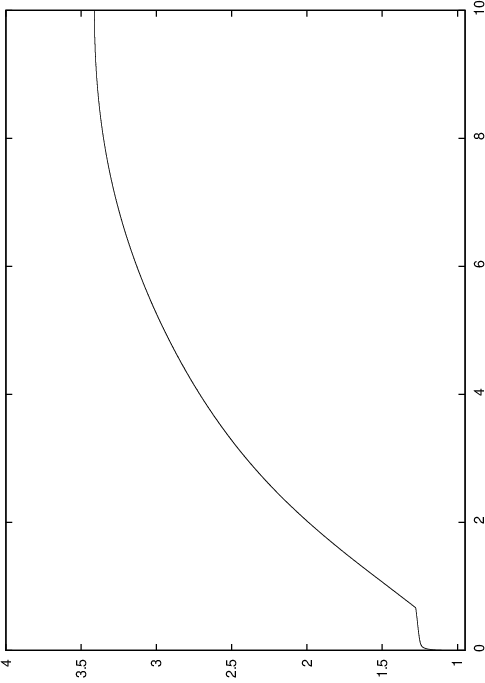}
\caption{
$(\BGNpwf_m)^h$
Geodesic elastic flow on the unit sphere.
The solutions $\vec X^m$ at times $t = 0, 1, 10$. On the right we visualize
$\vec\Phi(\vec X^m)$ at times $t=0$ (red), $t=1$ (blue) and $t=10$ (black), 
for \eqref{eq:gMercator}, with the two poles, $\pm\vec\ek_3$, 
represented by green dots.
Below a plot of the discrete energy $W^{m+1}_{g,\lambda}$,
as well as of the ratio \eqref{eq:ratio} for $(\BGNpwf_m)^h$
and $(\BGNpwfwf_m)^\star$.
} 
\label{fig:pwfsphere}
\end{figure}%
In Figure~\ref{fig:pwfspherelambda} we show the same evolution for
$\lambda=0.4$. We note that the final shape is not a steady state.
Of course, by considering a length-preserving variant,
where the parameter $\lambda$ depends on time, steady state solutions as shown
in e.g.\ \cite[Figs.\ 36, 37]{curves3d}
could also be produced by the numerical schemes presented here.}
\revised{%
However, as this goes beyond the scope of the paper, we omit such details
and the corresponding evolutions here.
\begin{figure}
\center
\includegraphics[angle=-90,width=0.25\textwidth]{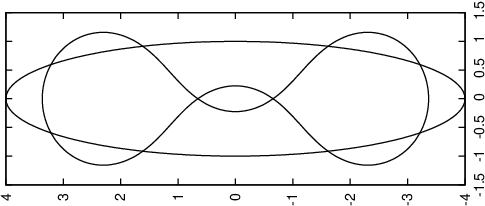}\qquad
\includegraphics[angle=-90,width=0.4\textwidth]{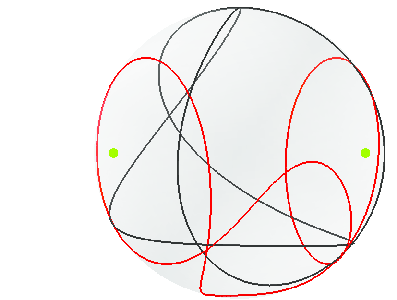} \\[2mm]
\includegraphics[angle=-90,width=0.3\textwidth]{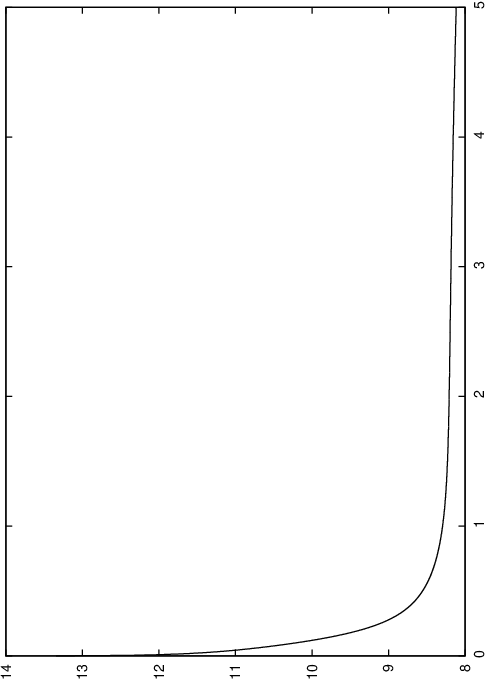}
\includegraphics[angle=-90,width=0.3\textwidth]{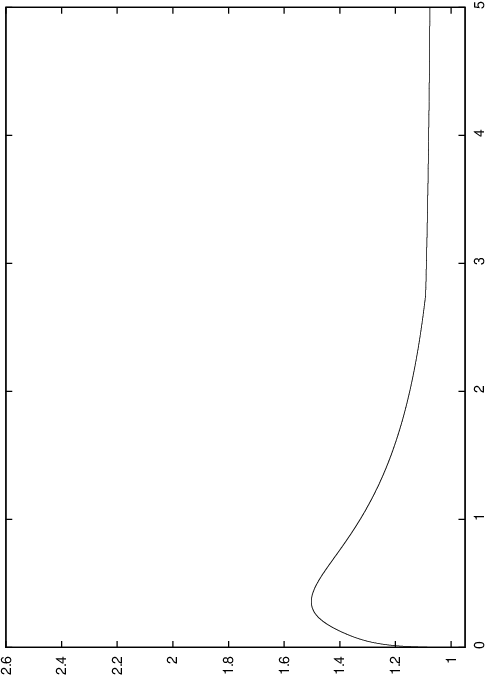}
\includegraphics[angle=-90,width=0.3\textwidth]{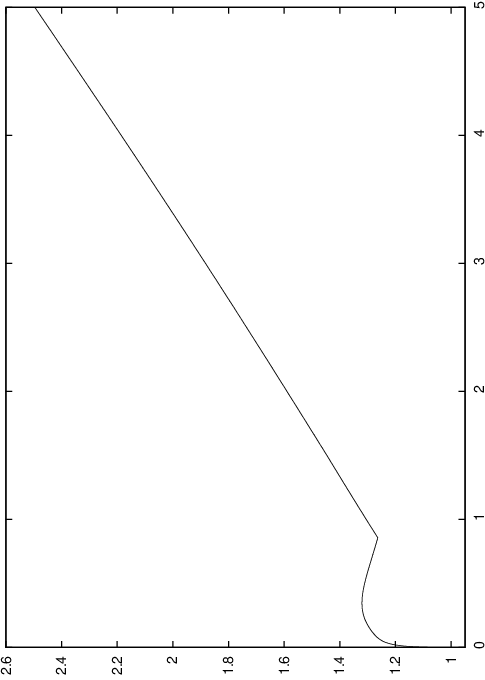}
\caption{
$(\BGNpwf_m)^h$
Generalized geodesic elastic flow, with $\lambda=0.4$, on the unit sphere.
The solutions $\vec X^m$ at times $t = 0, 5$. On the right we visualize
$\vec\Phi(\vec X^m)$ at times $t=0$ (red) and $t=5$ (black), 
for \eqref{eq:gMercator}, with the two poles, $\pm\vec\ek_3$, 
represented by green dots.
Below a plot of the discrete energy $W^{m+1}_{g,\lambda}$,
as well as of the ratio \eqref{eq:ratio} for $(\BGNpwf_m)^h$
and $(\BGNpwfwf_m)^\star$.
} 
\label{fig:pwfspherelambda}
\end{figure}%
}

\section*{\revised{Conclusions}}
\revised{%
We have derived and analysed two finite element schemes for the numerical
approximation of elastic flow in two-dimensional Riemannian manifolds. 
The Riemannian manifolds that can be considered in our framework
include the hyperbolic}
\revised{plane, the hyperbolic disk and the elliptic plane. 
More generally, any metric conformal to the two-dimensional Euclidean metric
can be considered. 
An example of}
\revised{this are two-dimensional manifolds in $\bR^d$, $d\geq3$, 
which are conformally parameterized.}

\revised{%
Our numerical simulations are based on the fully discrete schemes
$(\BGNpwf_m)^h$, $(\BGNpwfwf_m)^h$ and $(\BGNpwfwf_m)^\star$. Due to the
non-optimal convergence rates exhibited by $(\BGNpwfwf_m)^h$ in some numerical
experiments, as well as numerical breakdown in others,
we would advocate to use either $(\BGNpwf_m)^h$ or 
$(\BGNpwfwf_m)^\star$ in practice. Here the former has the advantage that the
vertices will be nearly equidistributed in practice, and that the assembly of
the linear systems is easier due the use of a mass-lumping quadrature.}

\begin{appendix}

\renewcommand{\theequation}{A.\arabic{equation}}
\setcounter{equation}{0}
\section{Consistency of weak formulations} \label{sec:A}

In this appendix we prove that solutions to $(\BGNpwf)$ and
$(\BGNpwfwf)$ indeed satisfy the strong form (\ref{eq:g_elastflowlambda}). 
Throughout this appendix we suppress the dependence of $g$ on $\vec x$.

For later use we note, on recalling (\ref{eq:tau}), (\ref{eq:varkappa})
and (\ref{eq:Gaussg}), that
\begin{subequations}
\begin{align}
\vec\nu_s & = -\varkappa\,\vec\tau\,,\label{eq:nus} \\
g^{-\frac12}\,(g^\frac12)_s & = -g^{\frac12}\,(g^{-\frac12})_s = 
\tfrac12\,(\ln\,g)_s \,,\label{eq:gm12g12s} \\
(\ln\,g)_s &= \vec\tau\,.\,\nabla\,\ln\,g\,,\label{eq:lngs} \\
(\ln\,g)_{ss} &
= \vec\tau\,.\,(\nabla\,\ln\,g)_s + \vec\tau_s\,.\,\nabla\,\ln\,g 
= \vec\tau\,.\,(D^2\,\ln\,g)\,\vec\tau +
\varkappa\,\vec\nu\,.\,\nabla\,\ln\,g\,, \label{eq:lngss} \\
(\vec\nu\,.\,\nabla\,\ln\,g)_s &= \vec\nu\,.\,(\nabla\,\ln\,g)_s +
\vec\nu_s\,.\,\nabla\,\ln\,g = \vec\nu\,.\,(D^2\,\ln\,g)\,\vec\tau
- \varkappa\,(\ln\,g)_s \,, \label{eq:nunablalngs} \\
g^{-\frac12}\,(\varkappa_g)_{ss} & 
= g^{-\frac12}\,(g^\frac12\,(\varkappa_g)_{s_g})_s 
= g^\frac12\,(\varkappa_g)_{s_gs_g} + g^{-\frac12}\,(g^\frac12)_s\,
  (\varkappa_g)_{s_g} \nonumber \\ &
= g^\frac12\,(\varkappa_g)_{s_gs_g} + \tfrac12\,(\ln g)_s\,(\varkappa_g)_{s_g} 
\,, \label{eq:appendix4} \\ 
-2\,g\,S_0(\vec x) & = \Delta\,\ln\,g
= \vec\nu\,.\,(D^2\,\ln\,g)\,\vec\nu + \vec\tau\,.\,(D^2\,\ln\,g)\,\vec\tau 
\,. \label{eq:2gS0}
\end{align}
\end{subequations}

\subsection{$(\BGNpwf)$}
We note from (\ref{eq:kappaid}), (\ref{eq:varkappag}) and (\ref{eq:nus}) that
\begin{equation} \label{eq:ydotnu}
\vec y\,.\,\vec\nu = \varkappa_g
= g^{-\frac12}(\vec{x})\left[\varkappa -\tfrac12\,
  \vec{\nu}\,.\, \nabla\, \ln g(\vec{x})\right]
\quad\text{and}\quad
\vec y_s\,.\,\vec\nu 
= (\varkappa_g)_s + \varkappa\,\vec y\,.\,\vec\tau\,, 
\end{equation}
and so it follows from (\ref{eq:dLdxflow2}), 
\revised{$\mathcal{V}_g = g^{\frac12}\,\vec{x}_t\,.\,\vec{\nu}$},
(\ref{eq:normg}), (\ref{eq:tau}) and (\ref{eq:aperp}) that
\begin{align}
& \left(g^\frac12\,\mathcal{V}_g, \vec\chi\,.\,\vec\nu\,
|\vec x_\rho|_g \right) 
= 
-\tfrac12 \left( g^\frac12\,[\varkappa_g^2 + 2\,\lambda], 
\vec\tau\,.\,\vec\chi_\rho\right)  
+ \tfrac14 \left( g^\frac12\,[\varkappa_g^2 - 2\,\lambda],
(\nabla\,\ln\,g)\,.\,\vec\chi\,|\vec x_\rho| \right)
\nonumber \\ & \qquad \qquad 
+ \tfrac12 \left( \varkappa_g\,(D^2\,\ln\,g)\,\vec\nu,
\vec\chi\,|\vec x_\rho| \right) 
-\tfrac12\left( \varkappa_g\,[\ln\,g]_s, \vec\nu\,.\,\vec\chi_\rho\right)
\nonumber \\ &\qquad \qquad  
+ \left(\left[ (\varkappa_g)_s + \varkappa\,\vec y\,.\,\vec\tau\right]
\vec\nu,\vec \chi_\rho\right)
+ \left( \left[g^\frac12\,\varkappa_g +
  \tfrac12\,\vec\nu\,.\,\nabla\,\ln\,g\right] \vec y^\perp,
\vec\chi_\rho \right)\nonumber \\
& \qquad = \tfrac12 \left( \left[g^\frac12\left[\varkappa_g^2 - 2\,\lambda\right]
+ \varkappa_g\,\vec\nu\,.\,\nabla\,\ln\,g \right]\vec\tau,\vec\chi_\rho \right)
\nonumber \\ & \qquad \qquad
+ \left( \left[ (\varkappa_g)_s + \varkappa\,\vec y\,.\,\vec\tau
 - \tfrac12\,\varkappa_g\,(\ln\,g)_s
- (\vec y\,.\,\vec\tau)\,(g^\frac12\,\varkappa_g + \tfrac12\,\vec\nu\,.\,
 \nabla\,\ln\,g) \right] \vec\nu, \vec\chi_\rho \right) \nonumber \\ & 
 \qquad \qquad
+\tfrac14 \left( g^\frac12 \left[\varkappa_g^2 - 2\,\lambda \right]  ,
 (\nabla\,\ln\,g)\,.\,\vec\chi\,|\vec x_\rho|\right) 
+ \tfrac12 \left(\varkappa_g\,(D^2\,\ln\,g)\,\vec\nu, \vec\chi\,
|\vec x_\rho| \right) \nonumber \\
&\qquad = \sum_{i=1}^4 S_i(\vec\chi)
\qquad \forall\ \vec\chi \in [H^1(I)]^2\,.
\label{eq:appendixS1}
\end{align}
Combining (\ref{eq:appendixS1}), 
\revised{\eqref{eq:varkappag}, integration by parts,}
(\ref{eq:varkappa}) and (\ref{eq:normg}) yields that
\begin{align}
S_1(\vec\chi) & = - \tfrac12 \left( g^\frac12\left[\varkappa_g^2 + 2\,\lambda
- 2\,g^{-\frac12}\,\varkappa\,\varkappa_g 
\right]\vec\tau,\vec\chi_\rho \right) \nonumber \\ & 
= \tfrac12 \left( \varkappa \left[\varkappa_g^2 + 2\,\lambda
- 2\,g^{-\frac12}\,\varkappa\,\varkappa_g 
\right]\vec\nu,\vec\chi\,|\vec x_\rho|_g \right) \nonumber \\ & \quad
+ \tfrac12 \left( g^{-\frac12}\left[ g^\frac12 \left[\varkappa_g^2 + 2\,\lambda
- 2\,g^{-\frac12}\,\varkappa\,\varkappa_g 
\right] \right]_s \vec\tau,\vec\chi\,|\vec x_\rho|_g \right) .
\label{eq:S1}
\end{align}
Combining (\ref{eq:appendixS1}) and (\ref{eq:ydotnu}),
on noting (\ref{eq:varkappag}), (\ref{eq:nus}), (\ref{eq:appendix4}) 
and (\ref{eq:normg}), yields that
\begin{align}
S_2(\vec\chi) & = 
\left( \left[ (\varkappa_g)_s - \tfrac12\,\varkappa_g\,(\ln\,g)_s\right]
\vec\nu, \vec\chi_\rho \right) \nonumber \\ & 
= - \left( g^{-\frac12}\left[
(\varkappa_g)_{ss} - \tfrac12\,((\ln\,g)_s\,\varkappa_g)_s 
\right] \vec\nu, \vec\chi\,|\vec x_\rho|_g \right) \nonumber \\ & \quad
+ \left( g^{-\frac12}\,\varkappa \left[
(\varkappa_g)_s - \tfrac12\,(\ln\,g)_s\,\varkappa_g \right] \vec\tau, 
\vec\chi\,|\vec x_\rho|_g \right) \nonumber \\ & 
= - \left( g^{\frac12}\,(\varkappa_g)_{s_gs_g} 
- \tfrac12\,g^{-\frac12}\,(\ln\,g)_{ss}\,\varkappa_g 
, \vec\chi\,.\,\vec\nu\,|\vec x_\rho|_g \right) \nonumber \\ & \quad
+ \left( \varkappa \left[ (\varkappa_g)_{s_g} 
- \tfrac12\,g^{-\frac12}\,(\ln\,g)_{s}\,\varkappa_g \right]
, \vec\chi\,.\,\vec\tau\,|\vec x_\rho|_g \right) .
\label{eq:S2}
\end{align}
Combining (\ref{eq:appendixS1}) and (\ref{eq:varkappag}),
on noting (\ref{eq:normg}) and (\ref{eq:lngs}), yields that
\begin{align}
S_3(\vec\chi) & 
= \tfrac14\left( \varkappa_g^2 - 2\,\lambda , 
(\nabla\,\ln\,g)\,.\,\vec\chi\,|\vec x_\rho|_g
\right) \nonumber \\ &
= \tfrac14\left( \varkappa_g^2 - 2\,\lambda , 
\left[2\,(\varkappa - g^\frac12\,\varkappa_g)
\,\vec\chi\,.\,\vec\nu + (\ln\,g)_s
\,\vec\chi\,.\,\vec\tau \right] |\vec x_\rho|_g \right) .
\label{eq:S3}
\end{align}
It follows from (\ref{eq:appendixS1}) and (\ref{eq:normg}) that
\begin{equation}
S_4(\vec\chi) 
= \tfrac12 \left(g^{-\frac12}\,\varkappa_g\,(D^2\,\ln\,g)\,\vec\nu , \left[
(\vec\chi\,.\,\vec\nu)\,\vec\nu + (\vec\chi\,.\,\vec\tau)\,\vec\tau \right]
|\vec x_\rho|_g \right) .
\label{eq:S4}
\end{equation}

Choosing $\vec\chi = \chi\,\vec\tau$, for $\chi \in H^1(I)$, in
(\ref{eq:appendixS1}), and noting (\ref{eq:S1}),
(\ref{eq:S2}), (\ref{eq:S3}) and (\ref{eq:S4}), we obtain for the right hand
side of (\ref{eq:appendixS1}) the value
\begin{align}
& \sum_{i=1}^4 S_i(\chi\,\vec\tau) 
= \tfrac12 \left( g^{-\frac12}\left[ g^\frac12 \left[\varkappa_g^2 + 2\,\lambda
- 2\,g^{-\frac12}\,\varkappa\,\varkappa_g 
\right] \right]_s , \chi\,|\vec x_\rho|_g \right) \nonumber \\ & \quad
+ \left( \varkappa \left[ (\varkappa_g)_{s_g} 
- \tfrac12\,g^{-\frac12}\,(\ln\,g)_{s}\,\varkappa_g \right]
, \chi\,|\vec x_\rho|_g \right) 
+ \tfrac14\left( \varkappa_g^2 - 2\,\lambda , (\ln\,g)_s\,
\chi\,|\vec x_\rho|_g \right) \nonumber \\ & \quad
+ \tfrac12 \left(g^{-\frac12}\,\varkappa_g\,\vec\tau\,.\,(D^2\,\ln\,g)\,\vec\nu 
, \chi \,|\vec x_\rho|_g \right) \nonumber \\ &
= \tfrac12
\left( g^{-\frac12}\,(g^\frac12)_s\left[
\varkappa_g^2 + 2\,\lambda - 2\,g^{-\frac12}\,\varkappa\,\varkappa_g \right]
+ \left[\varkappa_g^2 - 2\,g^{-\frac12}\,\varkappa\,\varkappa_g\right]_s
, \chi \,|\vec x_\rho|_g \right)
\nonumber \\ & \qquad \qquad
+ \left(\varkappa\left[ 
(\varkappa_g)_{s_g} - \tfrac12\,g^{-\frac12}\,(\ln\,g)_s\,\varkappa_g
\right]
+ (\tfrac14\,\varkappa_g^2 - \tfrac12\,\lambda)\,(\ln\,g)_s
, \chi \,|\vec x_\rho|_g \right)
\nonumber \\ & \qquad \qquad
+\tfrac12 \left(
g^{-\frac12}\,\varkappa_g\left[ (\vec\nu\,.\,\nabla\,\ln\,g)_s +
\varkappa\,(\ln\,g)_s \right] , \chi \,|\vec x_\rho|_g \right)
\nonumber \\ &
= \tfrac12 \left( (\ln\,g)_s \left[ \varkappa_g^2 -
g^{-\frac12}\,\varkappa\,\varkappa_g \right] 
+ \left[\varkappa_g^2 - 2\,g^{-\frac12}\,\varkappa\,\varkappa_g\right]_s
+ 2\,g^{-\frac12}\,\varkappa\,(\varkappa_g)_s, 
\chi \,|\vec x_\rho|_g \right)
\nonumber \\ & \qquad \qquad
+ \left( g^{-\frac12}\,\varkappa_g\,(\varkappa-g^{\frac12}\,\varkappa_g)_s, 
\chi \,|\vec x_\rho|_g \right) \nonumber \\ &
= \tfrac12\,\left( (\ln\,g)_s \left[ \varkappa_g^2 -
g^{-\frac12}\,\varkappa\,\varkappa_g 
+ g^{-\frac12}\,\varkappa\,\varkappa_g  - \varkappa_g^2 \right] , 
\chi \,|\vec x_\rho|_g \right)\nonumber \\ & \qquad \qquad
+\left(\varkappa_g\,(\varkappa_g)_s + g^{-\frac12}\left[
- (\varkappa\,\varkappa_g)_s + \varkappa\,(\varkappa_g)_s + 
\varkappa_g\,\varkappa_s \right] - \varkappa_g\,(\varkappa_g)_s, 
\chi \,|\vec x_\rho|_g \right) \nonumber \\ &
= 0\,,
\label{eq:appendixS7}
\end{align}    
as required, where we have recalled (\ref{eq:gm12g12s})
and (\ref{eq:varkappag}). 

Choosing $\vec\chi = \chi\,\vec\nu$, for $\chi \in H^1(I)$, in
(\ref{eq:appendixS1}), and noting (\ref{eq:S1}),
(\ref{eq:S2}), (\ref{eq:S3}) and (\ref{eq:S4}), we obtain 
\begin{align}
& \left(g^\frac12\,\mathcal{V}_g, \chi\,|\vec x_\rho|_g \right)
= \sum_{i=1}^4 S_i(\chi\,\vec\nu) 
= \tfrac12 \left( \varkappa \left[\varkappa_g^2 + 2\,\lambda
- 2\,g^{-\frac12}\,\varkappa\,\varkappa_g 
\right],\chi\,|\vec x_\rho|_g \right) \nonumber \\ & \quad
- \left( g^{\frac12}\,(\varkappa_g)_{s_gs_g} 
- \tfrac12\,g^{-\frac12}\,(\ln\,g)_{ss}\,\varkappa_g 
, \chi\,|\vec x_\rho|_g \right) \nonumber \\ & \quad
+ \tfrac12\left( \varkappa_g^2 - 2\,\lambda 
, (\varkappa - g^\frac12\,\varkappa_g)
\,\chi \,|\vec x_\rho|_g \right) 
+ \tfrac12 \left(g^{-\frac12}\,\varkappa_g\,\vec\nu\,.\,(D^2\,\ln\,g)\,\vec\nu 
, 
\chi\,|\vec x_\rho|_g \right)
\nonumber \\ & 
= \left( - g^{\frac12}\,(\varkappa_g)_{s_gs_g} 
+ g^\frac12\,\lambda\,\varkappa_g 
+ \tfrac12\,g^{-\frac12}\left[ \vec\tau\,.\,(D^2\,\ln\,g)\,\vec\tau
+ \vec\nu\,.\,(D^2\,\ln\,g)\,\vec\nu \right] \varkappa_g 
, \chi\,|\vec x_\rho|_g \right) \nonumber \\ & \qquad\qquad
+\tfrac12\left( \varkappa\,\varkappa_g^2 
- 2\,g^{-\frac12}\,\varkappa^2\,\varkappa_g +
  g^{-\frac12}\,\varkappa_g\,\varkappa\,\vec\nu\,.\,\nabla\,\ln\,g 
+ \varkappa_g^2\,(\varkappa - g^\frac12\,\varkappa_g)
, \chi\,|\vec x_\rho|_g \right) 
\nonumber \\ & 
= - \left( g^{\frac12}\left[(\varkappa_g)_{s_gs_g} 
+ \tfrac12\,\varkappa_g^3 + (S_0(\vec x) - \lambda)\,\varkappa_g\right]
, \chi\,|\vec x_\rho|_g \right) \nonumber \\ & \qquad\qquad
+ \left(g^{-\frac12}\,\varkappa\,\varkappa_g 
\left[g^\frac12\,\varkappa_g-\varkappa+\tfrac12\,\vec\nu\,.\,\nabla\,\ln\,g
\right] , \chi\,|\vec x_\rho|_g \right) \nonumber \\ & 
= - \left( g^{\frac12}\left[ 
(\varkappa_g)_{s_gs_g} + \tfrac12\,\varkappa_g^3 
+ (S_0(\vec x) - \lambda)\,\varkappa_g \right], \chi\,|\vec x_\rho|_g \right)
\qquad \forall\ \chi \in H^1(I)\,,
\label{eq:appendixS8}
\end{align}
where we have recalled (\ref{eq:lngss}), (\ref{eq:2gS0}) and
(\ref{eq:varkappag}).
Clearly, it follows from (\ref{eq:appendixS8}) that (\ref{eq:g_elastflowlambda})
holds.

\subsection{$(\BGNpwfwf)$}
It follows from (\ref{eq:BGNpwfwfa}),
\revised{$\mathcal{V}_g = g^{\frac12}\,\vec{x}_t\,.\,\vec{\nu}$},
(\ref{eq:kappagid}), (\ref{eq:tau}), (\ref{eq:normg}) and (\ref{eq:aperp}) that
\begin{align}
& \left(g^\frac12\,\mathcal{V}_g, \vec\chi\,.\,\vec\nu\,
|\vec x_\rho|_g \right) 
=
 -\tfrac12 \left( \varkappa_g^2 + 2\,\lambda
- \vec y_g\,.\,\nabla\,\ln\,g, \left[\vec\tau\,.\,\vec\chi_s + 
\tfrac12\,\vec\chi\,.\,\nabla\,\ln\,g\right] |\vec x_\rho|_g
 \right) 
\nonumber \\ & \qquad 
+ \tfrac12 \left((D^2\,\ln\,g)\,\vec y_g, \vec\chi\, |\vec x_\rho|_g \right) 
+  \left( \varkappa_g^2
+ \tfrac12\,(\vec y_g)_s\,.\,\vec\tau, (\nabla\,\ln\,g)\,.\,\vec\chi\,
|\vec x_\rho|_g \right) 
+ \left( g\,\varkappa_g\,\vec y_g^\perp,
\vec\chi_\rho\right) \nonumber \\ & \qquad 
+ \left( g^{\frac{1}{2}}(\vec y_g)_s\,.\,\vec\nu, 
\vec\chi_\rho\,.\vec \nu  \right) 
\nonumber \\ 
& \quad= -\tfrac12 \left( g^{\frac{1}{2}}\left[\varkappa_g^2 + 2\,\lambda
- \vec y_g\,.\,\nabla\,\ln\,g\right]\, \vec\tau,\vec\chi_\rho \right) 
 + \tfrac12 \left((D^2\,\ln\,g)\,\vec y_g, \vec\chi\,
|\vec x_\rho|_g \right)
\nonumber \\ & \qquad
+ \left( \revised{\tfrac14\left[3\,\varkappa_g^2 - 2\,\lambda+ \vec
y_g\,.\,\nabla\,\ln\,g  \right] + \tfrac12 \,(\vec y_g)_s \,.\,\vec\tau} ,
 (\nabla\,\ln\,g)\,.\,\vec\chi\,|\vec x_\rho|_g\right) 
\nonumber \\ & \qquad 
+ \left( g\,
\varkappa_g\,\vec y_g^\perp,
\vec\chi_\rho\right)
+ \left( g^{\frac{1}{2}}(\vec y_g)_s\,.\,\vec\nu, 
\vec\chi_\rho\,.\vec \nu  \right) 
\nonumber \\
 & \quad
 = \tfrac12 \left( g^\frac12\left[\varkappa_g^2 - 2\,\lambda
+ \vec y_g\,.\,\nabla\,\ln\,g \right]\vec\tau,\vec\chi_\rho \right)
+ \left( g^\frac12\,[(\vec y_g)_s\,.\,\vec\nu]\,\vec\nu
- g\,\varkappa_g\,(\vec y_g\,.\,\vec\tau)\,\vec\nu, 
\vec\chi_\rho \right) \nonumber \\ & \qquad
+ \left( \tfrac14\left[3\,\varkappa_g^2 - 2\,\lambda+ \vec
y_g\,.\,\nabla\,\ln\,g  
+ 2 \,(\vec y_g)_s \,.\,\vec\tau \right] ,
 (\nabla\,\ln\,g)\,.\,\vec\chi\,|\vec x_\rho|_g\right) \nonumber \\ & \qquad
+ \tfrac12 \left((D^2\,\ln\,g)\,\vec y_g, \vec\chi\,
|\vec x_\rho|_g \right) 
= \sum_{i=1}^4 T_i(\vec\chi)
\qquad \forall\ \vec\chi \in [H^1(I)]^2\,.
\label{eq:appendix1}
\end{align}
It follows from (\ref{eq:varkappag}), (\ref{eq:kappagid}) and 
(\ref{eq:lngs}) that
\begin{align}
\vec y_g\,.\,\nabla\,\ln\,g & 
= (\vec y_g\,.\,\vec\nu)\,\vec\nu\,.\,\nabla\,\ln\,g
+ (\vec y_g\,.\,\vec\tau)\,\vec\tau\,.\,\nabla\,\ln\,g \nonumber \\ &
= g^{-\frac12}\,\varkappa_g\,2\,(\varkappa - g^\frac12\,\varkappa_g)
+\vec y_g\,.\,\vec\tau\,(\ln\,g)_s 
= 2\,g^{-\frac12}\,\varkappa\,\varkappa_g - 2\,\varkappa_g^2 
+ \vec y_g\,.\,\vec\tau\,(\ln\,g)_s\,.
\label{eq:appendix2}
\end{align}
Combining (\ref{eq:appendix1}), (\ref{eq:appendix2}), (\ref{eq:S1}), 
(\ref{eq:varkappa}) and (\ref{eq:normg}) yields that
\begin{align}
& T_1(\vec\chi) = 
- \tfrac12 \left( g^\frac12\left[\varkappa_g^2 + 2\,\lambda
- 2\,g^{-\frac12}\,\varkappa\,\varkappa_g - \vec y_g\,.\,\vec\tau\,(\ln\,g)_s
\right]\vec\tau,\vec\chi_\rho \right) \nonumber \\ & 
= S_1(\vec\chi) + \tfrac12 \left( g^\frac12\,
\vec y_g\,.\,\vec\tau\,(\ln\,g)_s\,\vec\tau, \vec\chi_\rho \right) 
\nonumber \\ & 
= S_1(\vec\chi) - \tfrac12 \left( \varkappa 
\left[ \vec y_g\,.\,\vec\tau\,(\ln\,g)_s
\right]\vec\nu,\vec\chi\,|\vec x_\rho|_g \right)
- \tfrac12 \left( g^{-\frac12}\left[ g^\frac12 
\left[ \vec y_g\,.\,\vec\tau\,(\ln\,g)_s
\right] \right]_s \vec\tau,\vec\chi\,|\vec x_\rho|_g \right) .
\label{eq:T1}
\end{align}
It follows from (\ref{eq:kappagid}), (\ref{eq:nus}) and (\ref{eq:lngs}) that
\begin{align}
g^\frac12\,(\vec y_g)_s\,.\,\vec\nu & = \left[ g^\frac12\,\vec
y_g\,.\,\vec\nu\right]_s - \vec y_g\,.\left[ g^\frac12\,\vec\nu\right]_s
= (\varkappa_g)_s - (\vec y_g\,.\,\vec\nu)\,(g^\frac12)_s 
+ g^\frac12\,\varkappa\,\vec y_g\,.\,\vec\tau \nonumber \\ &
= (\varkappa_g)_s - \varkappa_g\,g^{-\frac12}\,(g^\frac12)_s 
+ g^\frac12\,\varkappa\,\vec y_g\,.\,\vec\tau
= (\varkappa_g)_s - \tfrac12\,(\ln\,g)_s\,\varkappa_g 
+ g^\frac12\,\varkappa\,\vec y_g\,.\,\vec\tau\,.
\label{eq:appendix3}
\end{align}
Combining (\ref{eq:appendix1}) and (\ref{eq:appendix3}),
on noting (\ref{eq:nus}), (\ref{eq:appendix4}) and (\ref{eq:normg}), 
yields that
\begin{align}
T_2(\vec\chi) & = - \left( g^{-\frac12}\left[
\revised{g^\frac12\,(\vec y_g)_s\,.\,\vec\nu}
- g\,\varkappa_g\,(\vec y_g\,.\,\vec\tau)\right]_s \vec\nu, 
\vec\chi\,|\vec x_\rho|_g \right) \nonumber \\ & \quad
+ \left( \varkappa \left[
\revised{(\vec y_g)_s\,.\,\vec\nu}
- g^\frac12\,\varkappa_g\,(\vec y_g\,.\,\vec\tau)\right] \vec\tau, 
\vec\chi\,|\vec x_\rho|_g \right) \nonumber \\ &
= - \left( g^{-\frac12}\left[
(\varkappa_g)_{ss} - \tfrac12\,((\ln\,g)_s\,\varkappa_g)_s +
\left[ \vec y_g\,.\,\vec\tau\,(g^\frac12\,\varkappa - g\,\varkappa_g)\right]_s 
\right] \vec\nu, \vec\chi\,|\vec x_\rho|_g \right) \nonumber \\ & \quad
+ \left( \varkappa \left[
\revised{(\vec y_g)_s\,.\,\vec\nu}
- g^\frac12\,\varkappa_g\,(\vec y_g\,.\,\vec\tau)\right] \vec\tau, 
\vec\chi\,|\vec x_\rho|_g \right) \nonumber \\ & 
= - \left( g^{\frac12}\,(\varkappa_g)_{s_gs_g} 
- \tfrac12\,g^{-\frac12}\,(\ln\,g)_{ss}\,\varkappa_g +
g^{-\frac12} \left[ 
\vec y_g\,.\,\vec\tau\,(g^\frac12\,\varkappa - g\,\varkappa_g)\right]_s 
, \vec\chi\,.\,\vec\nu\,|\vec x_\rho|_g \right) \nonumber \\ & \quad
+ \left( \varkappa \left[ (\varkappa_g)_{s_g} 
- \tfrac12\,g^{-\frac12}\,(\ln\,g)_{s}\,\varkappa_g +
\vec y_g\,.\,\vec\tau\,(\varkappa - g^\frac12\,\varkappa_g)\right]
, \vec\chi\,.\,\vec\tau\,|\vec x_\rho|_g \right) \nonumber \\ & 
= S_2(\vec\chi) - \left( g^{-\frac12} \left[ 
\vec y_g\,.\,\vec\tau\,(g^\frac12\,\varkappa - g\,\varkappa_g)\right]_s 
, \vec\chi\,.\,\vec\nu\,|\vec x_\rho|_g \right) \nonumber \\ & \quad
+ \left( \varkappa \,
\vec y_g\,.\,\vec\tau\,(\varkappa - g^\frac12\,\varkappa_g)
, \vec\chi\,.\,\vec\tau\,|\vec x_\rho|_g \right).
\label{eq:T2}
\end{align}
It follows from (\ref{eq:varkappa}) and (\ref{eq:kappagid}) that
\begin{equation} \label{eq:appendix5}
(\vec y_g)_s\,.\,\vec\tau 
= (\vec y_g\,.\,\vec\tau)_s - \vec y_g\,.\,\vec\tau_s 
= (\vec y_g\,.\,\vec\tau)_s - \varkappa\,\vec y_g\,.\,\vec\nu 
= (\vec y_g\,.\,\vec\tau)_s - g^{-\frac12}\,\varkappa\,\varkappa_g\,.
\end{equation}
Combining (\ref{eq:appendix1}), (\ref{eq:appendix2}), 
(\ref{eq:appendix5}) and (\ref{eq:varkappag}) yields that
\begin{align}
& T_3(\vec\chi)
= \tfrac14\left( 3\,\varkappa_g^2 - 2\,\lambda + \vec y_g\,.\,\nabla\,\ln\,g
+ 2\, (\vec y_g)_s\,.\,\vec\tau, (\nabla\,\ln\,g)\,.\,\vec\chi\,|\vec x_\rho|_g
\right) \nonumber \\ &
= \tfrac14\left( \varkappa_g^2 - 2\,\lambda + \vec y_g\,.\,\vec\tau\, (\ln\,g)_s
+ 2\, (\vec y_g\,.\,\vec\tau)_s, (\nabla\,\ln\,g)\,.\,\vec\chi\,|\vec x_\rho|_g
\right) \nonumber \\ &
= \tfrac14\left( \varkappa_g^2 - 2\,\lambda + \vec y_g\,.\,\vec\tau\, (\ln\,g)_s
+ 2\, (\vec y_g\,.\,\vec\tau)_s, \left[(\vec\nu\,.\,\nabla\,\ln\,g)
\,\vec\chi\,.\,\vec\nu + (\vec\tau\,.\,\nabla\,\ln\,g)
\,\vec\chi\,.\,\vec\tau \right] |\vec x_\rho|_g \right) \nonumber \\ &
= \tfrac14\left( \varkappa_g^2 - 2\,\lambda + \vec y_g\,.\,\vec\tau\, (\ln\,g)_s
+ 2\, (\vec y_g\,.\,\vec\tau)_s, \left[2\,(\varkappa - g^\frac12\,\varkappa_g)
\,\vec\chi\,.\,\vec\nu + (\ln\,g)_s
\,\vec\chi\,.\,\vec\tau \right] |\vec x_\rho|_g \right) \nonumber \\ &
= S_3(\vec\chi) + 
\tfrac14\left( \vec y_g\,.\,\vec\tau\, (\ln\,g)_s
+ 2\, (\vec y_g\,.\,\vec\tau)_s, \left[2\,(\varkappa - g^\frac12\,\varkappa_g)
\,\vec\chi\,.\,\vec\nu + (\ln\,g)_s
\,\vec\chi\,.\,\vec\tau \right] |\vec x_\rho|_g \right) .
\label{eq:T3}
\end{align}
It follows from (\ref{eq:kappagid}) that
\begin{align}
(D^2\,\ln\,g)\,\vec y_g & 
= \vec y_g\,.\,\vec\nu\, (D^2\,\ln\,g)\,\vec\nu 
+ \vec y_g\,.\,\vec\tau\, (D^2\,\ln\,g)\,\vec\tau \nonumber \\ & 
= g^{-\frac12}\,\varkappa_g\,(D^2\,\ln\,g)\,\vec\nu 
+ \vec y_g\,.\,\vec\tau\, (D^2\,\ln\,g)\,\vec\tau\,.
\label{eq:appendix6}
\end{align}
Combining (\ref{eq:appendix1}) and (\ref{eq:appendix6}) yields that
\begin{align}
T_4(\vec\chi) & 
= \tfrac12 \left(g^{-\frac12}\,\varkappa_g\,(D^2\,\ln\,g)\,\vec\nu 
+ \vec y_g\,.\,\vec\tau\, (D^2\,\ln\,g)\,\vec\tau, \vec\chi\,
|\vec x_\rho|_g \right) \nonumber \\ &
= \tfrac12 \left(g^{-\frac12}\,\varkappa_g\,(D^2\,\ln\,g)\,\vec\nu 
+ \vec y_g\,.\,\vec\tau\, (D^2\,\ln\,g)\,\vec\tau, 
\left[
(\vec\chi\,.\,\vec\nu)\,\vec\nu + (\vec\chi\,.\,\vec\tau)\,\vec\tau \right]
|\vec x_\rho|_g \right) \nonumber \\ &
= S_4(\vec\chi) + 
\tfrac12 \left( \vec y_g\,.\,\vec\tau\, (D^2\,\ln\,g)\,\vec\tau, 
\left[
(\vec\chi\,.\,\vec\nu)\,\vec\nu + (\vec\chi\,.\,\vec\tau)\,\vec\tau \right]
|\vec x_\rho|_g \right) .
\label{eq:T4}
\end{align}

Choosing $\vec\chi = \chi\,\vec\tau$, for $\chi \in H^1(I)$, in
(\ref{eq:appendix1}), and noting (\ref{eq:T1}),
(\ref{eq:T2}), (\ref{eq:T3}), (\ref{eq:T4}) and (\ref{eq:appendixS7}), 
we obtain for the right hand side of (\ref{eq:appendix1}) the value
\begin{align*}
& \sum_{i=1}^4 T_i(\chi\,\vec\tau) = \sum_{i=1}^4 S_i(\chi\,\vec\tau) 
+ \tfrac12 \left( -g^{-\frac12}\,(g^\frac12)_s\,(\ln\,g)_s 
- (\ln\,g)_{ss} , (\vec y_g\,.\,\vec\tau)\,\chi \,|\vec x_\rho|_g \right)
\nonumber \\ & \qquad \qquad
+ \tfrac12 \left( 2\,\varkappa\,(\varkappa - g^\frac12\,\varkappa_g)
+ \tfrac12\left[(\ln\,g)_s\right]^2
+ \vec\tau\,.\,(D^2\,\ln\,g)\,\vec\tau
, (\vec y_g\,.\,\vec\tau)\,\chi \,|\vec x_\rho|_g \right)
\nonumber \\ &
= \left( \varkappa\left[(\varkappa - g^\frac12\,\varkappa_g)
- \tfrac12\,\vec\nu\,.\,\nabla\,\ln\,g\right], 
(\vec y_g\,.\,\vec\tau)\,\chi \,|\vec x_\rho|_g \right)
= 0\,, 
\end{align*}    
as required, where we have recalled (\ref{eq:gm12g12s}), (\ref{eq:lngss}) 
and (\ref{eq:varkappag}). 

Choosing $\vec\chi = \chi\,\vec\nu$, for $\chi \in H^1(I)$, in
(\ref{eq:appendix1}), and noting (\ref{eq:T1}),
(\ref{eq:T2}), (\ref{eq:T3}), (\ref{eq:T4}) and (\ref{eq:appendixS8}), 
we obtain 
\begin{align}
& \left(g^\frac12\,\mathcal{V}_g, \chi\,|\vec x_\rho|_g \right) =
\sum_{i=1}^4 T_i(\chi\,\vec\nu) \nonumber \\ & 
= \sum_{i=1}^4 S_i(\chi\,\vec\nu) +
 \tfrac12\left( 
- \varkappa\,(\ln\,g)_s - 2\,g^{-\frac12}\left[g^\frac12\,(\varkappa -
  g^\frac12\,\varkappa_g) \right]_s  , (\vec y_g\,.\,\vec\tau)\,\chi\,|\vec x_\rho|_g \right)
\nonumber \\ & \qquad\qquad
+ \tfrac12\left((\ln\,g)_s\,(\varkappa - g^\frac12\,\varkappa_g) +
(\vec\nu\,.\,\nabla\,\ln\,g)_s +
\varkappa\,(\ln\,g)_s , (\vec y_g\,.\,\vec\tau)\,\chi\,|\vec x_\rho|_g \right)
\nonumber \\ & 
= - \left( g^{\frac12}\left[ 
(\varkappa_g)_{s_gs_g} + \tfrac12\,\varkappa_g^3 
+ (S_0(\vec x) - \lambda)\,\varkappa_g \right], \chi\,|\vec x_\rho|_g \right)
\qquad \forall\ \chi \in H^1(I)\,,
\label{eq:appendix8}
\end{align}
where we have recalled (\ref{eq:gm12g12s}) and
(\ref{eq:varkappag}).
Clearly, it follows from (\ref{eq:appendix8}) that (\ref{eq:g_elastflowlambda})
holds.
\end{appendix}

\noindent
{\bf Acknowledgements.}
The authors gratefully acknowledge the support 
of the Regensburger Universit\"atsstiftung Hans Vielberth.

\end{document}